\pdfoutput=1 

\RequirePackage{fix-cm}
\documentclass[smallextended]{svjour3}
\smartqed

\usepackage[margin=2.cm]{geometry}
\usepackage{amsmath,amssymb,epsfig,float,color,xcolor,enumitem,comment}
\definecolor{lightblue}{rgb}{0.22,0.45,0.70}
\definecolor{lightgreen}{rgb}{0.09, 0.45, 0.27}
\usepackage[colorlinks=true,breaklinks=true,linkcolor=lightblue,citecolor=lightgreen]{hyperref}

\renewenvironment{proof}{\noindent{\it Proof.}}{\hfill$\square$}

\numberwithin{equation}{section}
\numberwithin{theorem}{section}
\numberwithin{lemma}{section}
\numberwithin{corollary}{section}
\numberwithin{proposition}{section}
\numberwithin{remark}{section}
\numberwithin{figure}{section}
\numberwithin{table}{section}

\usepackage{diagbox}

\usepackage{multirow,multicol,array}

\def\div{\mathop{\mathrm{div}}\nolimits}

\def\CT{\mathcal{T}}

\def\rt{\mathrm{t}}

\newcommand\mG{\mathcal{G}}

\newcommand\mK{\mathcal{K}}

\newcommand\err{\texttt{err}}
\newcommand\eff{\texttt{eff}}

\newcommand\bE{\mathbf{E}}

\newcommand\bff{\boldsymbol{f}}
\newcommand\bL{\mathbf{L}}

\newcommand\bV{\mathbf{V}}
\newcommand\bU{\mathbf{U}}
\newcommand\bM{\mathbf{M}}
\newcommand\bW{\mathbf{W}}
\newcommand\bY{\mathbf{Y}}
\newcommand\bPi{\mathbf{\Pi}}
\newcommand\bS{\mathbf{S}}
\newcommand\bT{\mathbf{T}}

\newcommand\beps{\boldsymbol{\varepsilon}}
\newcommand\bnabla{\boldsymbol{\nabla}}

\newcommand\bg{\boldsymbol{g}}

\newcommand\bsig{\boldsymbol{\sigma}}

\newcommand\btau{\boldsymbol{\tau}}

\newcommand\bu{\boldsymbol{u}}
\newcommand\bv{\boldsymbol{v}}
\newcommand\bx{\boldsymbol{x}}

\newcommand\bz{\boldsymbol{z}}

\newcommand\bcQ{\boldsymbol{\mathcal{Q}}}

\newcommand\bcW{\boldsymbol{\mathcal{W}}}
\newcommand\bcX{\boldsymbol{\mathcal{X}}}

\newcommand\N{\mathbb{N}}
\newcommand\R{\mathbb{R}}
\newcommand{\wbPi}{\widehat{\bPi}}
\newcommand{\norm}[1]{\ensuremath{\left\|#1\right\|}}

\newcommand\dz{\,\mathrm{d}z}

\renewcommand\O{\Omega}

\renewcommand\H{\mathrm{H}}
\renewcommand\L{\mathrm{L}}
\newcommand\Q{\mathrm{Q}}

\newcommand\bdiv{\mathop{\mathbf{div}}\nolimits}
\newcommand\vdiv{\mathop{\mathrm{div}}\nolimits}

\newcommand\tr{\mathop{\mathrm{tr}}\nolimits}

\renewcommand\sp{\mathop{\mathrm{sp}}\nolimits}

\newcommand\bH{\mathbf{H}}

\newcommand\bn{\boldsymbol{n}}

\newcommand\bw{\boldsymbol{w}}

\newcommand\curl{\mathop{\mathbf{curl}}\nolimits}
\newcommand\rot{\mathop{\mathrm{rot}}\nolimits}

\def\bI{\mathbf{I}}
\def\CE{{\mathcal E}}
\def\CT{{\mathcal T}}

\newcommand{\cblue}[1]{\textcolor{blue}{#1}}

\newcommand{\vertiii}[1]{{\left\vert\kern-0.25ex\left\vert\kern-0.25ex\left\vert #1 
		\right\vert\kern-0.25ex\right\vert\kern-0.25ex\right\vert}}
\allowdisplaybreaks
\begin{document}
\titlerunning{FEM for a Herrmann-FSI spectral problem}
\authorrunning{Khan, Lepe, Mora, Ruiz-Baier \& Vellojin}
\title{Finite element analysis  for a Herrmann pressure formulation of the elastoacoustic problem with variable coefficients}
\author{A. Khan \and F. Lepe \and D. Mora \and R. Ruiz-Baier \and J. Vellojin}
\institute{Felipe Lepe \and David Mora \at
		GIMNAP, Departamento de Matem\'atica,
		Facultad de Ciencias,  
		Universidad del B\'io-B\'io, Casilla 5-C, Concepci\'on, Chile.\\
		\email{\{flepe,dmora\}@ubiobio.cl} \and
		Arbaz Khan \at
		Department of Mathematics, Indian Institute of Technology Roorkee, Roorkee 247667, India.\\
		\email{arbaz@ma.iitr.ac.in}
		\and
		David Mora \at
		CI$^2$MA, Universidad de Concepci\'on, Casilla 160-C, Concepci\'on, Chile.
		\and
		Ricardo Ruiz-Baier \at 
		School of Mathematics, Monash University, 9 Rainforest Walk, Melbourne 3800 VIC, Australia; and Universidad Adventista de Chile, Casilla 7-D, Chill\'an, Chile.\\
                \email{ricardo.ruizbaier@monash.edu}
         \and
        Jesus Vellojin
         \at 
        Departamento de Ciencias, Universidad Técnica Federico Santa María, Av. Federico Sta. María 6090, Viña del Mar, Chile.\\
        \email{jesus.vellojinm@usm.cl}
	}
	
	\date{Received: date / Revised version: date \hfill Updated: \today}

	\maketitle
	
	\begin{abstract}
			In two and three dimensions, this study is focused on the numerical analysis of an eigenproblem associated with a fluid-structure model for sloshing and elasto-acoustic vibration. We use a displacement-Herrmann pressure formulation for the solid, while  for the fluid, a pure displacement formulation is considered. Under this approach we propose a non conforming locking-free method based on classic finite elements to approximate the natural frequencies (of the eigenmodes) of the coupled system. Employing the theory for non-compact operators we prove convergence and error estimates. Also we propose an a posteriori error estimator for this coupled problem which is shown to be efficient and reliable. All the presented theory is contrasted with a set of numerical tests in 2D and 3D. 
		\end{abstract}


\section{Introduction}
The dynamics of fluid-structure interaction systems are of primary interest in a number of scientific and industrial applications. For example, vibrations occurring in pipes, parts of aerospace vessels, tanks, and many others. These applications are related to the design and development of different components, structures, and devices that are needed in different contexts. Several specific frameworks of interest are described in, e.g., \cite{MR1180076}. Although this problem has been studied for many years, new applications, formulations, methods, and challenges are still emerging in the literature, proving that research on this problem is in ongoing progress.

In the present paper, our contribution is related to the development of numerical methods for an elasto-acoustic problem where we model the motion of a fluid in a solid container. Regarding the mathematical formulation of this classical problem (see, e.g., \cite{boujot1987mathematical}), it is possible to use a displacement formulation for both subdomains and employ, at the discrete level, simple Lagrangian finite elements (FE) for the solid and Raviart--Thomas elements for the fluid. This type of formulation has been used for elasto-acoustic vibration models as well as hydroelastic and sloshing models. Some important references, such as \cite{bermudez1999finite,bermudez1997finite,MR1342293}, are focused on the study of formulations where the main unknowns are the displacements of the solid and the fluid. Hence, this choice of formulation leads to a numerical analysis that involves classical finite elements, such as piecewise linear functions to approximate the displacement of the solid and Raviart--Thomas elements for the displacement of the fluid. This choice of finite elements is shown to be spurious free and capable of approximating the spectrum of the elasto-acoustic problem accurately. However, the method depending on these families of finite elements leads to a non-conforming approximation scheme, which has its inherent difficulties. To avoid this non-conformity, \cite{MR3283363} proposes an alternative mixed formulation based on the pressure of the fluid and the stress of the solid, which, when numerically analyzed by approximating the stress with Brezzi--Douglas--Marini elements and the pressure with piecewise linear functions, is shown to be locking free, spurious free, and conforming. This formulation is consistent with the classical results in the literature. An alternative method, such as the one introduced in \cite{brenner2018nonconforming}, where only piecewise linear polynomials are considered for both media, also turns out to be theoretically and computationally accurate. Time-dependent elasto-acoustic problems can also be studied, such as in \cite{MR3606460}, where a semi-implicit time discretization with distributed Lagrange multipliers is analyzed, or \cite{MR4092281}, where a spatial discretization based on the finite element method is first considered and error estimates are also provided, followed by a fully discrete approximation based on a family of implicit finite difference schemes in time. Let us remark that an important physical property is the absence of viscosity in the model, particularly in the fluid. This assumption leads to a linear eigenvalue problem, which is not the case when internal dissipation is considered, where naturally the eigenvalue problem is quadratic, as in \cite{MR1770352}. Hence, to begin with a simple formulation, the presence of viscosity is not considered in this paper.

All the aforementioned references (and the references therein) are focused on a priori error estimates. However, the a posteriori analysis can also be performed. Indeed, we can mention \cite{MR4867703,MR2740815,MR4770432,MR3718007} as particularly interesting and well-developed contributions on this subject, which certainly provide inspiration for our work. However, we are not considering a priori and a posteriori error analysis separately; our plan is to unify the analysis under a new approach.

In the present paper, our contribution is focused on the mathematical and numerical analysis of an extension of the elasto-acoustic model studied in, for instance, \cite{MR1342293}, where sloshing effects are also considered, as in \cite{MR3283363}. More precisely, the classical elasto-acoustic model of \cite{MR1342293}, which considers the typical formulation for the elastic domain, is here extended with the incorporation of the so-called Herrmann pressure \cite{herrmann1965elasticity}. This new variable creates a connection between the classical Stokes problem and the elasticity problem. More precisely, for eigenvalue problems, it is possible to prove a relation between the spectra of the Stokes and elasticity eigenproblems via the Herrmann pressure. This fact has been proved and numerically analyzed in \cite{khan2023finite} with a finite element method involving inf-sup stable families of finite elements for Stokes. This approach has the advantage of providing a framework in which one can analyze both a fully incompressible solid (Stokes' limit) and the typical elastic structure by means of a locking-free numerical method. More precisely, with this formulation we are capable of considering a Poisson ratio equal to $1/2$, which classical formulations of the elasto-acoustic problem are not able to handle. Indeed, as in \cite{khan2023finite} for the elasticity eigenproblem, for the elasto-acoustic eigenvalue problem the locking-free property remains valid, as we present in our numerical findings. On the other hand, the introduction of variable coefficients, reflecting the physical attributes of the structure, adds another layer of complexity to our analysis. Specifically, these variable coefficients demand the use of a weighted norm with spatially dependent parameters in order to ensure robustness of the different estimates that we derive. For the fluid part, we consider a similar model to that of \cite{brenner2018nonconforming}. However, there are two main differences with respect to the aforementioned study. First, the numerical scheme in the present paper considers an inf-sup stable family for the solid domain, such as Taylor--Hood or MINI elements, while the fluid domain is discretized with Brezzi--Douglas--Marini elements. These approximations provide conformity along the interface at the expense of additional degrees of freedom. Second, we propose an a posteriori error analysis that relies on the weighted norm and a Helmholtz decomposition to provide a reliable and efficient estimator. 

Despite these relevant features of the Herrmann formulation and the framework of our analysis, the approximation of the solutions must be studied not only on the domains in which the fluid or the solid are located, but also on the interface. This drawback was rigorously studied in \cite{MR1342293}, where a corrected interpolant was constructed in order to obtain a proper approximation on the contact interface while accounting for the non-conformity of the method in that paper. For the mixed formulation in \cite{MR3283363}, this corrected interpolant operator was also needed for the same reason, but its analysis is based on the conformity of the method on the interface and hence, its construction was analyzed with a suitable extension operator. For the analysis of our model and its numerical method, we also need this corrected operator, and we construct it inspired by the aforementioned references.

Continuing with the contributions of the paper, we introduce a novel residual-based a posteriori error estimator tailored to the proposed model. The formulation of the estimator for the solid domain draws inspiration from the work of \cite{khan2023finite}. Leveraging the Helmholtz decomposition, we devise an innovative estimator for the fluid domain. Furthermore, we establish the reliability and efficiency of the proposed estimator, which are proved under the use of weighted norms and using the standard techniques for a posteriori error analysis such as \cite{MR1885308,MR3059294}. 

The paper is structured in the following manner: In Section \ref{sec:model} we describe the governing equations and derive a weak formulation of the eigenvalue problem, noting that the solution operator associated with the corresponding source problem is non-compact. Next, Section~\ref{sec:fem} describes the numerical discretization, introducing the main assumptions for the discrete spaces used herein (Taylor--Hood, MINI element, and Brezzi--Douglas--Marini), and stating approximation properties of the needed modification of classical interpolants. We discuss both the source and the discrete eigenvalue problem. The spectral approximation is postponed to Section~\ref{sec:spec_app} according to \cite{MR483400}, where we show the convergence (with optimal order) of the nonconforming method in a suitably defined mesh-dependent norm. In Section~\ref{sec:apost} we propose a residual-based a posteriori error estimator and show its robustness. We close in Section~\ref{sec:numerical-section} with a set of numerical experiments that report on the accuracy of the proposed schemes and that also illustrate the properties of the a posteriori error indicators. 

\section{Model problem}\label{sec:model}
We briefly describe the model problem according to \cite{MR3283363}, which represents the small-amplitude motion of a fluid in a container with a free surface for the fluid. Let $\O_f$ and $\O_s$ two polygonal/polyhedral bounded domains of $\mathbb{R}^d$, where $d\in\{2,3\}$, with Lipschitz boundary. We assume that $\O_f$ is the domain where the fluid is contained whereas $\O_s$ is the domain occupied by the structure. We split in two parts the boundary $\partial\O_f$: the first part corresponds to the interface contact with the structure, which we denote by $\Sigma$, and the second part is an open boundary that we dente by $\Gamma_0$. On the other hand, the boundary $\partial\O_s$ is such that $\partial\O_s:=\Sigma\cup\Gamma_D\cup\Gamma_N$, where $\Gamma_D$ is the part of the solid that we consider as clamped. We assume that $\Sigma$ is oriented by the normal vector $\bn$ outward to the boundary	 of $\O_f$. The unit normal outward to $\partial\O_s$ is also denoted by $\bn$ (see Figure~\ref{fig:reference-2D-domain}).

\begin{figure}[!hbt]
	\centering
	\includegraphics[scale=1.65]{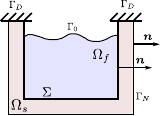}
	\caption{Sketch of a fluid-structure interaction domain with sub-boundaries. Here,  $\Omega_S$ and $\Omega_F$ denote the solid and fluid subdomains, respectively.}
	\label{fig:reference-2D-domain}
\end{figure}

On the solid subdomain, let us denote by $\bsig$ the Cauchy stress tensor, defined as
$$
\bsig:=2\mu\beps(\bu) + \lambda\tr(\beps(\bu))\mathbb{I},
$$
where $\bu$ denotes the displacement of the solid, $\boldsymbol{\varepsilon}(\bu):=\frac{1}{2}[\nabla\bu+(\nabla\bu)^{\texttt{t}}]$ is the infinitesimal strain tensor, and the Lam\'e parameters are given by
\begin{equation*}
	\lambda:=\frac{E(\bx)\nu}{(1+\nu)(1-2\nu)}\quad\text{and}\quad \mu:=\frac{E(\bx)}{2(1+\nu)},
\end{equation*}
with $E$ and $\nu$ representing the Young modulus and Poisson ratio, respectively, both assumed heterogeneous but uniformly bounded away from zero.
Let us define the solid pressure variable $p:=-\lambda\,\vdiv\bu$. Under the assumption of small oscillations  and inspired in \cite{MR3283363},  if $\omega>0$ represents the natural frequencies of the eigenmodes, the system of interest consists in finding the values of $\omega$ for which there is a tuple of solid and fluid displacements, together with Herrmann and fluid pressures $((\bu,\bw),(p,\textrm{p}_F))$ that satisfy the following system
\begin{align}
	-\bdiv(2\mu\beps(\bu)) + \nabla p&=\omega^2\rho_s\bu &\text{in }\O_s,\nonumber\\
	\vdiv \bu  + \frac{1}{\lambda} p &=0 &\text{in }\Omega_s,\nonumber\\
	\nabla\textrm{p}_F&=\omega^2\rho_f\bw&\text{in }\Omega_f,\nonumber\\
	\textrm{p}_F+\rho_fc^2\div\bw&=0&\text{in }\O_f,\nonumber\\
	\left(2\mu\beps(\bu) - p\mathbb{I} \right)\bn-(\rho_f c^2\vdiv\bw)\bn&=\boldsymbol{0} &\text{on}\,\,\Sigma, \label{eq:coupled-problem2}\\
	\bw\cdot\bn -\bu\cdot\bn&=0 &\text{on}\,\,\Sigma,\nonumber\\
	\rho_fg(\bw\cdot\bn) + \rho_fc^2 \vdiv\bw&=0 &\text{on}\,\,\Gamma_0,\nonumber\\
	\bu &= \boldsymbol{0} &\text{on }\Gamma_D,\\
	(2\mu\beps(\bu) - p\mathbb{I})\bn &= \boldsymbol{0} &\text{on } \Gamma_N,\nonumber
\end{align}
where $g$ is the gravity magnitude, $c$ is the sound speed, $\rho_s$ and $\rho_f$ denote the density of the solid and the fluid, respectively, and $\bw$ is the fluid displacement.  We observe that the pressure $\textrm{p}_F$ of the fluid can be eliminated replacing the fourth equation on \eqref{eq:coupled-problem2} on the third, leading to an expression for the fluid depending only on the displacement $\bw$.

Let us consider the following   spaces
$$
\bV=\left\{\bv\in\H^1(\O_s)^d \,:\, \bv=\boldsymbol{0}\;\text{ on }\Gamma_D\right\}, \qquad \Q=\L^2(\O_s),\quad  \bW:= \left\{\btau\in\H(\vdiv,\O_f) \;:\; \btau\cdot\bn\in \L^2(\Gamma_0) \right\},
$$
and proceed to test system \eqref{eq:coupled-problem2} by suitable test functions over $\bV$, $\Q$ and $\bW$, and imposing the boundary and interface conditions, we obtain with the following  variational problem: find $\omega\in\mathbb{R}^+$ and  $(\boldsymbol{0},\boldsymbol{0},0)\neq(\bu,\bw,p)\in\bV\times\bW\times \Q$ such that 
\begin{align}
\int_{\O_s}2\mu(\bx)\beps(\bu):\beps(\bv) - \int_{\O_s} p\vdiv \bv  + \int_{\Sigma} (\rho_f c^2\vdiv\bw)(\bv\cdot\bn) &= \omega^ 2\int_{\O_s}\rho_s\bu\cdot\bv, &\forall \bv\in \bV,\nonumber\\
-\int_{\O_s}\vdiv\bu q - \int_{\O_s}\frac{1}{\lambda(\bx)}p q &=0, &\forall q \in Q,\label{eq:coupled-problem-variational2}\\
\int_{\O_f}c^2\rho_f\vdiv\bw \vdiv \btau - \int_{\Sigma}(\rho_f c^2\vdiv\bw)(\btau\cdot\bn) + \int_{\Gamma_0}(g\rho_f\bw\cdot\bn)(\btau\cdot\bn) &=\omega^2\int_{\O_f}\rho_f\bw\cdot\btau, &\forall \btau\in \bW.\nonumber
\end{align}

Note that the formulation is not well-defined for a Darcy flux merely in $\H(\vdiv,\O_f)$, since in that case we only have $\bw\cdot\bn,\btau\cdot\bn \in \H^{-1/2}(\partial \Omega_f)$ and therefore the last term  on the  left-hand side of the third equation in \eqref{eq:coupled-problem-variational2} might not be bounded. This explains the need for the space $\bW$ defined above, which in turn makes sense as long as $\Gamma_0$ is sufficiently regular (see also, e.g., \cite{gjerde2021mixed}).			
{As commonly done in other formulations for fluid-structure interaction problems, we observe from the sixth equation in system \eqref{eq:coupled-problem2} that, in order to enforce the continuity of the normal displacement across the interface, we can define a functional space of kinematically admissible displacements as follows:
$$
\bY:=\left\{ (\bv,\btau)\in\bV\times\bW \;:\; \bv\cdot\bn-\btau\cdot\bn=0\text{ on }\Sigma\right\},
$$
where the transmission condition is understood (at least) in the $\H^{-1/2}(\Sigma)$ sense.
This allows us to define a strongly coupled system. 
Hence, using $(\bv,\btau)\in \bY$ in \eqref{eq:coupled-problem-variational2}, we arrive at the problem: find $\omega\in\mathbb{R}^+$ and $((\boldsymbol{0},\boldsymbol{0}),0)\neq((\bu,\bw),p)\in \bY\times \Q$ such that 
\begin{align}
	\int_{\O_s}2&\mu(\bx)\beps(\bu):\beps(\bv) - \int_{\O_s} p\vdiv \bv + \int_{\O_f}c^2\rho_f\vdiv\bw \vdiv \btau + {\int_{\Gamma_0} g\rho_f(\bw\cdot\bn)\,(\btau\cdot\bn)} 
	\nonumber     \\
	&= \omega^ 2\left(\int_{\O_s}\rho_s\bu\cdot\bv + \int_{\O_f}\rho_F\bw\cdot\btau\right) &\forall(\bv,\bw)\in\bY,\label{eq:coupled-problem-variational3}\\
	&-\int_{\O_s}q\vdiv\bu - \int_{\O_s}\frac{1}{\lambda(\bx)}p q =0  &\forall q\in \Q,\nonumber
\end{align}

Let us define  $\kappa:=\omega^2$ and $\bH:=\bY\times Q$ in order to rewrite \eqref{eq:coupled-problem-variational3}   as follows:
find $\kappa\in\mathbb{R}^+$ and  $((\boldsymbol{0},\boldsymbol{0}),0)\neq((\bu,\bw),p)\in \bH$ such that
\begin{equation}\label{def:limit_system_eigen_complete}
	A(((\bu,\bw),p),((\bv,\btau),q))   =(\kappa+1) B(((\bu,\bw),p),((\bv,\btau),q)),\qquad \forall((\bv,\btau),q)\in\bH,
\end{equation}
where the continuous bilinear forms
$A:\bH\times\bH\to\R$ and
$B:\bH\times\bH\to\R$
are defined by
\begin{align*}
	A(((\bu,\bw),p),((\bv,\btau),q)) &:=\int_{\O_s}2\mu(\bx)\beps(\bu):\beps(\bv) - \int_{\O_s} p\vdiv \bv + \int_{\O_f}c^2\rho_f\vdiv\bw \vdiv \btau \\
	& +{\int_{\Gamma_0} g\rho_f(\bw\cdot\bn)\,(\btau\cdot\bn)}
	-\int_{\O_s}q\,\vdiv\bu - \int_{\O_s}\frac{1}{\lambda(\bx)}p q  + B(((\bu,\bw),p),((\bv,\btau),q)),
	\\
	B(((\bu,\bw),p),((\bv,\btau),q))&:=\int_{\O_s}\rho_s\bu\cdot\bv + \int_{\O_f}\rho_f\bw\cdot\btau,
\end{align*}
for all $((\bu,\bw), p),((\bv,\btau), q)\in\bH$.  We observe that the formulation associated with the fluid corresponds to the acoustic equations written in terms of the displacement of the fluid. Following the ideas of \cite{MR1342293}, it is necessary to decompose $\H(\text{div},\O_f)$ in a suitable way. Also, 
we observe that the eigenspace associated with $\kappa=0$ is 
$
\mathcal{K}:=\big\{((\boldsymbol{0},\curl\xi),0):\ \xi\in\H_0^1(\O_f)\big\}\subset\bH,
$
and its orthogonal complement (in $\L^2$) denoted as $\mathcal{G}$ and defined by (see, e.g., \cite{MR1342293})
\begin{equation*}
	\mathcal{G}:=\big\{((\bu,\nabla\varphi),p): \bu\in\bV, \varphi\in \H^{1}(\O_f)^{d}, p\in\Q\big\}.
\end{equation*}
We also define the subspace $\mathcal{G}_{\bH}:=\mathcal{G}\cap\bH$. Moreover, it holds  that $\bH=\mathcal{K}\oplus\mathcal{G}_{\bH}$ (\cite[Lemma 2.3]{MR1342293}).

To perform the analysis,
we define the following parameter-weighted norm  for  all $((\bv,\btau),q)\in\bH$ 
\begin{align*}
	\vertiii{((\bv,\btau),q)}_{\bH}^2 & :=\Vert\mu(\bx)^{1/2}\nabla\bv\Vert_{0,\O_s}^2 + \Vert\mu(\bx)^{-1/2}q\Vert_{0,\O_s}^2 +\Vert\lambda(\bx)^{-1/2}q\Vert_{0,\O_s}^2 + \Vert\rho_S^{1/2} \bv\Vert_{0,\O_s}\\
	&\quad + \Vert (c^2\rho_f)^{1/2} \vdiv \btau\Vert_{0,\O_f} + \Vert\rho_f^{1/2}\btau \Vert_{0,\O_f} +\Vert(g\rho_f)^{1/2}\btau\cdot\bn\Vert_{0,\Gamma_0}^2.
\end{align*}

Let us now define the solution operator
$$ 			\bT:\bH\to\bH,\qquad 
((\bff,\bg),g)\mapsto \bT((\bff,\bg),g):=((\overline{\bu},\overline{\bw}),\overline{p}),
$$
where the triplet $((\overline{\bu},\overline{\bw}),\overline{p}) \in\bH$ is the solution  of  the following source problem:
\begin{equation}
	\label{eq:source_cont}
	A(((\overline{\bu},\overline{\bw}),\overline{p}),((\bv,\btau),q))   = B(((\bff,\bg),g),((\bv,\btau),q)),\qquad \forall((\bv,\bw),q)\in\bH.
\end{equation}

Let us recall the following result proved
in \cite[Theorem~2.5]{MR1342293}.
\begin{lemma}\label{regG}
	If $((\bff,\bg),g)\in\mathcal{G}$, then
	$\bT((\bff,\bg),g):=((\overline{\bu},\overline{\bw}),\overline{p})\in\mathcal{G}_{\bH}$ is the unique solution of \eqref{eq:source_cont}. Moreover, 
	there exist $\alpha\in(1/2,1]$, $\beta\in (0,1]$ and $C>0$ independent of $\lambda$,
	such that the following estimate holds true
	\begin{equation*}
		\|\overline{\bu}\|_{1+\beta,\O_s}+\|\overline{\bw}\|_{\alpha,\O_f}
		+\|\div\overline{\bw}\|_{1,\O_f}+\|\overline{p}\|_{\beta,\O_s}\leq
		C\|((\bff,\bg),g)\|_{\L^2(\Omega_s)^{d}\times \L^2(\Omega_f)^{d}\times \L^2(\Omega_s)}.
	\end{equation*}
\end{lemma}

We observe that Lemma \ref{regG} establishes that source problem \eqref{eq:source_cont} is well posed and, as a consequence of this, we have that $\bT$ is well defined. We also  notice that  operator $\bT$ is  non-compact, since $\H(\text{div},\O_f)$ is not compactly embedded in $\L^2(\O_f)^d$. Moreover, we observe that $(\kappa,(\bu,\bw),p)\in\mathbb{R}\times\bH$
solves \eqref{def:limit_system_eigen_complete} if and only if $((\kappa+1)^{-1},(\bu,\bw),p)$ is an eigenpair of $\bT$, i.e., if 	$((\bu,\bw),p)\neq ((\boldsymbol{0},\boldsymbol{0}),0)$ and $\bT((\bu,\bw),p)=(\kappa+1)^{-1}((\bu,\bw),p)$.

Moreover, as a consequence of Lemma \ref{regG}, we have the following additional regularity for the eigenfunctions, which holds when $((\bff,\bg),g)=(\kappa+1)((\bu,\bw),p)$. 
\begin{corollary}\label{reg_eigenfunctions}
	There exist $\alpha_1\in(1/2,1]$, $\beta_1\in (0,1]$ and $C>0$ independent of $\lambda$ but depending on the eigenvalue $\kappa$,
	such that the following estimate holds true
	\begin{equation*}
		\|\bu\|_{1+\beta_1,\O_s}+\|\bw\|_{\alpha_1,\O_f}
		+\|\div\bw\|_{1,\O_f}+\|p\|_{\beta_1,\O_s}\leq
		C\|((\bu,\bw),p)\|_{\L^2(\Omega_s)^{d}\times \L^2(\Omega_f)^{d}\times \L^2(\Omega_s)}.
	\end{equation*}
\end{corollary}

Next, the following result provides a spectral characterization of $\bT$. See \cite[Theorem~2.7]{MR1342293} for instance.
\begin{theorem}
	\label{CHAR_SP}
	The spectrum of $\bT$ decomposes as follows:
	$\sp(\bT)=\left\{0,1\right\}\cup\left\{\mu_k\right\}_{k\in\N}$, where:
	\begin{enumerate}
		\item[\textit{i)}] $\kappa=1$ is an infinite$-$multiplicity eigenvalue of
		$\bT$ and its associated eigenspace is $\mathcal{K}.$
		\item[\textit{ii)}] $\left\{\kappa_k\right\}_{k\in\N}\subset(0,1)$ is a
		sequence of finite-multiplicity eigenvalues of $\bT$ which converge to
		$0$
		and the corresponding eigenspaces lie in
		$\{(\bu,\bw),p)\in\H^{1+\beta_1}(\Omega_s)^d\times\H^{\alpha_1}(\Omega_f)^d \times\H^{\beta_1}(\Omega_s): \vdiv\bw\in\H^1(\Omega_f)\}$.
	\end{enumerate}
	
\end{theorem}

\section{Numerical discretization}
\label{sec:fem}
In this section our aim is to describe a FE discretization of problem \eqref{def:limit_system_eigen_complete}.   Let $\mathcal{T}_h(\O_s)$ and $\mathcal{T}_h(\O_f)$ be a conforming partition  of the polyhedral domains $\overline{\O}_S$ and $\Omega_F$, respectively, into triangles (tetrahedrons) $T$ with size $h_T=\text{diam}(T)$. Define $h:=\max\{h_T\,:\, T\in\mathcal{T}_h(\O_s)\cup\O_f\}$. Both the interface and the meshes are assumed conformal, that is, meshes are constructed such that the vertices of $\mathcal{T}_h(\O_s)$ and $\mathcal{T}_h(\O_f)$ coincide on $\Sigma_h=\Sigma$. 
On the other hand, the numerical method that we propose is nonconforming.  To make matters precise, on the solid part of the problem, the discretization will be considered under the approach of inf-sup stable families of FEs for Stokes, whereas for the fluid, Brezzi--Douglas--Marini elements will be considered. This is the key point on the conforming or non conforming nature of the methods. With this in mind, the analysis presented in \cite{MR3335232} will become essential for our purposes.  

Let $k\geq 1$ and $S\subseteq\mathbb{R}^{d}$. We denote by $\mathbb{P}_k(S)$ the space of polynomial functions defined on $S$ of total degree $\leq k$. In particular, given two families of  inf-sup stable FEs $\bV_h$ and $\Q_h$ to approximate the solid displacement $\bu_S$ and the pressure $p$ we can take, for example,  
\begin{itemize}
	\item[(a)] the MINI element \cite[Section 4.2.4]{MR2050138}: 
	\begin{align*}
		&\bV_h=\{\bv_{h}\in\boldsymbol{C}(\overline{\O})\ :\ \bv_{h}|_T\in[\mathbb{P}_1(T)\oplus\mathbb{B}(T)]^{d} \ \forall \ T\in\mathcal{T}_h(\O_s)\}\cap\bV,\\
		&\Q_h=\{ q_{h}\in C(\overline{\O})\ :\ q_{h}|_T\in\mathbb{P}_1(T) \ \forall \ T\in\mathcal{T}_h(\O_s) \},
	\end{align*}
	where $\mathbb{B}(T)$ denotes the space spanned by local bubble functions; or 
	\item[(b)] the Taylor--Hood element \cite[Section 4.2.5]{MR2050138}: 
	\begin{align*}
		&\bV_h=\{\bv_{h}\in\boldsymbol{C}(\overline{\O})\ :\ \bv_h|_T\in\mathbb{P}_{k+1}(T)^{d} \ \forall \ T\in\mathcal{T}_h(\O_s)\}\cap\bV,\\
		\label{eq:P_TH}
		&\Q_h=\{ q_{h}\in C(\overline{\O})\ :\ q_{h}|_T\in\mathbb{P}_k(T) \ \forall\ T\in\mathcal{T}_h(\O_s) \}.
	\end{align*}
\end{itemize}

Let us introduce the space to approximate the fluid displacement. We use the well-known  Brezzi--Douglas--Marini finite element space  $\mathbf{BDM}_k:=\mathbb{P}_k(\CT_h)^{d}\text{ with } k\geq 1$ (see \cite{MR799685}).
We set $\bW_h:=\mathbf{BDM}_k\cap\bW$ to be the corresponding global space. Then, as in \cite{MR1342293}, we need to impose a weaker condition than $(\bw -\bu)\cdot\bn=0$ on $\Sigma$,  for the discrete space. Then, we introduce the following space
$$
\bY_h:=\left\{ (\bv_h,\btau_h)\in\bV_h\times\bW_h \;:\; \bv_h\cdot\bn-\btau_h\cdot\bn=0\text{ on }\Sigma\right\}.
$$
Let us remark that the choice of BDM elements to approximate what concerns to the fluid leads to a conforming discretization since  $\bY_h\subset\bY$.

Now we  recall some well-known approximation properties for this FE family. 
First, for the fluid, given $s\in(0,1]$, let $\bPi_h: \H^s(\O_h)^{d}\cap\bW\rightarrow\bW_h$ be the classic global lowest order BDM interpolant operator which satisfies
\begin{equation*}
	\displaystyle\int_{E}\left(\bPi_h\btau\cdot\bn_{E}\right)\zeta=\int_{E}(\btau\cdot\bn_{E})\zeta,\hspace{0.3cm}\forall\zeta\in\mathcal{P}_1(E)^d,
\end{equation*}
where $E$ is an edge of any  $T\in\mathcal{T}_h(\O_f)$, and the following commutative diagram property holds true
\begin{equation*}
	\label{diagrama}
	\text{div}\left(\bPi_h\btau\right)=\mathcal{P}_h\left(\text{div}\btau\right)\hspace{0.3cm}\forall\btau\in \H^s(\O_f)^{d}\cap\H(\text{div};\O_f),
\end{equation*}
where $\mathcal{P}_h:\L^2(\O_f)^{d}\rightarrow\mathcal{U}_h$ is the $\L^2(\O_f)^{d}$-orthogonal projection onto 
\[
\mathcal{U}_h:=\{q_h\in L^2(\O_f): q_h|_T\in\mathbb{P}_{k-1}(T)\ \forall\,T\in\mathcal{T}_h(\O_f)\}.
\]
Moreover, for $s\in (0,1]$, the operators  $\bPi_h$ and $\mathcal{P}_h$  satisfy the following approximation properties
\begin{subequations}
	\begin{align}\label{errorL2}
		\|\boldsymbol{\tau}-\bPi_h\btau\|_{0,\O_f} &\leq Ch^s\left(\|\boldsymbol{\tau}\|_{s,\O_f}+\|\text{div}\boldsymbol{\tau}\|_{0,\O_f}\right)\quad\forall\boldsymbol{\tau}\in\H(\text{div};\O_f)\cap\H^s(\O_f)^{d},\\
		\label{errorDiv}
		\|\boldsymbol{\tau}-\bPi_h\boldsymbol{\tau}\|_{\text{div},\O_f}&\leq Ch^s \|\boldsymbol{\tau}\|_{\H^s(\text{div};\O_f)}\qquad\forall\boldsymbol{\tau}\in\H^s(\text{div};\O_f)\cap\bW_h,
		\\
		\label{errorPh}
		\|\bv-\mathcal{P}_h\bv\|_{0,\O_f}&\leq C h^s\|\bv\|_{s,\O_f}\qquad\forall\bv\in\H^{s}(\O_f)^{d}.
	\end{align}
\end{subequations}
In  the solid domain, we also introduce the orthogonal projection $\Lambda_h: \H^{1}(\O_s)^d\rightarrow \mathcal{M}_h$, where
\[
\mathcal{M}_h:=\{\bv_h\in\bV_h : \bv_h|_T\in\mathbb{P}_{k+1}(T)^{d}\ \forall\,T\in\mathcal{T}_h(\O_s)\},
\]
with $k=0$ for MINI element, and $k\geq 1$ for Taylor-Hood element. The projection satisfies 
\begin{equation}\label{errorH1}
	\|\bv-\Lambda_h\bv\|_{1,\O_s}\leq C h^s\|\bv\|_{1+s,\O_s}\quad\forall \bv\in\H^{1+s}(\O_s)^d.
\end{equation}
\subsection{The interface approximation}
Our goal is not only to approximate the solution $(\bu,\bw)\in\bY$ in $\O_f$ and $\O_s$, but also to prove the convergence of the proposed method on the interface $\Sigma$. This implies the construction of a corrected BDM interpolation operator to approximate the solution on $\Sigma$. With this aim, and inspired in the analysis of \cite{MR3283363}, the following sequel of results will provide the required corrected operator. Given $\bv\in\H^1(\O_s)^d$, let $\varphi\in\H^1(\O_f)$ be the solution of the following Neumann problem: 
\begin{equation}\label{neumannE}
	\displaystyle-\Delta\varphi=\frac{1}{|\O_f|}\int_{\Sigma}\bv\cdot\bn\quad\mbox{in}\,\,\O_f,\quad
	\displaystyle\frac{\partial\varphi}{\partial \bn}=0\quad\mbox{in}\,\,\Gamma_0,\quad\displaystyle\frac{\partial\varphi}{\partial \bn}=\bv\cdot\bn\quad\mbox{on}\,\,\Sigma.
\end{equation}


Applying classic regularity results for elliptic problems (see, e.g., \cite{MR0961439,jerison1989functional} for the case of Lipschitz domains not necessarily convex), we know that there exists $\hat{\upsilon}>0$ such that $\varphi \in \mathrm{W}^{2,\hat{\upsilon}}(\O_f)$ with $\frac32 - \varepsilon < \hat{\upsilon} < 3 + \varepsilon$ for a $\varepsilon >0$ depending on the domain. Therefore $\nabla \varphi \in \H^\upsilon(\Omega_f)^d$ for every $\upsilon < 2 - \frac{2}{\hat{\upsilon}}$, giving the worst bound $\upsilon < \frac{2}{3}$.
As a consequence of the trace inequality,  the following estimate holds true
\begin{equation}\label{acotagrad}
	\|\nabla\varphi\|_{\upsilon,\O_f}\leq C\|\bv\|_{1,\O_s}.
\end{equation}
On the other hand,  we define $\boldsymbol{\widehat{\psi}}:=-\nabla\varphi$ such that $\div\boldsymbol{\widehat{\psi}}\in\mathbb{R}$, according to the right hand side of the first equation of \eqref{neumannE}.
Let us define the following linear operator
\[
\bE: \H^1(\O_s)^d\rightarrow{\bW},\qquad 
\bv\mapsto\bE\bv:=-
\boldsymbol{\widehat{\psi}}.
\]
Observe  that  $\bE$ is bounded and provides an extension from $\O_s$ to $\O_f$ which, according to \eqref{acotagrad}, satisfies $\bE\bv\in\H^{\upsilon}(\O_f)^d$ with $\|\bE\bv\|_{\upsilon,\O_f}\leq C\|\bv\|_{1,\O_s}$ for all $\bv\in\H^1_{\Gamma_D}(\O_s)^d$. The following properties of $\bE$ hold true.
\begin{lemma}\label{estimaciones E}
	There exists a constant $C>0$ independent of $h$ such that
	\begin{align*}
		\|\bPi_h\bE\bv\|_{0,\O_f}&\leq  C\|\bv\|_{1,\O_s}\hspace{0.3cm}\forall \bv\in\H^1_{\Gamma_D}(\O_s)^d,\\
		\|\bE\bv-\bPi_h \bE\bv\|_{0,\O_f}&\leq  Ch^{\upsilon}\|\bv\|_{1,\O_s}\hspace{0.3cm}\forall\bv\in\H^1_{\Gamma_D}(\O_s)^d.
	\end{align*}
\end{lemma}
\begin{proof}
	The proof of these estimates follows verbatim the arguments of   \cite[Lemma~5.1]{MR3283363}.
\end{proof}

Next, let  $\bE_h$ be the discrete counterpart of $\bE$, defined by
$ \bE_h\bv:=\bPi_h\bE(\Lambda_h\bv)\in\bW_h$, for any $\bv\in\H^1_{\Gamma_D}(\O_s)^d$.
The following result provides an approximation property between operators $\bE$
and $\bE_h$ and its proof is available in \cite[Lemma 5.2]{MR3283363}.
\begin{lemma}\label{ERROR-E_E_h}
	There exists a constant $C>0$, independent of $h$, such that 
	\begin{equation*}
		\|\bE\bv-\bE_h\bv\|_{\div,\O_f}\leq C\left(h^{\upsilon}\|\bv\|_{1,\O_s}
		+\|\bv-\Lambda_h\bv\|_{1,\O_s}\right)\hspace{0.3cm}\forall \bv\in \H^1_{\Gamma_D}(\O_s)^d.
	\end{equation*} 
\end{lemma}

Our next task  is to prove that any smooth enough $(\bu,\bw)\in \bY$ is well approximated by functions of $\bY_h$. With this aim, we need to correct the BDM interpolant on the fluid to impose safely the continuity of the degrees of freedom of the BDM elements and the piecewise linear functions on the interface $\Sigma$.  To do this task, we define the corrected interpolant operator $\wbPi_h:\bY\rightarrow\bY_h$ as follows
\begin{equation*}\label{eq:wbPi}
	\wbPi_h(\bu,\bw):=\big(\Lambda_h\bu,\bPi_h\bw+(\bE_h\bu-\bPi_h\bE\bu)\big)\quad\forall (\bu,\bw)\in\bY.
\end{equation*}
The following lemma proves that $\wbPi_h(\bu,\bw)$ indeed lies in $\bY_h$. See \cite[Lemma~5.3]{MR3283363} for the details.
\begin{lemma}\label{ERROR-APROX-V-VH}
	Let $(\bu,\bw)\in \bY$ with $\bw\in\H^{\upsilon}(\O_f)^d$ and let  $(\bu_h,\bw_h):=\wbPi_h(\bu,\bw)$. Then, $(\bu_h,\bw_h)\in\bY_h$ and 
	\begin{equation*}\label{int_hat}
		\|(\bu,\bw)-(\bu_h,\bw_h)\|_{\bW\times\bV}\leq C\left(\|\bw-\bPi_h\bw\|_{\div,\O_f}+\|\bu-\Lambda_h\bu\|_{1,\O_s}\right).
	\end{equation*}
\end{lemma}

The first step is to construct an adequate interpolant for the correct approximation in the domain occupied by the fluid, the domain of the solid and the interface. With this aim, we introduce the \emph{corrected} operator $\bM_h:\bY\times \Q\rightarrow\bY_h\times \Q_h$, defined as follows:
$$
\bM_h((\bv,\btau),q):=
\begin{cases}
	((\boldsymbol{0},\boldsymbol{0}), \mathcal{P}_hq)\hspace{0.3cm}\mbox{if}\, T\subset\O_s,\\
	(\wbPi_h(\bv,\btau),0), \hspace{0.3cm} \mbox{otherwise.}
\end{cases}$$ 
Moreover,  the operator $\bM_h$ satisfies the following approximation property.
\begin{lemma}\label{APROX_M}
	There exists a positive constant $C$ such that,
	for all $((\bv,\btau),q)\in\H^{1+\beta}(\O_s)^d\times\H^{\alpha}(\O_f)^d\times \H^\beta(\O_s)$
	with $\div\btau\in\H^1(\O_f)$ there holds
	\begin{equation*}
		\vertiii{((\bv,\btau),q)-\bM_h((\bv,\btau),q)}_{\bH}\leq Ch^r(\|\btau\|_{\alpha,\O_f}+\|\div\btau\|_{1,\O_f}
		+\|\bv\|_{1+\beta,\O_s}+\|q\|_{\beta,\O_s}),
	\end{equation*}
	where $r=\min\{\alpha-1/2,\beta,k\}$, and $\alpha,\beta$ are as in Lemma \ref{regG}.
\end{lemma}
\begin{proof}
	Let $((\bv,\btau),q)\in\H^{1+\beta}(\O_s)^d\times\H^s(\O_f)^d\times \H^\beta(\O_s)$ with $\div\btau\in\H^1(\O_f)$.
	From the definition of $\bM_h$, direct computations reveal that
	\begin{align*}
		\vertiii{((\bv,\btau),q)-\bM_h((\bv,\btau),q)}_{\bH}^2 & \leq \Vert \mu(\bx)^{1/2}\nabla(\bv-\Lambda_h\bv)\Vert_{0,\O_S}^2 + \Vert\mu(\bx)^{-1/2}(q-\mathcal{P}_hq)\Vert_{0,\O_S}\\
		& \quad + \Vert\lambda(\bx)^{-1/2}(q-\mathcal{P}_hq)\Vert_{0,\O_S}^2 
		+ \Vert \rho_S^{1/2}(\bv-\Lambda_h\bv)\Vert_{0,\O_S}^2 \\
		& \quad + \Vert (c^2\rho_F)^{1/2}\div(\btau - (\bPi_{h}\btau+\bE_h\bv -\bPi_{h}\bE\bv))\Vert_{0,\O_f}^2\\
		&\quad +\Vert \rho_F^{1/2}(\btau - (\bPi_{h}\btau+\bE_h\bv -\bPi_{h}\bE\bv))\Vert_{0,\O_f}^2 \\
		&\quad + \Vert (g\rho_F)^{1/2}(\btau - (\bPi_{h}\btau+\bE_h\bv -\bPi_{h}\bE\bv))\cdot\bn\Vert_{0,\Gamma_0}^2 \\
		&\leq \bI_1 + \bI_2,
	\end{align*}
	where 
	\begin{align*}
		&\bI_1:=C_1\left(\Vert \nabla(\bv-\Lambda_h\bv)\Vert_{0,\O_S}^2 + \Vert q - \mathcal{P}_{h}q\Vert_{0,\O_S}^2 + \Vert \bv - \Lambda_h\bv\Vert_{0,\O_S}^2 + \Vert \div (\btau -\bPi_{h}\btau)\Vert_{0,\O_f}^2 \right. \\
		&\hspace{6cm}\left. + \Vert \btau -\bPi_{h}\btau\Vert_{0,\O_f}^2 +\Vert(g\rho_F)^{1/2} (\btau-\bPi_{h}\btau)\cdot\bn\Vert_{0,\Gamma_0}^2\right),\\
		&\bI_2:=C_2\left( \Vert \div(\bPi_{h}\bE(\bv-\Lambda_h\bv))\Vert_{0,\O_f}^2+ \Vert\bPi_{h}\bE(\bv-\Lambda_h\bv)\Vert_{0,\O_f}^2 +\Vert (\bPi_{h}\bE(\bv-\Lambda_h\bv)\cdot\bn)\Vert_{0,\Gamma_0}^2\right),
	\end{align*}
	with constants $C_1,C_2>0$   depending on $\mu_{\max},\mu_{\min},\lambda_{\min},\rho_S,\rho_F,c$ and $g$. 
	
	From  inverse inequality and the following estimate  (valid for the BDM interpolation applied to any $\btau \in \H^\alpha(\Omega_f)^d)$,
		$\Vert\bPi_{h}\btau\cdot\bn-\btau\cdot\bn\Vert_{0,\Gamma_0}\leq C h^{\zeta}\vert \btau\vert_{\alpha,\Omega_f}$, with $\zeta=\min\{\alpha-1/2,k\}$, we have,
		\begin{align*}
			\Vert \bPi_{h}\bE(\bv-\Lambda_h\bv)\cdot\bn\Vert_{0,\Gamma_0}&\leq \Vert \bPi_{h}\bE(\bv-\Lambda_h\bv)\cdot\bn-\bE(\bv-\Lambda_h\bv)\cdot\bn\Vert_{0,\Gamma_0}+\Vert \bE(\bv-\Lambda_h\bv)\cdot\bn\Vert_{0,\Gamma_0}\\
			&\leq C \left(h^{\zeta}\vert \bE( \bv - \Lambda_h\bv)\vert_{1/2+\zeta,\O_f} + \Vert \bE(\bv - \Lambda_h\bv)\Vert_{1/2+\zeta,\O_f}\right)\\
			&\leq C \left(h^{\zeta}\Vert \bv - \Lambda_h\bv\Vert_{1,\O_s} + \Vert \bv - \Lambda_h\bv\Vert_{1,\O_s}\right).
		\end{align*}
		Hence, the desired estimate follows using \eqref{errorL2}-\eqref{errorH1}, Lemma \ref{estimaciones E} and Lemma \ref{ERROR-E_E_h}.	
	\end{proof}
	
	To continue with the analysis,  we require a suitable Helmholtz decomposition. This result has been stated in \cite[Lemma 4.1]{MR1770352}, and with minor manipulations can be adapted to our purposes. We skip details for  brevity.
	\begin{lemma}\label{HELMOLTZ-DISC}
		For $((\bu_h,\bw_h),p_h)\in\mG_h$ we have the following decomposition 
		\[((\bu_h,\bw_h),p_h)=((\bu_h,\nabla\xi),p_h)+((\boldsymbol{0},\boldsymbol{\chi}),0) \quad \text{with} \quad \xi\in\H^{1+\widetilde{\upsilon}}(\O_f) \quad \text{and} \quad ((\boldsymbol{0},\boldsymbol{\chi}),0)\in\mK,\]  
		which satisfies
		\begin{align*}
			\|\nabla\xi\|_{\widetilde{\upsilon},\O_f}&\leq C(\|\div\bu_h\|_{0,\O_f}+\|\bw_h\|_{1,\O_s}+\|p_h\|_{0,\O_s}), \qquad \text{and} 
			\\\
			\|\boldsymbol{\chi}\|_{0,\O_f}&\leq Ch^{\widetilde{\upsilon}}(\|\div\bu_h\|_{0,\O_f}+\|\bw_h\|_{1,\O_s}+\|p_h\|_{0,\O_s}).
		\end{align*}
		where $\widetilde{\upsilon}:=\min\{\frac{1}{2},\upsilon  \}$.
	\end{lemma}

	\subsection{The discrete eigenvalue problem}
	Now we are in position to  introduce the discrete counterpart of the eigenvalue problem \eqref{def:limit_system_eigen_complete}. It consists in finding $((\bu_h,\bw_h),p_h))\in \bH_h:=\bY_h\times \Q_h$ such that 
	\begin{equation}\label{def:limit_system_eigen_complete_discrete}
		A(((\bu_h,\bw_h),p_h),((\bv_h,\btau_h),q_h))   =(\kappa_h+1) B(((\bu_h,\bw_h),p_h),((\bv_h,\btau_h),q_h)),\qquad \forall((\bv_h,\bw_h),q_h)\in\bH_h.
	\end{equation} 		
	Let us define the corresponding discrete solution operator
	$$
	\bT_h:\bH\to \bH_h,\qquad 
	((\bff,\bg),g)\mapsto \bT_h((\bff,\bg),g):=((\overline{\bu}_h,\overline{\bw}_h),\overline{p}_h),
	$$
	where $((\overline{\bu}_h,\overline{\bw}_h),\overline{p}_h)$ is the solution of the following discrete source problem:
	\begin{equation}
		\label{eq:source_disc}
		A(((\overline{\bu}_h,\overline{\bw}_h),\overline{p}_h),((\bv_h,\btau_h),q_h))   = B(((\bff,\bg),g),((\bv_h,\btau_h),q_h)),\qquad \forall((\bv_h,\bw_h),q_h)\in\bH_h.
	\end{equation}
	It is straightforward to prove that problem \eqref{eq:source_disc} is well posed and hence, the operator $\bT_h$ is well defined. On the other hand, 
	we observe that $(\kappa_h,(\bu_h,\bw_h),p_h)\in\mathbb{R}\times\bH_h\in\Q_h$
	solves \eqref{def:limit_system_eigen_complete_discrete} if and only if $((\kappa_h+1)^{-1},(\bu_h,\bw_h),p_h)$ is an eigenpair of $\bT_h$, i.e., if 	$((\bu_h,\bw_h),p_h)\neq ((\boldsymbol{0},\boldsymbol{0}),0)$ and $\bT_h((\bu_h,\bw_h),p_h)=(\kappa_h+1)^{-1}((\bu_h,\bw_h),p_h)$.

	Since $\mG_h\nsubseteq\mG$ we need to apply  the second Strang lemma  in order to derive an  approximation error. For problems \eqref{eq:source_cont} and \eqref{eq:source_disc} the Strang estimate reads as follows:
	\begin{align*}
		\nonumber\displaystyle\vertiii{((\overline{\bu},\overline{\bw}),\overline{p})-((\overline{\bu}_{h},\overline{\bw}_{h}),\overline{p}_h)}\leq & \left\{C\inf_{((\bv_h,\btau_h),q_h)\in\mG_h}\vertiii{((\overline{\bu},\overline{\bw}),\overline{p})-((\bv_h,\btau_h),q_h)}\right.\\
		&\left.+\sup_{((\boldsymbol{0},\boldsymbol{0}),0)\neq((\bv_h,\btau_h),q_h)\in\mG_h}\frac{A(((\overline{\bu},\overline{\bw}),\overline{p})-((\overline{\bu}_{h},\overline{\bw}_{h}),\overline{p}_h),((\bv_h,\btau_h),q_h))}{\vertiii{((\bv_h,\btau_h),q_h)}}\right\}.
	\end{align*} 
	Since $\bM_h((\overline{\bu},\overline{\bw}),\overline{p})\in\bH_h$ and $\bH_h=\mG_h\oplus\mK_h$, there exist $((\overline{\bu}_h,\overline{\bw}_h),\overline{p}_h)\in\mG_h$ and $((\overline{\bu}_{\mK_h}, \overline{\bw}_{\mK_h}),\overline{p}_{\mK_h})\in\mK_h$ such that $\bM_h((\overline{\bu},\overline{\bw}),\overline{p})=((\overline{\bu}_h,\overline{\bw}_h),\overline{p}_h)+((\overline{\bu}_{\mK_h}, \overline{\bw}_{\mK_h}),\overline{p}_{\mK_h})$. Then, since $((\overline{\bu},\overline{\bw}),\overline{p})\in\mG$ is orthogonal to $((\overline{\bu}_{\mK_h}, \boldsymbol{0}),0)\in\mK_h$, we observe that 
	\begin{align*}
		\vertiii{((\overline{\bu},\overline{\bw}),\overline{p})-((\overline{\bu}_h,\overline{\bw}_h),\overline{p}_h)}^2&\leq \vertiii{((\overline{\bu},\overline{\bw}),\overline{p})-((\overline{\bu}_h,\overline{\bw}_h),\overline{p}_h)}^2+\vertiii{((\bu_{\mK_h}, \boldsymbol{0}),0)}^2\\
		&=\vertiii{((\overline{\bu},\overline{\bw}),\overline{p})-((\overline{\bu}_h,\overline{\bw}_h),\overline{p}_h)+((\bu_{\mK_h}, \boldsymbol{0}),0)}^2\\
		&=\vertiii{((\overline{\bu},\overline{\bw}),\overline{p})-\bM_h((\overline{\bu},\overline{\bw}),\overline{p})}^2\\
		& \leq Ch^r(\|\overline{\bw}\|_{s,\O_f}+\|\div\overline{\bw}\|_{1,\O_f}+\|\overline{\bu}|_{1+\beta,\O_s}+\|\overline{p}\|_{1,\O_s}),
	\end{align*}
	where we have used the approximation property of $\bM_h$ given by Lemma \ref{APROX_M}. 
	
	Now for the consistency term, we procede as follows: Let $((\bv_h,\btau_h),q_h)$. Invoking Lemma \ref{HELMOLTZ-DISC}, let us consider the following decomposition
	$((\bv_h,\btau_h),q_h)=((\bv_h,\nabla\xi),q_h)+((\boldsymbol{0},\boldsymbol{\chi}),0)$ where $((\bv_h,\nabla\xi),q_h)\in\mG_h$ and 
	$((\boldsymbol{0},\boldsymbol{\chi}),0)\in\mK$, which holds according to Lemma \ref{HELMOLTZ-DISC}. Hence
	\begin{align}
		\label{eq:err_cont}
		A(((\overline{\bu},\overline{\bw}),&\overline{p}),((\bv_h,\btau_h),q_h)) =2\int_{\O_s}\mu(\bx)\beps(\overline{\bu}):\beps(\bv_h) - \int_{\O_s} \overline{p}\vdiv \bv_h + \int_{\O_f}c^2\rho_f\vdiv\overline{\bw} \vdiv \btau_h \nonumber \\
		& \quad + 
		\langle g\rho_f\overline{\bw}\cdot\bn,\btau_h\cdot\bn\rangle_{\Gamma_0}
		-\int_{\O_s}\vdiv\overline{\bu} q_h - \int_{\O_s}\frac{1}{\lambda}\overline{p} q_h +\int_{\O_s}\rho_S\overline{\bu}\cdot\bv_h + \int_{\O_f}\rho_f\overline{\bw}\cdot\btau_h \nonumber \\
		&=2\int_{\O_s}\mu(\bx)\beps(\overline{\bu}):\beps(\bv_h) - \int_{\O_s} \overline{p}\vdiv \bv_h + \int_{\O_f}c^2\rho_f\vdiv\overline{\bw} \vdiv (\nabla\xi+\boldsymbol{\chi}) \nonumber \\
		& \quad + 
		\langle g\rho_f\overline{\bw}\cdot\bn,(\nabla\xi+\boldsymbol{\chi})\cdot\bn\rangle_{\Gamma_0}
		-\int_{\O_s}\vdiv\overline{\bu} q_h - \int_{\O_s}\frac{1}{\lambda}\overline{p} q_h +\int_{\O_s}\rho_S\overline{\bu}\cdot\bv_h \nonumber \\
		& \quad + \int_{\O_f}\rho_f\overline{\bw}\cdot(\nabla\xi+\boldsymbol{\chi}) \nonumber \\
		&			=2\int_{\O_s}\mu(\bx)\beps(\overline{\bu}):\beps(\bv_h) - \int_{\O_s} \overline{p}\vdiv \bv_h - \int_{\O_s}\frac{1}{\lambda}\overline{p} q_h + \int_{\O_f}c^2\rho_f\vdiv\overline{\bw} \vdiv \nabla\xi \nonumber \\
		&\quad +
		\langle g\rho_f\overline{\bw}\cdot\bn,\nabla\xi\cdot\bn\rangle_{\Gamma_0} 
		-\int_{\O_s}\vdiv\overline{\bu} q_h  +\int_{\O_s}\rho_S\overline{\bu}\cdot\bv_h + \int_{\O_f}\rho_f\overline{\bw}\cdot\nabla\xi \nonumber \\
		&			=\int_{\O_s}\rho_S\boldsymbol{f}_h\cdot\bv_h + \int_{\O_f}\rho_f\boldsymbol{g}_h\cdot\nabla\xi,
	\end{align}
	where we have used that $\mathcal{G}\oplus\mathcal{K}$, where $\vdiv\boldsymbol{\chi}=0$ in $\O_f$ and $\boldsymbol{\chi}\cdot\boldsymbol{n}=0$ on $\Gamma_0$ and that $((\overline{\bu},\overline{\bw}),\overline{p})$ is the solution of the source problem \eqref{eq:source_cont}.  
	
	On the other hand, there holds
	\begin{multline}
		\label{eq:err_disc}
		A(((\overline{\bu}_h,\overline{\bw}_h),\overline{p}_h),((\bv_h,\btau_h),q_h))=2\int_{\O_s}\mu(\bx)\beps(\overline{\bu}_h):\beps(\bv_h) - \int_{\O_s} \overline{p}_h\vdiv \bv_h + \int_{\O_f}c^2\rho_f\vdiv\overline{\bw}_h \vdiv \btau_h \\
		+ \int_{\Gamma_0}(g\rho_f\overline{\bw}_h\cdot\bn)(\btau_h\cdot\bn) 
		-\int_{\O_s}\vdiv\overline{\bu}_h q_h - \int_{\O_s}\frac{1}{\lambda}\overline{p}_h q_h +\int_{\O_s}\rho_S\overline{\bu}_h\cdot\bv_h + \int_{\O_f}\rho_f\overline{\bw}_h\cdot\btau_h\\
		=\int_{\O_s}\rho_S\boldsymbol{f}_h\cdot\bv_h + \int_{\O_f}\rho_f\boldsymbol{g}_h\cdot(\nabla\xi+\boldsymbol{\chi}).
	\end{multline}
	Now, subtracting \eqref{eq:err_disc} from \eqref{eq:err_cont}, applying the Cauchy--Schwarz inequality and invoking Lemma \ref{HELMOLTZ-DISC} we have
	\begin{multline*}
		A(((\overline{\bu},\overline{\bw}),\overline{p})-((\overline{\bu}_{h},\overline{\bw}_{h}),\overline{p}_h),((\bv_h,\btau_h),q_h))=\int_{\O_f}\rho_f\boldsymbol{g}_h\cdot\boldsymbol{\chi}\\
		\leq C_{\rho_f}\|\boldsymbol{g}_h\|_{0,\O_f}\|\boldsymbol{\chi}\|_{0,\O_f}\leq  Ch^s\|\boldsymbol{g}_h\|_{0,\O_f}(\|\div\btau_h\|_{0,\O_f}+\|\bv_h\|_{1,\O_s}+\|q_h\|_{0,\O_s}).
	\end{multline*}
	Finally, taking supremum, we readily obtain the bound
	\begin{equation*}
		\sup_{((\boldsymbol{0},\boldsymbol{0}),0)\neq((\bv_h,\btau_h),q_h)\in\mG_h}\frac{A(((\overline{\bu},\overline{\bw}),\overline{p})-((\overline{\bu}_{h},\overline{\bw}_{h}),\overline{p}_h),((\bv_h,\btau_h),q_h))}{\vertiii{((\bv_h,\btau_h),q_h)}}\leq Ch^r\|\boldsymbol{g}_h\|_{0,\O_f}.
	\end{equation*}
	
	\section{Spectral approximation}
	\label{sec:spec_app}
	We begin this section recalling some definitions from classical spectral theory. Let $\mathcal{X}$ be a generic Hilbert space and let $\bS$ be a linear bounded operator defined by $\bS:\mathcal{X}\rightarrow\mathcal{X}$. If $\boldsymbol{I}$ represents the identity operator, the spectrum of $\bS$ is defined by $\sp(\bS):=\{z\in\mathbb{C}:\,\,(z\boldsymbol{I}-\bS)\,\,\text{is not invertible} \}$ and the resolvent is its complement $\rho(\bS):=\mathbb{C}\setminus\sp(\bS)$. For any $z\in\rho(\bS)$, we define the resolvent operator of $\bS$ corresponding to $z$ by $R_z(\bS):=(z\boldsymbol{I}-\bS)^{-1}:\mathcal{X}\rightarrow\mathcal{X}$. 
	
	Also, if $\mathcal{X}$ and $\mathcal{Y}$ are vectorial fields, we denote by $\mathcal{L}(\mathcal{X},\mathcal{Y})$ the space of all the linear and bounded operators acting from $\mathcal{X}$ to $\mathcal{Y}$. We define $\bcX:=\bcW\times\bcQ$ and $\bcX_h:=\bcW_h\times\bcQ_h$. 
	
	Let $\mathcal{S}$ be a linear operator defined by $\mathcal{S}:\mathcal{X}\rightarrow\mathcal{X}$. We define the norm $\norm{\cdot}_h$, associated with $\mathcal{S}$ as follows
	\begin{equation}
		\label{eq:normh}
		\displaystyle\norm{\mathcal{S}}_h:=\sup_{\boldsymbol{x}_h\in\bcX_h}\frac{\vertiii{\mathcal{S}\boldsymbol{x}_h}}{\vertiii{\boldsymbol{x}_h}}.
	\end{equation}
	
	Let $\boldsymbol{x}\in\bcX$ and   $\mathcal{H}$ and $\mathcal{J}$ be two closed subspaces of $\bcX$. We define the following distances
	\begin{equation*}
		\displaystyle\delta\big(\boldsymbol{x}, \mathcal{H}\big):=\inf_{\boldsymbol{y}\in\mathcal{H}}\norm{\boldsymbol{x}-\boldsymbol{y}}, \quad\delta\big(\mathcal{H},\mathcal{J}\big):=\sup_{\boldsymbol{y}\in\mathcal{H}, \norm{\boldsymbol{y}}=1}\delta\big(\boldsymbol{y}, \mathcal{J} \big),
	\end{equation*}
	and hence, the so-called gap between two subspaces
	\begin{equation*}
		\displaystyle\widehat{\delta}\big(\mathcal{H}, \mathcal{J} \big):=\max\left\{\delta\big(\mathcal{Y}, \mathcal{J} \big), \delta\big(\mathcal{J}, \mathcal{H}\big) \right\}.
	\end{equation*}
	
	As we claimed before, our aim is to apply the non-compact theory of \cite{MR483400} to derive the convergence of the proposed method. This implies the validation of the following properties
	\begin{enumerate}
		\item[[P1\!\!]]  $\norm{\bT-\bT_h}_h\rightarrow 0$, as $h\rightarrow 0$.
		\\
		\item[[P2\!\!]] $\forall\boldsymbol{x}\in\boldsymbol{\mathcal{X}}$, $\displaystyle \lim_{h\rightarrow 0}\delta\big(\boldsymbol{x}, \boldsymbol{\mathcal{X}}_h \big)=0$.
	\end{enumerate}
	
	We observe that property [P2] holds due to the smoothness of the eigenfunctions given by Corollary \ref{reg_eigenfunctions} and the approximation properties \eqref{errorL2}, \eqref{errorDiv},\eqref{errorPh}, \eqref{errorH1}, and Lemma \ref{APROX_M}. The task now is to prove [P1].

	\begin{lemma}\label{lem:P1}
		Property $\mathrm{[P1]}$ holds true in the following sense: there exists a  constant $C>0$ independent of $h$, such that
		\begin{equation*}
			\|\bT-\bT_h\|_h\leq Ch^r,
		\end{equation*}
		where the parameter $r$ is given by  Lemma \ref{APROX_M}.
	\end{lemma}
	\begin{proof}
		Since $\bH_h\subset\bH$, for any discrete source $\boldsymbol{0}\neq((\bff_h,\bg_h),g_h)\in\bH_h$, operators $\bT$ and $\bT_h$ are well defined. Then, using the definition \eqref{eq:normh}, Lemma 3.4,  and the dependence on the data,  we obtain
		$$
		\begin{aligned}
			\Vert \bT-\bT_h\Vert_h &= \sup_{\boldsymbol{0}\neq((\bff_h,\bg_h),g_h)\in\bH_h} \frac{\vertiii{(\bT-\bT_h)((\bff_h,\bg_h),g_h)}_{\bH}}{\vertiii{((\bff_h,\bg_h),g_h)}_{\bH}}\\
			&=\sup_{\boldsymbol{0}\neq((\bff_h,\bg_h),g_h)\in\bH_h} \frac{\vertiii{((\overline{\bu},\overline{\bw}),\overline{p}) - ((\overline{\bu}_h,\overline{\bw}_h),\overline{p}_h)}_{\bH}}{\vertiii{((\bff_h,\bg_h),g_h)}_{\bH}}\\
			&\leq \sup_{\boldsymbol{0}\neq((\bff_h,\bg_h),g_h)\in\bH_h} \frac{\vertiii{((\overline{\bu},\overline{\bw}),\overline{p}) - \bM_h ((\overline{\bu},\overline{\bw}),\overline{p})}_{\bH}}{\vertiii{((\bff_h,\bg_h),g_h)}_{\bH}}\\
			&\leq \sup_{\boldsymbol{0}\neq((\bff_h,\bg_h),g_h)\in\bH_h} \frac{Ch^r(\|\overline{\bw}\|_{s,\O_f}+\|\div\overline{\bw}\|_{1,\O_f}
				+\|\overline{\bu}\|_{1+\beta,\O_s}+\|\overline{p}\|_{\beta,\O_s})}{\vertiii{((\bff_h,\bg_h),g_h)}_{\bH}}\\
			&\leq \sup_{\boldsymbol{0}\neq((\bff_h,\bg_h),g_h)\in\bH_h} \frac{Ch^r\vertiii{((\bff_h,\bg_h),g_h)}_{\bH}}{\vertiii{((\bff_h,\bg_h),g_h)}_{\bH}} \leq Ch^r.
		\end{aligned}
		$$
		This concludes the proof.
	\end{proof}
	
	A key  consequence of the previous result is that we are in position to apply the well established theory of  \cite{MR0203473}  to conclude that the proposed numerical method does not introduce spurious eigenvalues. This is stated in the following theorem.
	\begin{theorem}
		\label{thm:spurious_free}
		Let $V\subset\mathbb{C}$ be an open set containing $\sp(\bT)$. Then, there exists $h_0>0$ such that $\sp(\bT_h)\subset V$ for all $h<h_0$.
	\end{theorem}
	
	Let us recall the definition of spectral projectors. Let $\mu$ be a nonzero isolated eigenvalue of $\bT$ with algebraic multiplicity $m$ and let $\Gamma$
	be a disk of the complex plane centered in $\mu$, such that $\mu$ is the only eigenvalue of $\bT$ lying in $\Gamma$ and $\partial\Gamma\cap\sp(\bT)=\emptyset$. With these considerations at hand, we define the spectral projections of $\boldsymbol{E}$  associated with $\bT$ as follows:
	$$\displaystyle \boldsymbol{E}:=\frac{1}{2\pi i}\int_{\partial\Gamma} (z\boldsymbol{I}-\bT)^{-1}\dz.$$
	Note that $\boldsymbol{E}$  is the projection onto the generalized eigenvector space $R(\boldsymbol{E})$.
	A consequence of Lemma \ref{lem:P1} is that there exist $m$ eigenvalues, which lie in $\Gamma$, namely $\mu_h^{(1)},\ldots,\mu_h^{(m)}$, repeated according their respective multiplicities, that converge to $\mu$ as $h$ goes to zero. With this result at hand, we introduce a spectral projection
	\begin{equation*}
		\boldsymbol{E}_h:=\frac{1}{2\pi i}\int_{\partial\Gamma} (z\boldsymbol{I}-\bT_h)^{-1}\dz,
	\end{equation*}
	which is a projection onto the discrete invariant subspace $R(\boldsymbol{E}_h)$ of $\bT$, spanned by the generalized eigenvectors of $\bT_h$ corresponding to 
	$\mu_h^{(1)},\ldots,\mu_h^{(m)}$.

	For the eigenfunctions, the following estimate holds true, which is nothing else but an implication of properties [P1]-[P2], and  the theory of \cite{MR483400}.
	\begin{theorem}\label{gap-eigenspaces}
		There exist constants $h_0>0$ and $C>0$ such that, for all $h\leq h_0$,   
		\begin{equation*}
			\widehat{\delta}\left( R(\boldsymbol{E}_h),R(\boldsymbol{E})\right)\leq Ch^r.
		\end{equation*}
	\end{theorem}
	On the other hand, regarding the eigenvalues, we have the following estimates. 
	\begin{lemma}
		There exists a positive constant $C$ independent of $\lambda$ such that
		\begin{equation*}
			|\kappa-\kappa_h|\leq C h^{2r}.
		\end{equation*}
	\end{lemma}
	\begin{proof}
		Let us define $\bU:=((\bu,\bw),p)$ and $\bU_h:=((\bu_h,\bw_h),p_h)$ such that  the following algebraic identity holds 
		\begin{equation}
			\label{eq:padra}
			A(\bU-\bU_h,\bU-\bU_h)-\kappa B(\bU-\bU_h,\bU-\bU_h)=(\kappa_h-\kappa)B(\bU_h,\bU_h).
		\end{equation}
		Taking modulus on the above identity and using Young's inequality, we have
		\begin{align}
			\label{eq:first_bound}
			|A(\bU-\bU_h,&\bU-\bU_h)|\\
			&\leq 2\|\mu^{1/2}\beps(\bu-\bu_h)\|_{0,\O_s}^2+2\|p-p_h\|_{0,\O_s}\|\text{div}(\bu-\bu_h)\|_{0,\O_s}+\|(c\rho_f)^{1/2}\text{div}(\bw-\bw_h)\|_{0,\O_f}^2 \nonumber \\
			&\quad  +\|(g\rho_f)^{1/2}(\bw-\bw_h)\cdot\boldsymbol{n}\|_{0,\Gamma_0}^2 +\|\lambda^{-1/2}(p-p_h)\|_{0,\O_s}^2\nonumber \\
			&\quad +\|\rho_f^{1/2}(\bu-\bu_h)\|_{0,\O_f}^2 +\|\rho_s^{1/2}(\bw-\bw_h)\|_{0,\O_f}^2\nonumber\\
			& \leq 2C\|\bnabla(\bu-\bu_h)\|_{0,\O_s}^2+(C_{\lambda_{\min}}+1)\|\lambda^{-1/2}(p-p_h)\|_{0,\O_s}^2+\|\rho_s^{1/2}(\bu-\bu_h)\|_{0,\O_s}^2\nonumber\\
			&\quad +\|\text{div}(\bu-\bu_h)\|_{0,\O_s}^2 +(C_{\rho_f,c}+1)\|(\rho_f c^2)^{1/2}\text{div}(\bw-\bw_h)\|_{0,\O_f}^2\nonumber\\
			&\quad +\|(g\rho_f)^{1/2}(\bw-\bw_h)\cdot\boldsymbol{n}\|_{0,\Gamma_0}^2+\|\rho_s^{1/2}(\bw-\bw_h)\|_{0,\O_f}^{1/2} \nonumber \\
			& \leq C (\widehat{\delta}\left( R(\boldsymbol{E}_h),R(\boldsymbol{E})\right))^2\leq Ch^{2r},
		\end{align}
		where we have used the definition of the continuous and discrete spectral projections and the gap.  
		Also
		\begin{equation}
			\label{eq:second_bound}
			|\kappa| |B(\bU-\bU_h,\bU-\bU_h)|\leq C(\|\rho_s^{1/2}(\bw-\bw_h)\|_{0,\O_f}^{1/2}+\|\rho_s^{1/2}(\bu-\bu_h)\|_{0,\O_s}^2)\leq C(\widehat{\delta}\left( R(\boldsymbol{E}_h),R(\boldsymbol{E})\right))^2\leq Ch^{2r}.
		\end{equation}
		On the other hand, it is straightforward to prove that $B(\bU_h,\bU_h)>0$. Hence, replacing  \eqref{eq:first_bound}, \eqref{eq:second_bound} in \eqref{eq:padra}, we conclude the proof.
	\end{proof}	
	
	\section{A posteriori error analysis}
	\label{sec:apost}
	The objective of this section is to present a  residual-based error estimator and  to demonstrate the equivalence between the proposed estimator and the true error. Through all this section we will be focused on eigenvalues with simple multiplicity. 
	
	Let us set some preliminary definitions and notations. For any   $K\in \CT_h$, we denote by $\CE_{K}$ its set of facets 
	and 
	$$\CE_h:=\bigcup_{K\in\CT_h}\CE_{K}.$$
	We partition $\CE_h$ into $\CE_{\O}$ and $\CE_{\partial\O}$, where $\CE_{\partial\O}$ is defined as the set of edges $\ell\in \CE_h$ such that $\ell$ lies on the boundary $\partial\O$, and $\CE_{\O}$ comprises the edges in $\CE$ excluding those in $\CE_{\partial\O}$. 
	Let $K$ and $K'$ be the two elements in $\CT_{h}$ that share the edge $\ell$, and let $\boldsymbol{n}_{K}$ and $\boldsymbol{n}_{K'}$ denote their respective outer unit normal vectors.
	For each internal edge $\ell\in \CE_{\O}$ and any sufficiently smooth function $\bv$, we define the jump of its normal derivative on $\ell$ as follows:
	$$\left[\!\!\left[ \dfrac{\partial \bv}{\partial{ \boldsymbol{n}}}\right]\!\!\right]_\ell:=\nabla (\bv|_{K})  \cdot \boldsymbol{n}_{K}+\nabla ( \bv|_{K'}) \cdot \boldsymbol{n}_{K'}.$$
	
	
	\subsection{Residual based a posteriori error estimator}
	Regarding the elastic structure, we can express the local element residual within each element (denoted by $\eta_{K,S}$) and the discrepancy between neighboring elements' solutions at their edges (denoted by $\eta_{J,S}$) as
	\begin{equation*}
		\eta_{K,S}^2:= h_K^2\|\rho_1^S\textbf{R}_{1,T}^S\|_{0,T}^2+\|(\rho_2^S)^{1/2}R_{2,T}^S\|_{0,T}^2,\quad \eta_{J,S}^2:=h_E\|\rho_E^SJ_\ell^S \|_{0,\ell}^2,
	\end{equation*}
	where 
	\begin{equation*}
		\textbf{R}_{1,T}^S := \bdiv(2\mu_h\beps(\bu_h)) -\nabla p_h + \omega^2_h\rho_S \bu_h \quad\text{and}\quad
		R_{2,T}^S := \vdiv\bu_h +\frac{1}{\lambda}p_h.
	\end{equation*}
	For the jump terms, we define 
	\begin{align*}
		J_\ell^S &:= \begin{cases}
			\displaystyle \frac{1}{2}[\![{(2\mu_h\beps(\bu_h) -p_h\mathbb{I})\bn}]\!], & E\in \Omega_S\cap\mathcal{E}_h\\
			{(2\mu_h\beps(\bu_h) -p_h\mathbb{I})\bn}, & E\in \Gamma_N\cap\mathcal{E}_h\\
			\boldsymbol{0}, & E\in \Gamma_D\cap\mathcal{E}_h.
		\end{cases}
	\end{align*}
	Here, the parameters $\rho_1$, $\rho_2$ and $\rho_3$ are defined by
	\[
	\rho_{1}^S:=  (2\mu_h)^{-1/2},\quad  \rho_2^S:= \left[(2\mu_h)^{-1}+\lambda^{-1}\right]^{-1},\quad
	\rho_E^S := (2\mu_{h})^{-1/2}/\sqrt{2}, 
	\]
	where $\mu_h$ corresponds to the $\L^2$ polynomial projection of $\mu$.  The local data oscillation is characterized by $ \Theta_{K,S}^2=\|\rho_1^S(\mu-\mu_h)\beps(\bu_h)\|^2_0$. To conclude, let us delve into the definition of the global a posteriori estimator $\eta_S$ alongside the global data oscillation error $\Theta_S$, expressed as:
	\begin{equation*}
		\eta_S^2: = \sum_{K\in\mathcal{T}_h}\eta_{K,S}^2,\quad \Theta_S^2 := \sum_{K\in\mathcal{T}_h}\Theta_{K,S}^2.
	\end{equation*}
	
	Next, for fluid part, we can express the local element residual within each element (denoted by $\eta_{K,F}$) and the discrepancy between neighboring elements' solutions at their edges (denoted by $\eta_{J,F}$) as
	\begin{equation*}
		\eta_{K,F}^2:= h_K^2\|\rho^F\textbf{R}_{1,T}^F\|_{0,T}^2+h_K^2\|\rho^FR_{2,T}^F\|_{0,T}^2,\quad \eta_{J,F}^2:=h_E\|\rho_E^FJ_{1,\ell}^F \|_{0,\ell}^2+h_E\|\rho_E^FJ_{2,\ell}^F \|_{0,\ell}^2,
	\end{equation*}
	where 
	\begin{equation*}
		\textbf{R}_{1,T}^F := c^2\rho_F\nabla(\vdiv\bw_h)+\omega^2_h\rho_F\bw_h,\qquad 
		R_{2,T}^F := \rot(\omega^2_h\rho_F\bw_h). 
	\end{equation*}
	For the jump terms, we define 
	\begin{align*}
		J_{1,\ell}^F &:= \begin{cases}
			\displaystyle \frac{1}{2}[\![c^2\rho_F(\vdiv\bw_h)\bn]\!], & E\in \Omega_f\cap\mathcal{E}_h\\
			c^2\rho_F(\vdiv\bw_h) \bn, & E\in \Gamma_N\cap\mathcal{E}_h\\
			\boldsymbol{0}, & E\in \Gamma_D\cap\mathcal{E}_h.
		\end{cases};
		\quad
		J_{2,\ell}^F := \begin{cases}
			\displaystyle \frac{1}{2}[\![\omega^2_h\rho_F\bw_h\times\bn]\!], & E\in \Omega_f\cap\mathcal{E}_h\\
			\omega^2_h\rho_F\bw_h\times \bn, & E\in \Gamma_N\cap\mathcal{E}_h\\
			\boldsymbol{0}, & E\in \Gamma_D\cap\mathcal{E}_h.
		\end{cases}
	\end{align*}
	Here, the parameters $\rho^F$ and $\rho_E^F$ are defined by
	$$
	\rho^F:=  (c^2\rho_F)^{-1/2}, \quad    	\rho_E^F := \min\{(c^2\rho_F)^{-1/2}/\sqrt{2}, (\omega^2_h\rho_F)^{-1/2}/\sqrt{2}\}.  
	$$
	
	Finally, denoting $\rho^I=\min\{\rho^F_E,\rho^S_E\}$, we define of the estimator for the interface as
	$$
	\eta_{E,I}^2 :=h_E\| \rho^{I}(\left(2\mu\beps(\bu_h) - p\mathbb{I} \right)\bn-(\rho_f c^2\vdiv\bw_h)\bn)\|^2_{0,E}+h_E \|\rho_E^F\omega^2_h\rho_F\bw_h\times\bn\|^2_{0,E}.
	$$

	\subsection{Reliability} 
	Let us begin with the reliability analysis for the  a posteriori estimator. The following result is instrumental.} 
\begin{lemma}\label{rel11}
	For every $((\bu,\bw),p)\in \bY\times \Q$, there exists $((\bv,\btau),q)\in \bY\times \Q$ with $\vertiii{((\bv,\btau),q)}_{\bH}\lesssim \vertiii{((\bu,\bw),p)}_{\bH}$ such that 
	\begin{align*}
		\vertiii{((\bu,\bw),p)}_{\bH}^2 \lesssim {A}(((\bu,\bw),p),((\bv,\btau),q)).
	\end{align*}
\end{lemma}
\begin{proof}
	From the definition of bilinear form $A(\cdot,\cdot)$, it follows that
	\begin{align*}
		{A}(((\bu,\bw),p),((\bu,\bw),-p))=& \Vert(2\mu(\bx))^{1/2}\beps(\bu)\Vert_{0,\O_s}^2 +\Vert\lambda(\bx)^{-1/2}p\Vert_{0,\O_s}^2 + \Vert\rho_S^{1/2} \bu\Vert_{0,\O_s}^2\\
		&+ \Vert (c^2\rho_f)^{1/2} \vdiv \bw\Vert_{0,\O_f} + \Vert\rho_f^{1/2}\bw \Vert_{0,\O_f} +\Vert(g\rho_f)^{1/2}\bw\cdot\bn\Vert_{0,\Gamma_0}^2.
	\end{align*}
	Applying the standard inf-sup condition \cite[Eq. 2.6 ]{khan2023finite} shows that for every $p\in \Q$, there exists an element $\tilde{\bv}\in \bY$ such that 
	\begin{align*}
		-(p,\rm{div} \tilde{\bv})\ge C_2 \|p\|_{0,\O_s}^2,\qquad \Vert\mu(\bx)^{1/2}\beps(\tilde{\bv})\Vert_{0,\O_s}\le C_1 \|\mu(\bx)^{-1/2}p\|_{0,\O_s}.
	\end{align*}  
	This finding implies that
	\begin{align*}
		{A}(((\bu,\bw),p)&,((\tilde{\bv},\boldsymbol{0}),0))=2\int_{\O_s}\mu(\bx)\beps(\bu):\beps(\bv) - \int_{\O_s} p\vdiv \tilde{\bv} + \int_{\O_s}\rho_S\bu\cdot\tilde{\bv}\\
		&\ge  C_2 \|p\|_{0,\O_s}^2- C_1  \Vert\mu(\bx)^{1/2}\beps(\bu)\Vert_{0,\O_s} \Vert\mu(\bx)^{-1/2}p\Vert_{0,\O_s}  - \Vert\rho_S^{1/2} \bu\Vert_{0,\O_s} \Vert\rho_S^{1/2} \tilde{\bv}\Vert_{0,\O_s}\\
		&\ge (C_2-1/2\epsilon_1)  \|p\|_{0,\O_s}^2- (C_1^2\epsilon_1/2 ) \Vert\mu(\bx)^{1/2}\beps(\bu)\Vert_{0,\O_s}^2- \Vert\rho_S^{1/2} \bu\Vert_{0,\O_s} C_3 \|\mu(\bx)^{-1/2}p\|_{0,\O_s}\\
		&\ge (C_2-1/2\epsilon_1-1/2\epsilon_2)  \|p\|_{0,\O_s}^2- (C_1^2\epsilon_1/2 ) \Vert\mu(\bx)^{1/2}\beps(\bu)\Vert_{0,\O_s}^2- (C_3^2\epsilon_2/2 )\Vert\rho_S^{1/2} \bu\Vert_{0,\O_s}^2.
	\end{align*}
	By taking the specific choices $\bv:= \bu+\delta \tilde{\bv}$, $\btau=\bw$ and $q=-p$, it follows that
	\begin{align*}
		{A}(((\bu,&\bw),p),((\bv,\btau),q))={A}(((\bu,\bw),p),((\bu,\bw),-p))+\delta{A}(((\bu,\bw),p),((\tilde{\bv},\boldsymbol{0}),0))\\
		&\ge \Vert\mu(\bx)^{1/2}\beps(\bu)\Vert_{0,\O_s}^2 +\Vert\lambda(\bx)^{-1/2}p\Vert_{0,\O_s}^2 + \Vert\rho_S^{1/2} \bu\Vert_{0,\O_s}^2\\
		&\quad+ \Vert (c^2\rho_f)^{1/2} \vdiv \bw\Vert_{0,\O_f}^2 + \Vert\rho_f^{1/2}\bw \Vert_{0,\O_f}^2 +\Vert(g\rho_f)^{1/2}\bw\cdot\bn\Vert_{0,\Gamma_0}^2\\
		&\quad +\delta((C_2-1/2\epsilon_1-1/2\epsilon_2)  \|p\|_{0,\O_s}^2- (C_1^2\epsilon_1/2 ) \Vert\mu(\bx)^{1/2}\beps(\bu)\Vert_{0,\O_s}^2- (C_3^2\epsilon_2/2 )\Vert\rho_S^{1/2} \bu\Vert_{0,\O_s}^2).
	\end{align*}
	Choosing $\epsilon_1 =\epsilon_2= 2/C_2$ and $\delta=1/2\min\{C_2/C_1^2,C_2/C_3^2\}$ implies that
	\begin{align*}
		{A}(((\bu,\bw),p),((\bv,\btau),q))\gtrsim \vertiii{((\bu,\bw),p)}_{\bH}^2.
	\end{align*}
	Moreover, it also holds that
	$$
	\vertiii{((\bv,\btau),q)}_{\bH}^2=\vertiii{((\bu+\tilde{\bv},\bw),p)}_{\bH}^2\lesssim  \vertiii{((\bu,\bw),p)}_{\bH}^2.
	$$
\end{proof}

To derive the upper bounds of the the fluid part,  
we employ the Helmholtz-decomposition based  approach as discussed in \cite{caucao2022posteriori}.
For sufficiently smooth scalar $\psi$, and vector $\bv := (v_1,v_2)^\rt$,   we let
\[
\curl(\psi):=\Big(\frac{\partial\psi}{\partial x_2}\,,\,-\frac{\partial\psi}{\partial x_1}\Big)^{\tt t}
\,,\quad
\rot(\bv):=\frac{\partial v_2}{\partial x_1}-\frac{\partial v_1}{\partial x_2}\,.
\]
By applying  \cite[Lemma 4.4]{caucao2022posteriori}, for each $\btau \in  \bW$, we can find $\bz\in \H^2(\O_f)$ and $\phi\in \H^1_{\Sigma}(\O_f)$, such that
\begin{align*}
	\btau= \nabla \bz+\curl \phi\;\mbox{in} \;\O_f,\quad \mbox{and}\quad \|\bz\|_{2,\O_f} +\|\phi\|_{1,\O_f}\lesssim\|\btau\|_{  \H(\vdiv,\O_f)}.
\end{align*}
Finally, we can define a discrete function $\btau_h\in\bW_h$ such that
\[
\btau_h=I_h(\nabla \bz)+\curl \phi_h. 
\] 
Then
\begin{align*}
	\bdiv(\btau-\btau_h)=\bdiv(\nabla\bz-I_h(\nabla \bz))=(\mathbb{I}-\mathcal{P}_h)(\bdiv\btau),
\end{align*}
is $\bL^2(\Omega)$-orthogonal. 

\begin{theorem}
	Consider $(\kappa,(\bu,\bw),p)\in\mathbb{R}\times\bH\in\Q$ as a solution of the spectral problem (\ref{eq:coupled-problem-variational2}) and let $(\kappa_h,(\bu_h,\bw_h),p_h)\in\mathbb{R}\times\bH_h\in\Q_h$
	be the discrete solution of the spectral problem \cblue{(\ref{def:limit_system_eigen_complete_discrete})}.  Then, for every $h_0\ge h$ there holds:
	\begin{align*}
		\vertiii{((\bu-\bu_h,\bw-\bw_h),p-p_h)}_{\bH}\lesssim &\eta +\Theta+ \|(\rho_S(\kappa\bu-\kappa_h\bu_h))\|_{0,\Omega_S}+\|(\rho_S(\bu-\bu_h))\|_{0,\Omega_S}\\
		&\quad+\|(\rho_F(\kappa\bw-\kappa_h\bw_h))\|_{0,\Omega_f}+\|(\rho_F(\bw-\bw_h))\|_{0,\Omega_f},
	\end{align*}
	where the hidden constant is independent of the mesh size, $\nu$ and the discrete solutions. Moreover,
	the eigenvalue also satisfies the following reliability bound:
	\begin{align*}
		|\kappa-\kappa_h|& \lesssim \eta^2+\Theta^2+\|(\rho_S(\kappa\bu-\kappa_h\bu_h))\|_{0,\Omega_S}^2+\|(\rho_S(\bu-\bu_h))\|_{0,\Omega_S}^2\\
		&\quad+\|(\rho_F(\kappa\bw-\kappa_h\bw_h))\|_{0,\Omega_f}^2+\|(\rho_F(\bw-\bw_h))\|_{0,\Omega_f}^2.
	\end{align*}
\end{theorem}
\begin{proof}
	By applying an application of the stability result of Lemma \ref{rel11}, for  $((\bu-\bu_h,\bw-\bw_h),p-p_h)\in \bY\times \Q$, there exists $((\bv,\btau),q)\in \bY\times \Q$ with $\vertiii{((\bv,\btau),q)}_{\bH}\lesssim \vertiii{((\bu,\bw),p)}_{\bH}$ such that 
	\begin{align*}
		\vertiii{((\bu-\bu_h,\bw-\bw_h),p-p_h)}_{\bH}^2 \lesssim {A}(((\bu-\bu_h,\bw-\bw_h),p-p_h),((\bv,\btau),q)).
	\end{align*}
	Using the definition of the weak formulation (\ref{def:limit_system_eigen_complete}), it follows that
	\begin{align*}
		\vertiii{((\bu-\bu_h,\bw-\bw_h),p-p_h)}_{\bH}^2 &\lesssim(\kappa+1) B(((\bu,\bw),p),((\bv,\btau),q))-{A}(((\bu_h,\bw_h),p_h),((\bv,\btau),q)),\\
		&\lesssim (\kappa+1) B(((\bu,\bw),p),((\bv,\btau),q))-(\kappa_h+1) B(((\bu_h,\bw_h),p),((\bv,\btau),q))\\
		&\qquad+(\kappa_h+1) B(((\bu_h,\bw_h),p_h),((\bv,\btau),q))
		-{A}(((\bu_h,\bw_h),p_h),((\bv,\btau),q)).
	\end{align*}
	By the definition of the bilinear form $B(\cdot,\cdot)$ and using Cauchy--Schwarz's inequality, we obtain
	\begin{align*}
		|(\kappa+1) B(((\bu,\bw),p),((\bv,\btau),q))&-(\kappa_h+1) B(((\bu_h,\bw_h),p),((\bv,\btau),q))|\\
		&=|((\kappa+1) \rho_S\bu-(\kappa_h+1) \rho_S\bu_h,\bv)+((\kappa+1) \rho_F\bw-(\kappa_h+1) \rho_F\bw_h,\btau)|\\
		&\lesssim (\|(\rho_S(\kappa\bu-\kappa_h\bu_h))\|_{0,\Omega_S}+\|(\rho_S(\bu-\bu_h))\|_{0,\Omega_S})\|\bv\|_{0,\Omega_S}\\
		&\quad+(\|(\rho_F(\kappa\bw-\kappa_h\bw_h))\|_{0,\Omega_f}+\|(\rho_F(\bw-\bw_h))\|_{0,\Omega_f})\|\btau\|_{0,\Omega_f}.
	\end{align*}
	In order to estimate
	\[(\kappa_h+1) B(((\bu_h,\bw_h),p_h),((\bv,\btau),q))-{A}(((\bu_h,\bw_h),p_h),((\bv,\btau),q)),\]
	we first add 
	\begin{align*}
		{A}(((\bu_h,\bw_h),p_h),((\bv_h,\btau_h),q_h))-(\kappa_h+1) B(((\bu_h,\bw_h),p_h),((\bv_h,\btau_h),q_h))=0,
	\end{align*}
	and then apply integration by parts element-wise with the Helmholtz decomposition for $\btau-\btau_h$.  Then we obtain
	\begin{align*}
		(\kappa_h+1) B(((\bu_h,\bw_h),p),((\bv-\bv_h,\btau-\btau_h),q-q_h))-{A}(((\bu_h,\bw_h),p_h),&((\bv-\bv_h,\btau-\btau_h),q-q_h))\\
		&=I_S+I_F+I_I,
	\end{align*}
	where 
	\begin{align*}
		I_S &:= \sum_{K\in\mathcal{T}_h\cap \O_s}\int_{K}(\div(2\mu({\bx})\beps(\bu_h))-\nabla p_h+\kappa\bu_h)\cdot(\bv-\bv_h)  +\int_{\O_s}\left(\vdiv\bu_h +\frac{1}{\lambda}p_h\right) (q-q_h)	\\	
		&\quad -\sum_{E\in\mathcal{E}(\mathcal{T}_h)\cap\O_s}\int_E  ({(2\mu_h\beps(\bu_h) -p_h\mathbb{I})\bn})\cdot (\bv-\bv_h),\\
		I_F &:=   
		\int_{\O_f}(\nabla(c^2\rho_f\vdiv\bw_h)+\kappa_h \rho_F\bw_h)\cdot (\nabla\bz-I_h(\nabla \bz)+\int_{\O_f}(\kappa_h \rho_F\bw_h)\cdot (\curl \phi-\curl \phi_h)\\  
		&\quad -\sum_{E\in\mathcal{E}(\mathcal{T}_h)\cap\O_f}\int_E  ((c^2\rho_f\vdiv\bw_h)\bn)\cdot (\nabla\bz-I_h(\nabla \bz)- \int_{\Gamma_0}(g\rho_f\bw_h\cdot\bn)((\nabla\bz-I_h(\nabla \bz)\cdot\bn),
		\\
		I_I &:= \sum_{E\in\mathcal{E}(\mathcal{T}_h)\cap\Sigma}\int_E  ({(2\mu_h\beps(\bu_h) -p_h\mathbb{I})\bn})\cdot (\bv-\bv_h)-\sum_{E\in\mathcal{E}(\mathcal{T}_h)\cap\Sigma}\int_E  ((c^2\rho_f\vdiv\bw_h)\bn)\cdot (\nabla\bz-I_h(\nabla \bz)).		
	\end{align*}
	Applying integration by parts, along with the Cauchy--Schwarz inequality and the approximation results of the Cl\'ement interpolant, yields:
	\begin{align*}
		I_S\lesssim(\eta_S +\Theta_S)\vertiii{((\bv,\btau),q)}_{\bH}. 
	\end{align*}
	Using integration by parts and   the Cauchy--Schwarz inequality implies that
	\begin{equation*}
		|I_F|  \le (\eta_f+\Theta_f) \|\btau\|_{  \H(\vdiv,\O_f)},\qquad
		|I_I|\le \eta_I \vertiii{((\bv,\btau),q)}_{\bH},
	\end{equation*}
	which completes the proof. 
\end{proof}

\subsection{Efficiency}
We now focus on proving the efficiency bound for the a posteriori error estimator. The strategy for this is the standard based in the technique of using localization with bubble functions (see \cite{MR1885308,MR3059294}).

\begin{lemma}[Interior bubble functions]
	For any $K\in \CT_h$, let $\psi_{K}$ be the corresponding interior bubble function.
	Then, there holds
	\begin{align*}
		\|q\|_{0,K}^2&\lesssim \int_{K}\psi_{K} q^2\leq \|q\|_{0,K}^2\qquad \forall q\in \mathbb{P}_k(K),\\
		\| q\|_{0,K}&\lesssim \|\psi_{K} q\|_{0,K}+h_K\|\nabla(\psi_{K} q)\|_{0,K}\lesssim \|q\|_{0,K}\qquad \forall q\in \mathbb{P}_k(K),
	\end{align*}
	where the hidden constants are 
	independent of  $h_K$.
\end{lemma}
\begin{lemma}[Facet bubble functions]
	\label{burbuja}
	For any $K\in \CT_h$ and $\ell\in\CE_{K}$, let $\psi_{\ell}$
	be the corresponding facet bubble function. Then, there holds
	\begin{equation*}
		\|q\|_{0,\ell}^2\lesssim \int_{\ell}\psi_{\ell} q^2 \leq \|q\|_{0,\ell}^2\qquad
		\forall q\in \mathbb{P}_k(\ell).
	\end{equation*}
	Moreover, for all $q\in\mathbb{P}_k(\ell)$, there exists an extension of  $q\in\mathbb{P}_k(K)$ (again denoted by $q$) such that
	\begin{align*}
		h_K^{-1/2}\|\psi_{\ell} q\|_{0,K}+h_K^{1/2}\|\nabla(\psi_{\ell} q)\|_{0,K}&\lesssim \|q\|_{0,\ell},
	\end{align*}
	where the hidden constants are independent of  $h_K$.
\end{lemma}

The efficiency of the elastic estimator $\eta_S$ directly follows from \cite[Lemma 4.3]{khan2023finite}. 
On the other hand, the efficiency of the fluid estimator $\eta_F$ follows by collecting the bounds in the following result.
\begin{lemma}\label{eff11}
	There holds:
	\begin{align*}
		h_K^2\|\rho^F\textbf{R}_{1,T}^F\|_{0,T}^2&\le  \|c\rho_F^{1/2}\vdiv(\bw-\bw_h)\|_{0,T}^2+\|(\rho^F)^{-1/2}(\omega^2_h\rho_F\bw_h -\omega^2\rho_F\bw)\|_{0,T}^2,\\
		h_K^2\|\rho^FR_{2,T}^S\|_{0,T}^2&\le\|(\rho^F)^{-1/2}(\omega^2_h\rho_F\bw_h -\omega^2\rho_F\bw)\|_{0,T}^2,\\
		h_E\|\rho_E^FJ_{1,\ell}^F \|_{0,\ell}^2&\le   \|c\rho_F^{1/2}\vdiv(\bw-\bw_h)\|_{0,\omega_T}^2+\|(\rho^F)^{-1/2}(\omega^2_h\rho_F\bw_h -\omega^2\rho_F\bw)\|_{0,\omega_T}^2,\\
		h_E\|\rho_E^FJ_{2,\ell}^F \|_{0,\ell}^2&\le \|(\rho^F)^{-1/2}(\omega^2_h\rho_F\bw_h -\omega^2\rho_F\bw)\|_{0,\omega_T}^2.
	\end{align*}
\end{lemma}
\begin{proof}
	Firstly, we define $\bv_T= \chi_Th_T^2(\rho^F)^2\textbf{R}_{1,T}^F$, where  $\chi_T$ is the element bubble function defined on $T$. Then
	\begin{align*}
		h_T^2\|\rho^F\textbf{R}_{1,T}^F\|_{0,T}^2\le (\textbf{R}_{1,T}^F, \bv_T)&=\int_T  (c^2\rho_F\nabla(\vdiv\bw_h)+\omega^2_h\rho_F\bw_h) \bv_T\\
		&=\int_T  (c^2\rho_F\nabla(\vdiv\bw_h)+\omega^2_h\rho_F\bw_h -c^2\rho_F\nabla(\vdiv\bw)+\omega^2\rho_F\bw) \bv_T\\
		&=\int_T  (-c^2\rho_F\nabla(\vdiv(\bw-\bw_h))\bv_T +\int_T(\omega^2_h\rho_F\bw_h -\omega^2\rho_F\bw)) \bv_T=T_1+T_2.
	\end{align*} 
	An application of integration of parts on $T_1$ gives that
	\begin{align*}
		T_1&= \int_T  (-c^2\rho_F\nabla(\vdiv(\bw-\bw_h))\bv_T\\
		&= \int_T  (c^2\rho_F\vdiv(\bw-\bw_h))\vdiv\bv_T \\
		&\le \|c\rho_F^{1/2}\vdiv(\bw-\bw_h)\|_{0,T}h_T\|\rho^F\textbf{R}_{1,T}^F\|_{0,T}.
	\end{align*}
	Finally, in order to estimate $T_2$ we simply apply the Cauchy--Schwarz inequality to obtain  
	\begin{align*}
		T_2= \int_T(\omega^2_h\rho_F\bw_h -\omega^2\rho_F\bw)) \bv_T\le \|(\rho^F)^{-1/2}(\omega^2_h\rho_F\bw_h -\omega^2\rho_F\bw)\|_{0,T}h_T\|\rho^F\textbf{R}_{1,T}^F\|_{0,T}.
	\end{align*}
	By applying the inverse inequality, we derive the second stated result. Next, we define $$\bv_\ell:=\psi_\ell h_E(\rho_E^F)^2[\![c^2\rho_F(\vdiv\bw_h)\bn]\!],$$ where $\psi_\ell$ is the bubble function that satisfies the properties outlined in Lemma \ref{burbuja}. We then proceed to estimate the term $h_E\|\rho_E^{F}J_{\ell}\|_{0,\ell}^2$, leading to  
	\begin{align}\label{burbuja11}
		h_E\|\rho_EJ_{1,\ell}\|_{0,\ell}^2&\lesssim ([\![{c^2\rho_F(\vdiv\bw_h)\bn}]\!], \bv_\ell)_\ell= ([\![{c^2\rho_F(\vdiv\bw_h)\bn}]\!], \bv_\ell)_\ell.
	\end{align}
	Applying integration by parts on $\omega_T$ yields
	\begin{align*}
		([\![{c^2\rho_F(\vdiv\bw_h)\bn}]\!], \bv_\ell)_\ell=\sum_{T\in\omega_T} &\Big((c^2\rho_F(\vdiv\bw_h)-c^2\rho_F(\vdiv\bw), \vdiv(\bv_\ell))_T\\
		&+(\textbf{R}_{1,T}^F, \bv_\ell)_T
		+(\omega^2_h\rho_F\bw_h -\omega^2\rho_F\bw, \bv_\ell)_T\Big).
	\end{align*}
	Using the Cauchy--Schwarz inequality along with Lemma \ref{burbuja} and combining it with (\ref{burbuja11}), we obtain that
	\begin{align*}
		h_E\|\rho_E^FJ_{1,\ell}\|_{0,\ell}^2&\lesssim ( \|c\rho_F^{1/2}\vdiv(\bw-\bw_h)\|_{0,\omega_T}+\|(\rho^F)^{-1/2}(\omega^2_h\rho_F\bw_h -\omega^2\rho_F\bw)\|_{0,\omega_T}) h_E^{1/2}\|\rho_E^FJ_{1,\ell}\|_{0,\ell}.
	\end{align*}
	The last estimates directly follows from \cite[Theorem 2]{duran1999posteriori}.
\end{proof}

Lastly, the efficiency of the interface estimator $\eta_I$ is a direct consequence of \cite[Lemma 4.3]{khan2023finite} in conjunction with Lemma \ref{eff11}.

\section{Numerical experiments}\label{sec:numerical-section}
This section presents a series of computational tests using the open-source FE library \texttt{FEniCS} \cite{AlnaesEtal2015} together with the special modules \texttt{FeniCS$_{ii}$} \cite{kuchta2020assembly} and \texttt{multiphenics} \cite{ballarin2019multiphenics} for the treatment of bulk-surface coupling mechanisms. The meshes have been constructed using the mesh generation/manipulation library \texttt{GMSH}  \cite{geuzaine2009gmsh}.

\begin{figure}[t!]\centering
	\begin{minipage}{0.49\linewidth}\centering
		{\footnotesize $\Omega_1$}\\
		\includegraphics[scale=0.18,trim=45cm 2cm 42cm 4cm,clip]{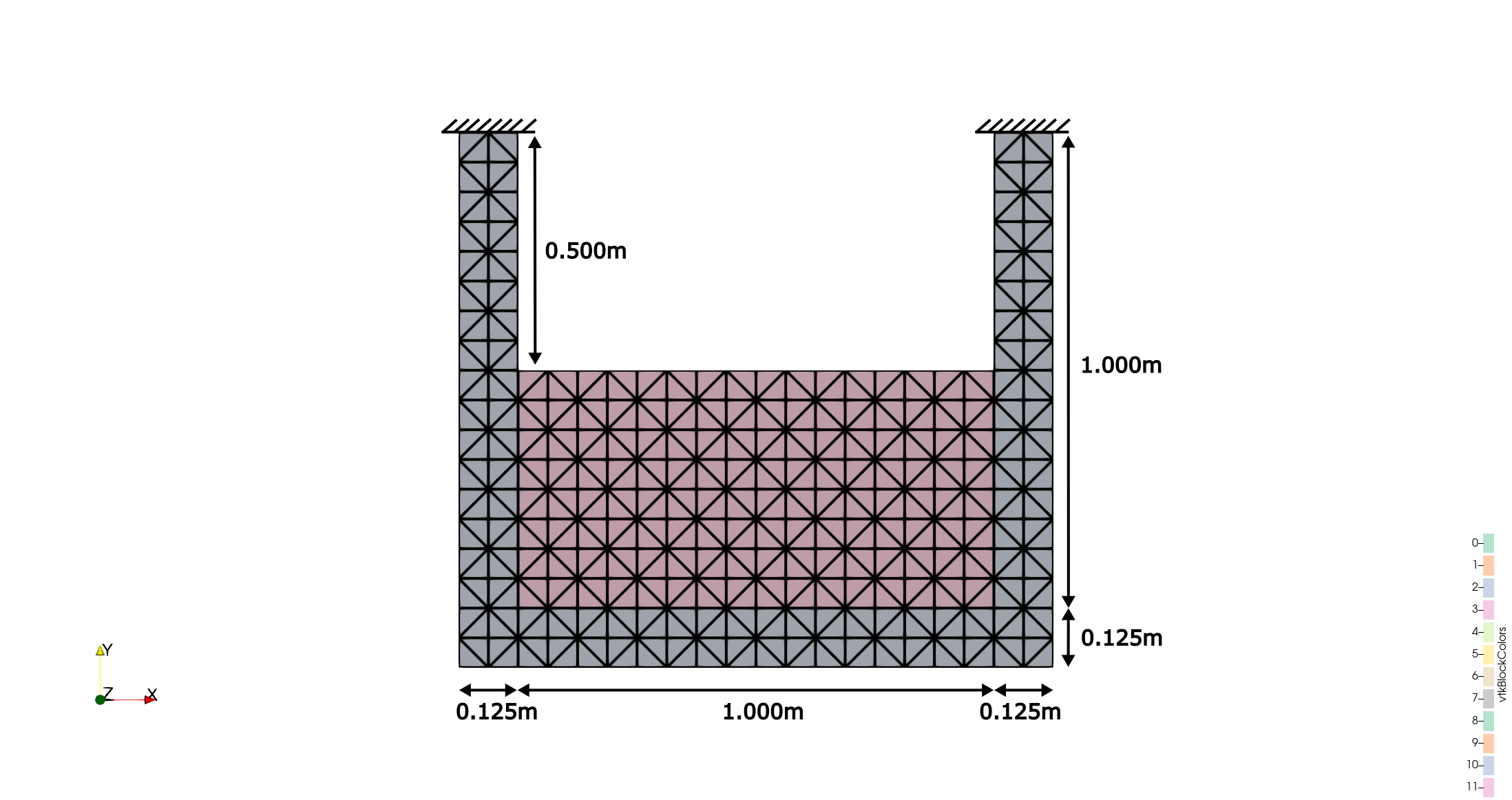}
	\end{minipage}
	\begin{minipage}{0.49\linewidth}\centering
		{\footnotesize $\Omega_2$}\\
		\includegraphics[scale=0.125]{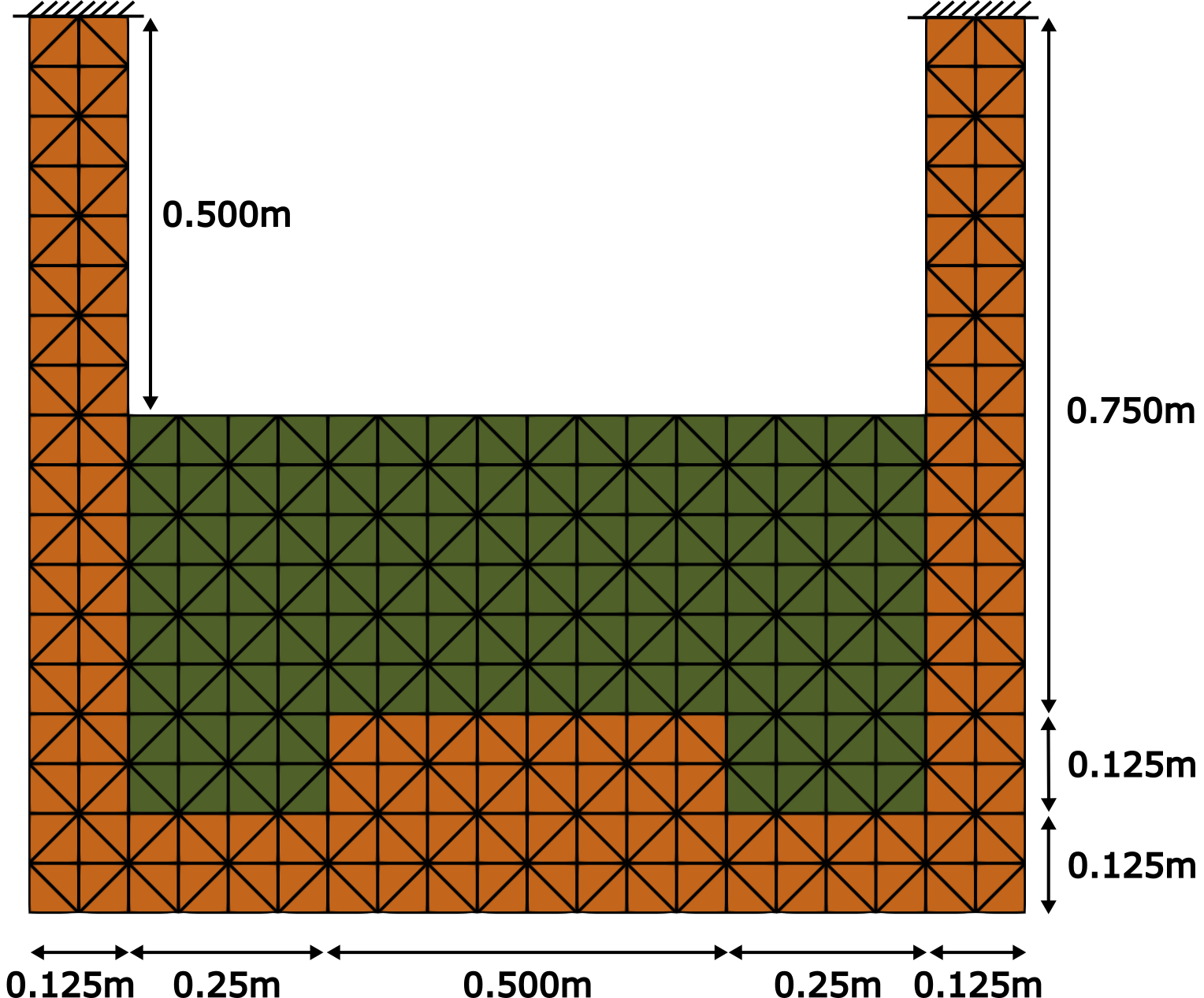}
	\end{minipage}
	\caption{Test \ref{subsec:accuracy-test}. Schematic of 2D computational domains indicating dimensions, boundaries, subdomains, and using a coarse mesh with $N=2$.}
	\label{fig:2D-meshed-domain}
\end{figure}

As we typically do not know the closed-form solutions for the eigenvalues, their convergence rates have been  obtained with a standard least-squares fitting and highly refined meshes.  
We  denote by $N$ the mesh refinement level (number of edges along the shortest edge) and $\texttt{dof}$ denotes the number of degrees of freedom. We denote by $\omega_{h,i}$ the i-th discrete eigenfrequency and denote the error on the $i$-th eigenvalue by $\err( \omega_i)$ with 
\begin{equation*}
	\err(\omega_i):=\vert \omega_{h,i}^2-\omega_{i}^2\vert.
\end{equation*}		

\subsection{Accuracy test}\label{subsec:accuracy-test}
In this section we perform a convergence analysis with respect to mesh refinement using the first three eigenvalues of the coupled problem. With the aim of comparing the numerically obtained results against the simulations reported in \cite{MR3283363}, the first series of tests employ the following parameter values
$\rho_s= 7700$kg$/$m$^3$, 
$E= 1.44\times 10^{11}$Pa, 
$\nu = 0.35$, 
$\rho_f= 1000$kg$/$m$^3$, 
$c = 1430$m$/$s, 
$g = 9.8$m$/$s$^2$, 
and we focus on the elasto-acoustic modes. We consider two configurations for the geometry. First, we consider $\Omega_1$ such that the fluid domain is a rectangle and the solid container is a polygon as depicted in Figure \ref{fig:2D-meshed-domain}(left). The second geometry, denoted as $\Omega_2$, is such that we have re-entrant corners in the solid and fluid subdomains \ref{fig:2D-meshed-domain}(right).  The  domains are discretized such that there is conformity between $\Omega_S$ and $\Omega_f$. 

The error history for is outlined in he two blocks of Table~\ref{tabla:results-accuracy-elastoacoustic}  for the elasto-acoustic modes in $\Omega_1$ and $\Omega_2$, respectively. We observe an asymptotic linear convergence  for the case of Taylor--Hood elements, and suboptimal rate for MINI-elements. However, results on Taylor--Hood on the first part of Table \ref{tabla:results-accuracy-elastoacoustic} are closer to the reference values from \cite{MR3283363}. Suboptimal convergence results are expected in $\Omega_i$ because of the strong interaction between the solid and the fluid, where the solid contains at least two angle singularities and four points where boundary conditions change from Dirichlet to Neumann.

\begin{table}[t!]\centering
	{\footnotesize
		\setlength{\tabcolsep}{9.5pt}
			\caption{Test \ref{subsec:accuracy-test}. Lowest computed  elasto-acoustic eigenvalues on two different geometry configurations $\Omega_i$.}
			\label{tabla:results-accuracy-elastoacoustic}
			\begin{tabular}{@{}m{0.75cm}@{}c@{}} 
				\rotatebox{90}{Configuration $\Omega_1$} &
				\begin{tabular}{|r r r r |c| r| r|}
					\hline
					\hline
					$N=8$ & $N=10$ & $N=12$ & $N=14$ & Order & $\omega_{extr}$ &\cite{MR3283363} \\ 
					\hline
					\multicolumn{7}{|c|}{MINI-element + $\mathbb{BDM}_1$}  \\
					\hline
					452.4554  & 449.4600  & 447.7540  & 446.6792  & 1.83 & 443.3200 & 442.71   \\
					1495.3899 & 1487.1441 & 1482.4956 & 1479.5959 & 1.88 & 1470.8135 & 1469.4  \\
					2634.2701 & 2618.4326 & 2609.1272 & 2603.1127 & 1.68 & 2582.4208 & 2578.33  \\
					2813.9435 & 2796.5619 & 2786.7383 & 2780.5963 & 1.87 & 2761.8755 & 2758.94  \\
					\hline
					\multicolumn{7}{|c|}{Taylor--Hood + $\mathbb{BDM}_1$}  \\
					\hline
					443.7421 & 443.5416 & 443.4114 & 443.3204 & 1.18 & 442.8549 & 442.71  \\
					1471.5727 & 1471.1864 & 1470.9341 & 1470.7567 & 1.15 & 1469.8245 & 1469.4   \\
					2586.2547 & 2584.7255 & 2583.7213 & 2583.0128 & 1.12 & 2579.1717 & 2578.33   \\
					2763.4399 & 2762.6174 & 2762.0797 & 2761.7013 & 1.14 & 2759.6912 & 2758.94  \\
					\hline
					\hline
			\end{tabular}  \end{tabular}
			
			\medskip 
			
			\begin{tabular}{@{}m{0.75cm}@{}c@{}}\centering 
				\rotatebox{90}{Configuration $\Omega_2$} &
				\begin{tabular}{|r r r r |c| r|}
					\hline
					\hline
					$N=10$             &  $N=12$         &   $N=14$         & $N=16$ & Order & $\omega_{extr}$ \\ 
					\hline
					\multicolumn{6}{|c|}{MINI-element + $\mathbb{BDM}_1$}  \\
					\hline
					405.0943  &   403.5271  &   402.5353  &   401.8617  & 1.76 &   399.2754  \\
					1597.7890  &  1592.0717  &  1588.3495  &  1585.7601  & 1.59 &  1574.6229  \\
					2634.5808  &  2625.2576  &  2619.2071  &  2615.0098  & 1.61 &  2597.2107  \\
					2654.9989  &  2645.1907  &  2638.9963  &  2634.7990  & 1.78 &  2618.9133  \\
					\hline
					\multicolumn{6}{|c|}{Taylor--Hood + $\mathbb{BDM}_1$}  \\
					\hline
					399.6299  &   399.4973  &   399.4046  &   399.3362  & 1.48 &   399.0376 \\
					1576.9642  &  1576.2474  &  1575.7438  &  1575.3713  & 1.14 &  1573.0690  \\
					2600.9286  &  2599.7732  &  2598.9654  &  2598.3703  & 1.17 &  2594.7947  \\
					2620.8404  &  2620.0780  &  2619.5463  &  2619.1554  & 1.18 &  2616.8270  \\
					\hline
					\hline
				\end{tabular}
			\end{tabular}

	}
\end{table} 		

\begin{figure}[t!]
	\centering
	\begin{minipage}{0.24\linewidth}
		\centering
		{\footnotesize $\bu_{h,1}$}\\
		\includegraphics[scale=0.1,trim=56cm 5cm 56cm 5cm,clip]{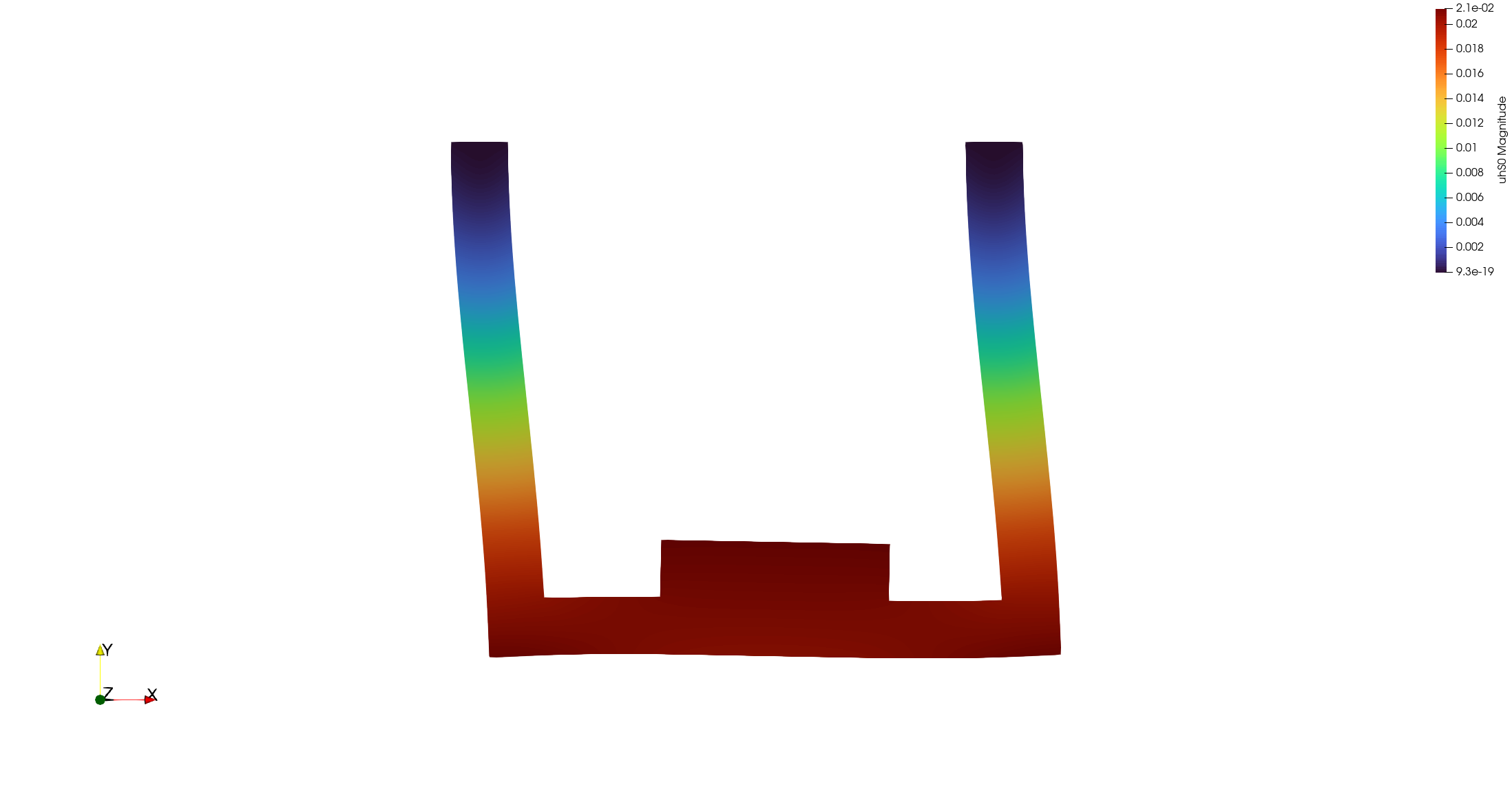}
	\end{minipage}
	\begin{minipage}{0.24\linewidth}
		\centering
		{\footnotesize $\bu_{h,2}$}\\
		\includegraphics[scale=0.1,trim=56cm 5cm 56cm 5cm,clip]{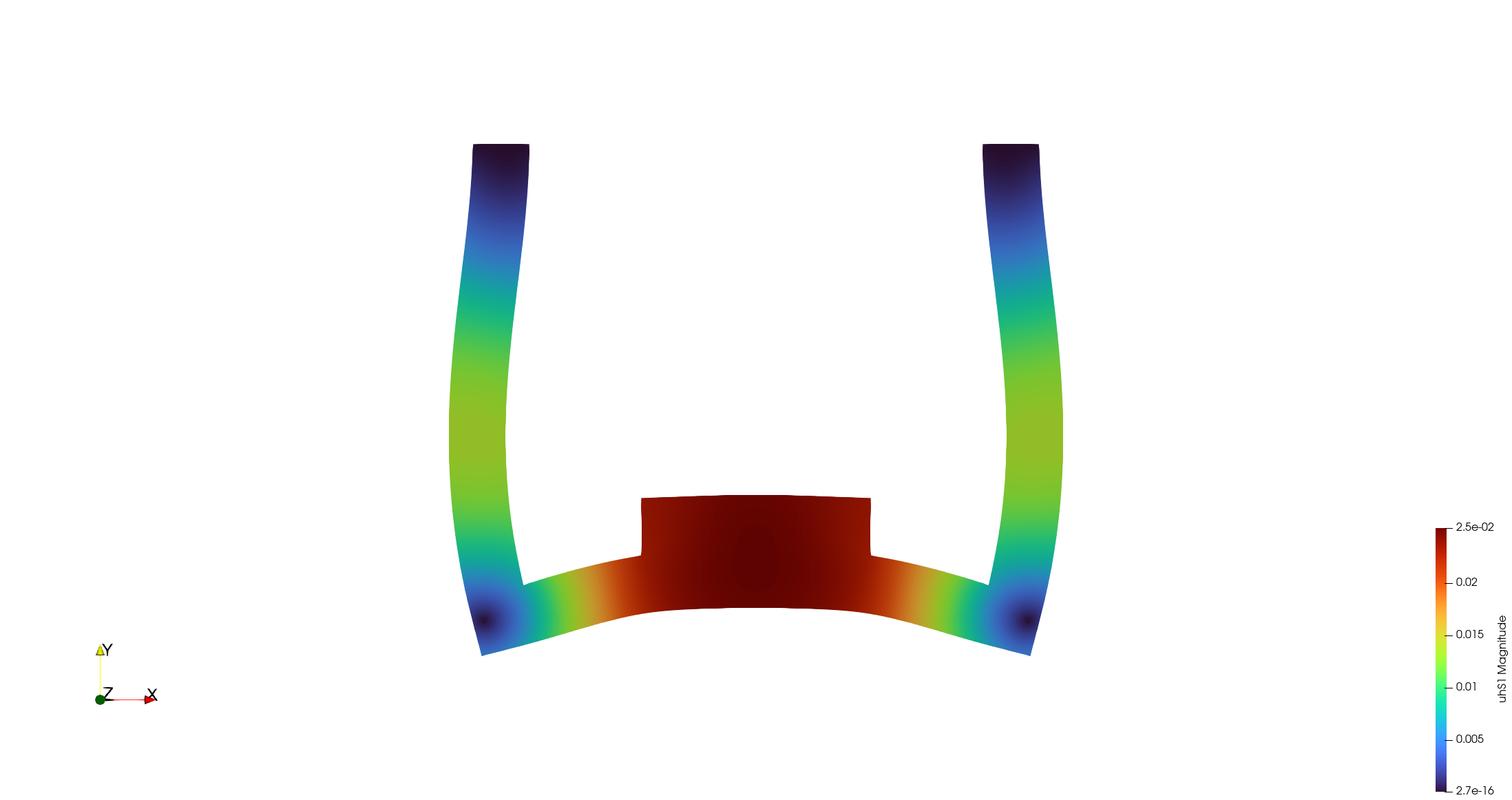}
	\end{minipage}
	\begin{minipage}{0.24\linewidth}
		\centering
		{\footnotesize $\bu_{h,3}$}\\
		\includegraphics[scale=0.1,trim=56cm 5cm 56cm 5cm,clip]{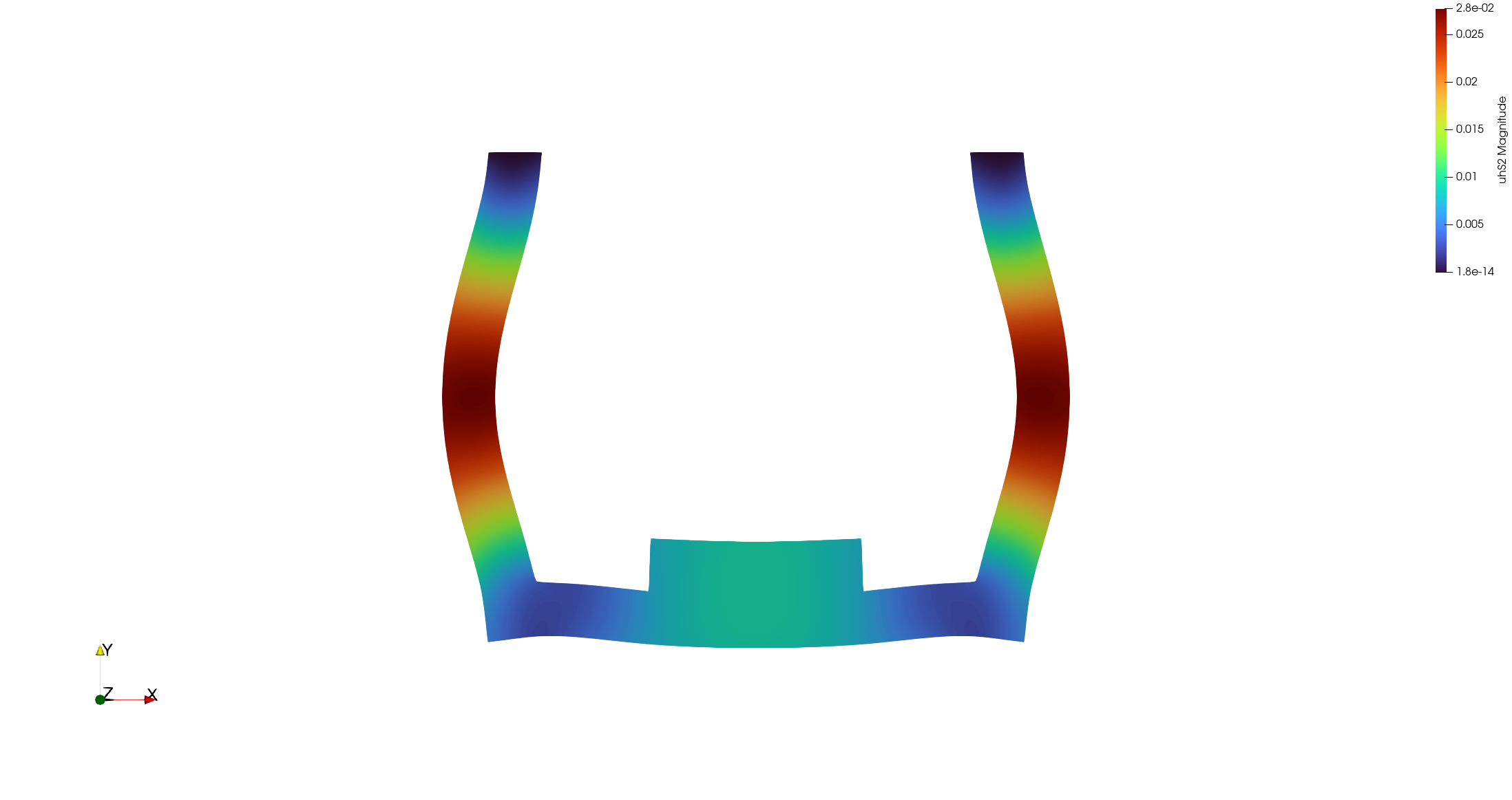}
	\end{minipage}
	\begin{minipage}{0.24\linewidth}
		\centering
		{\footnotesize $\bu_{h,4}$}\\
		\includegraphics[scale=0.1,trim=56cm 5cm 56cm 5cm,clip]{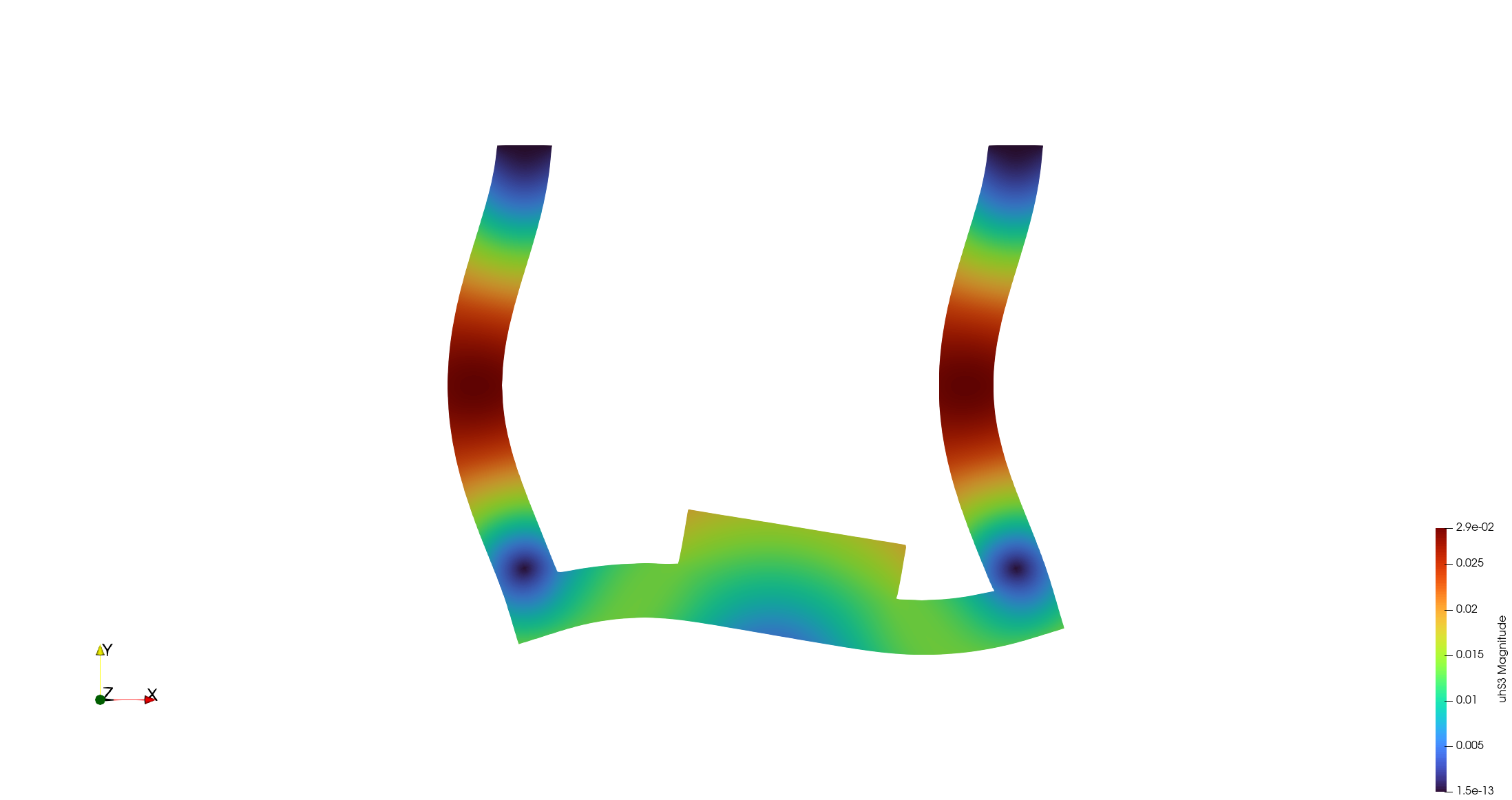}
	\end{minipage}
	\begin{minipage}{0.24\linewidth}
		\centering
		{\footnotesize $p_{h,1}$}\\
		\includegraphics[scale=0.1,trim=56cm 5cm 56cm 5cm,clip]{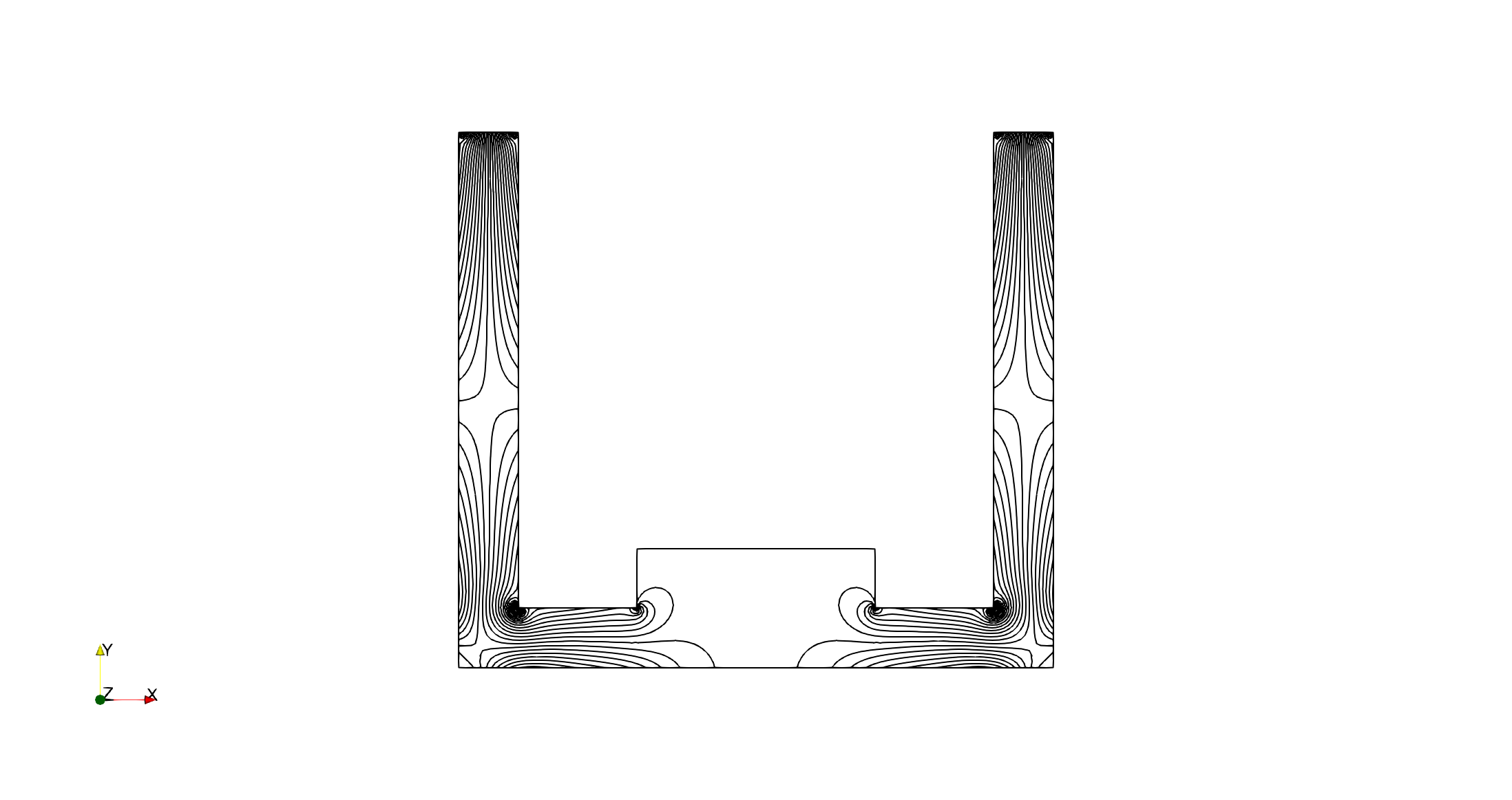}
	\end{minipage}
	\begin{minipage}{0.24\linewidth}
		\centering
		{\footnotesize $p_{h,2}$}\\
		\includegraphics[scale=0.1,trim=55cm 5cm 55cm 5cm,clip]{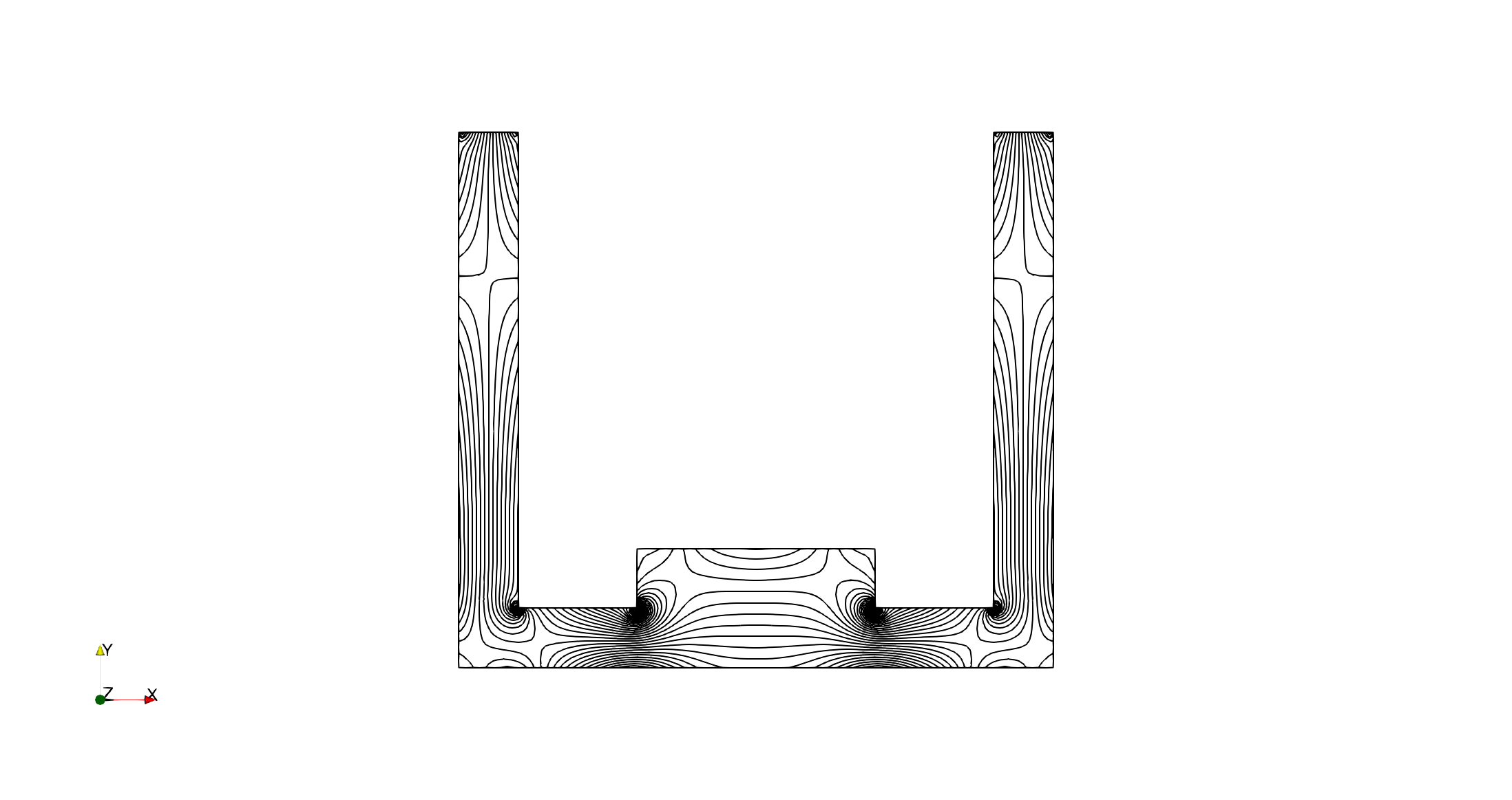}
	\end{minipage}
	\begin{minipage}{0.24\linewidth}
		\centering
		{\footnotesize $p_{h,3}$}\\
		\includegraphics[scale=0.1,trim=55cm 5cm 55cm 5cm,clip]{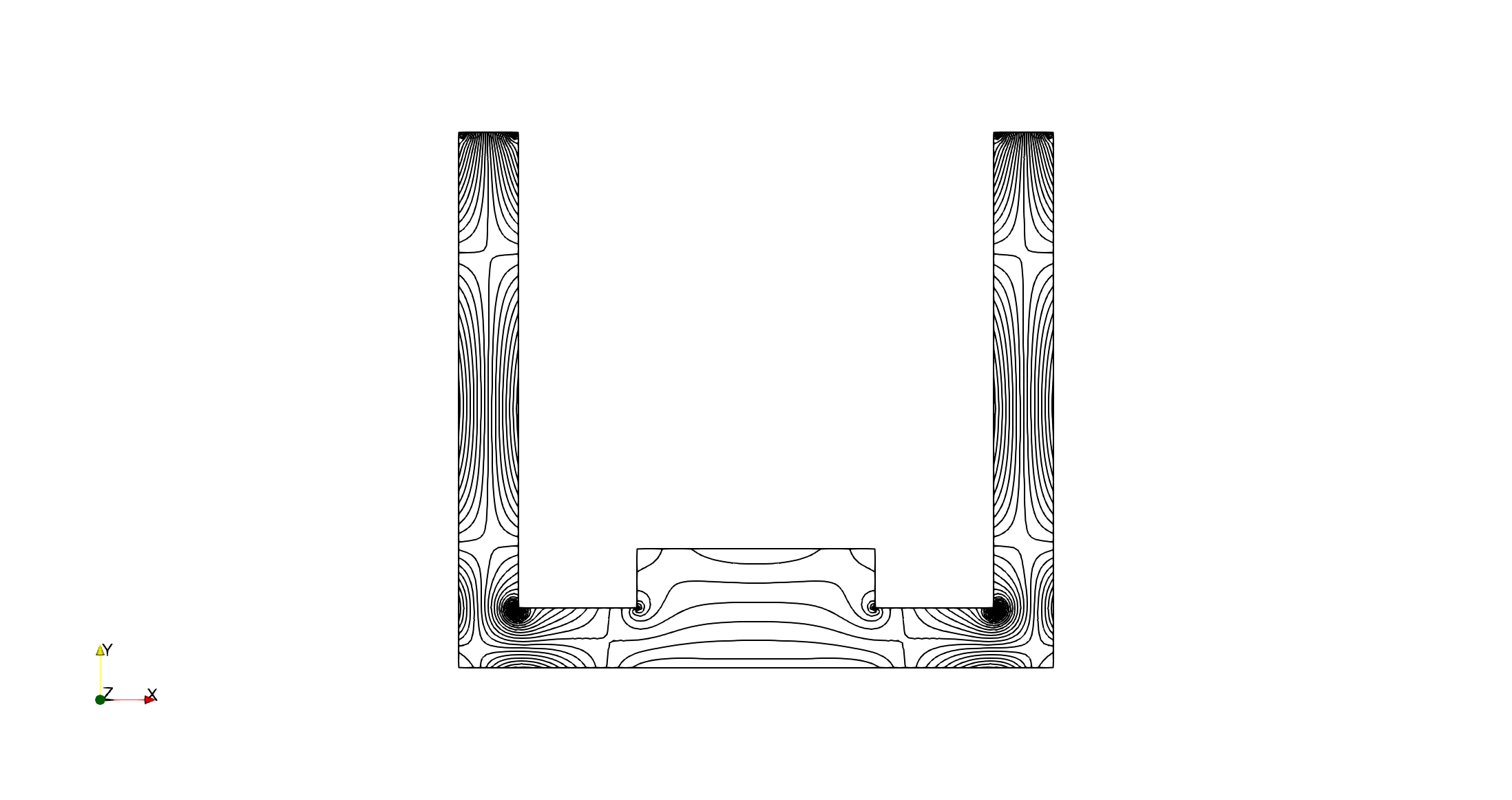}
	\end{minipage}
	\begin{minipage}{0.24\linewidth}
		\centering
		{\footnotesize $p_{h,4}$}\\
		\includegraphics[scale=0.1,trim=55cm 5cm 55cm 5cm,clip]{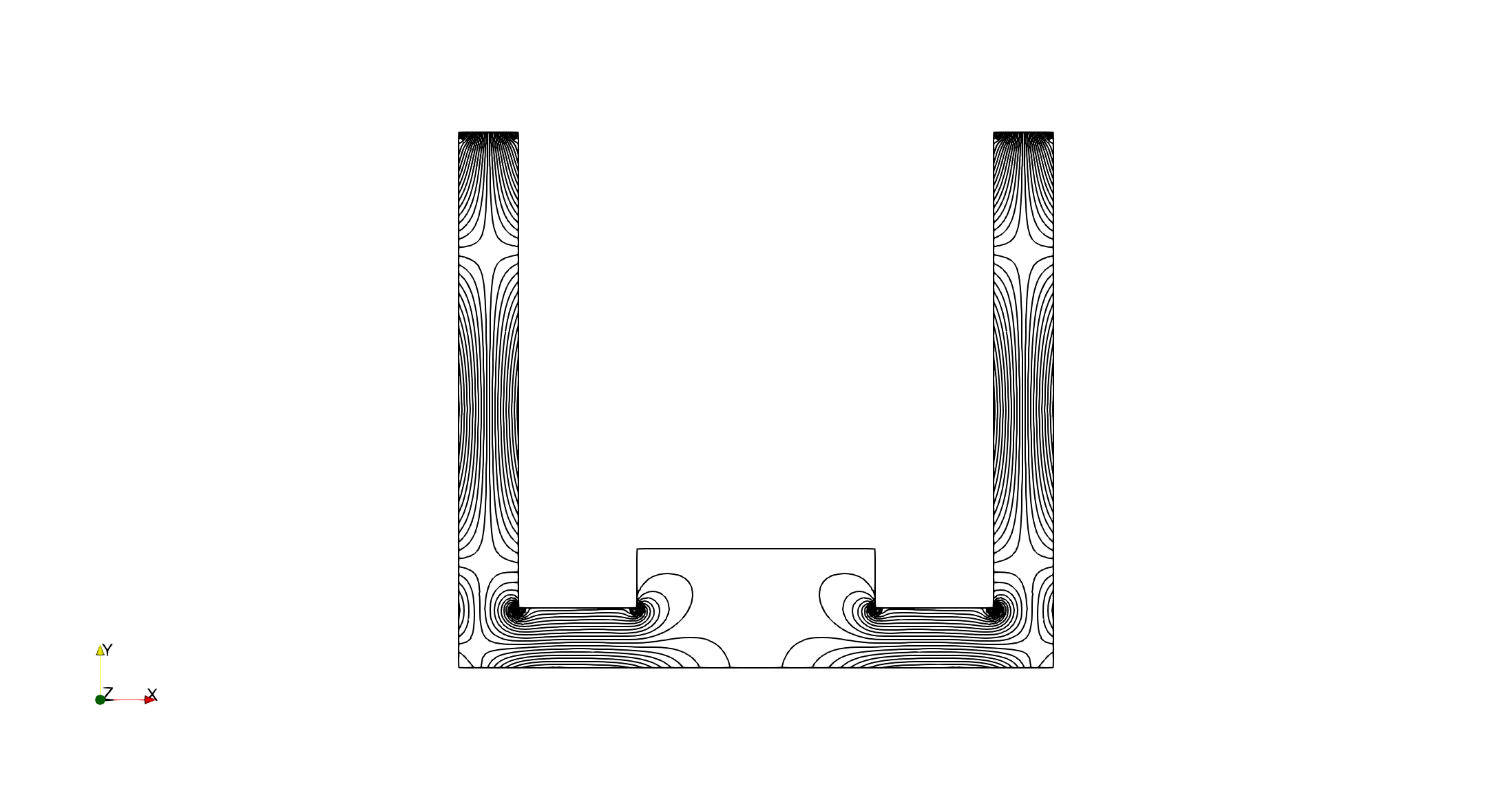}
	\end{minipage}
	\begin{minipage}{0.24\linewidth}
		\centering
		{\footnotesize $\bw_{h,1}$}\\
		\includegraphics[scale=0.1,trim=55cm 10cm 55cm 10cm,clip]{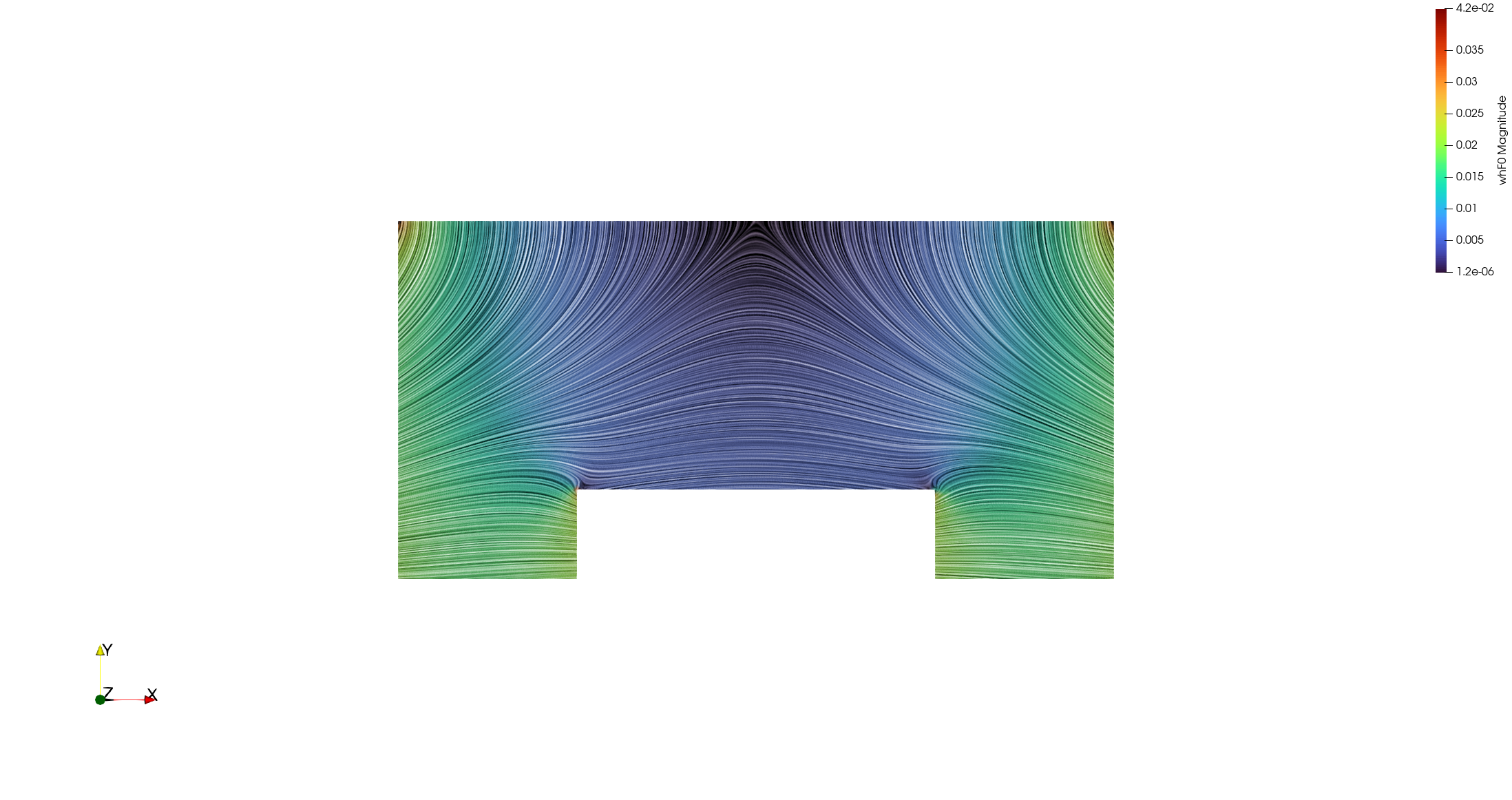}
	\end{minipage}
	\begin{minipage}{0.24\linewidth}
		\centering
		{\footnotesize $\bw_{h,2}$}\\
		\includegraphics[scale=0.1,trim=55cm 10cm 55cm 10cm,clip]{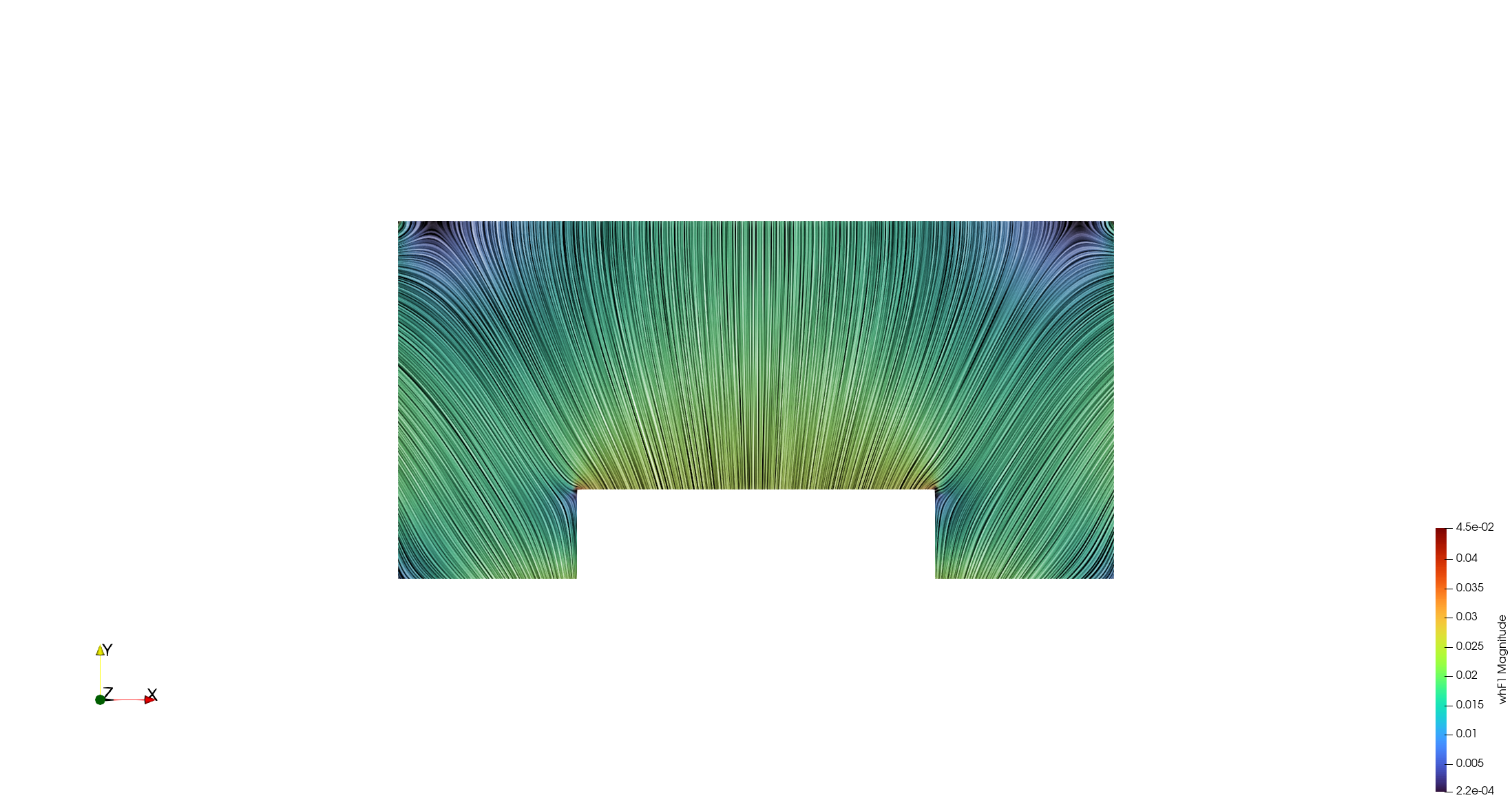}
	\end{minipage}
	\begin{minipage}{0.24\linewidth}
		\centering
		{\footnotesize $\bw_{h,3}$}\\
		\includegraphics[scale=0.1,trim=55cm 10cm 55cm 10cm,clip]{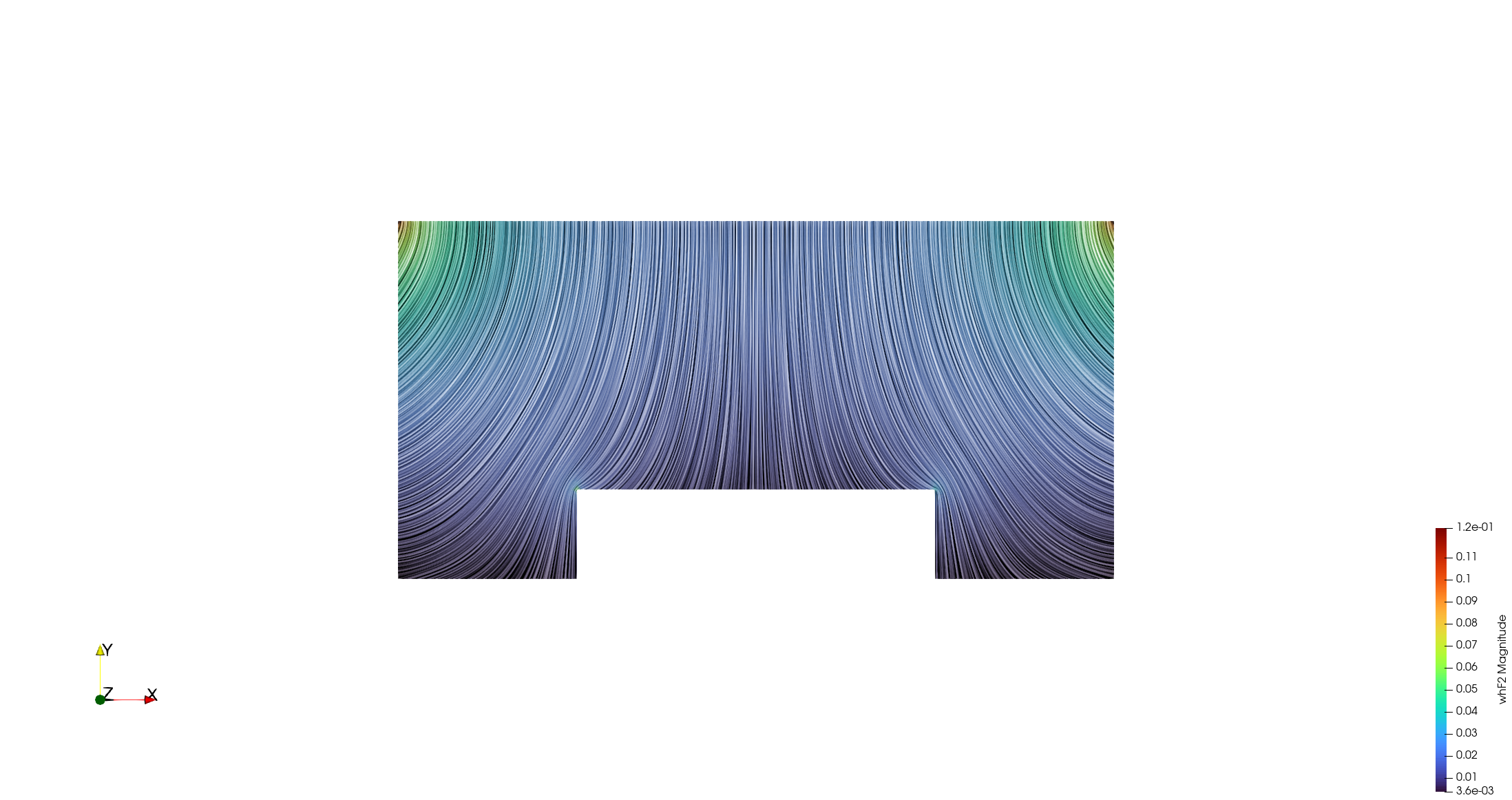}
	\end{minipage}
	\begin{minipage}{0.24\linewidth}
		\centering
		{\footnotesize $\bw_{h,4}$}\\
		\includegraphics[scale=0.1,trim=55cm 10cm 55cm 10cm,clip]{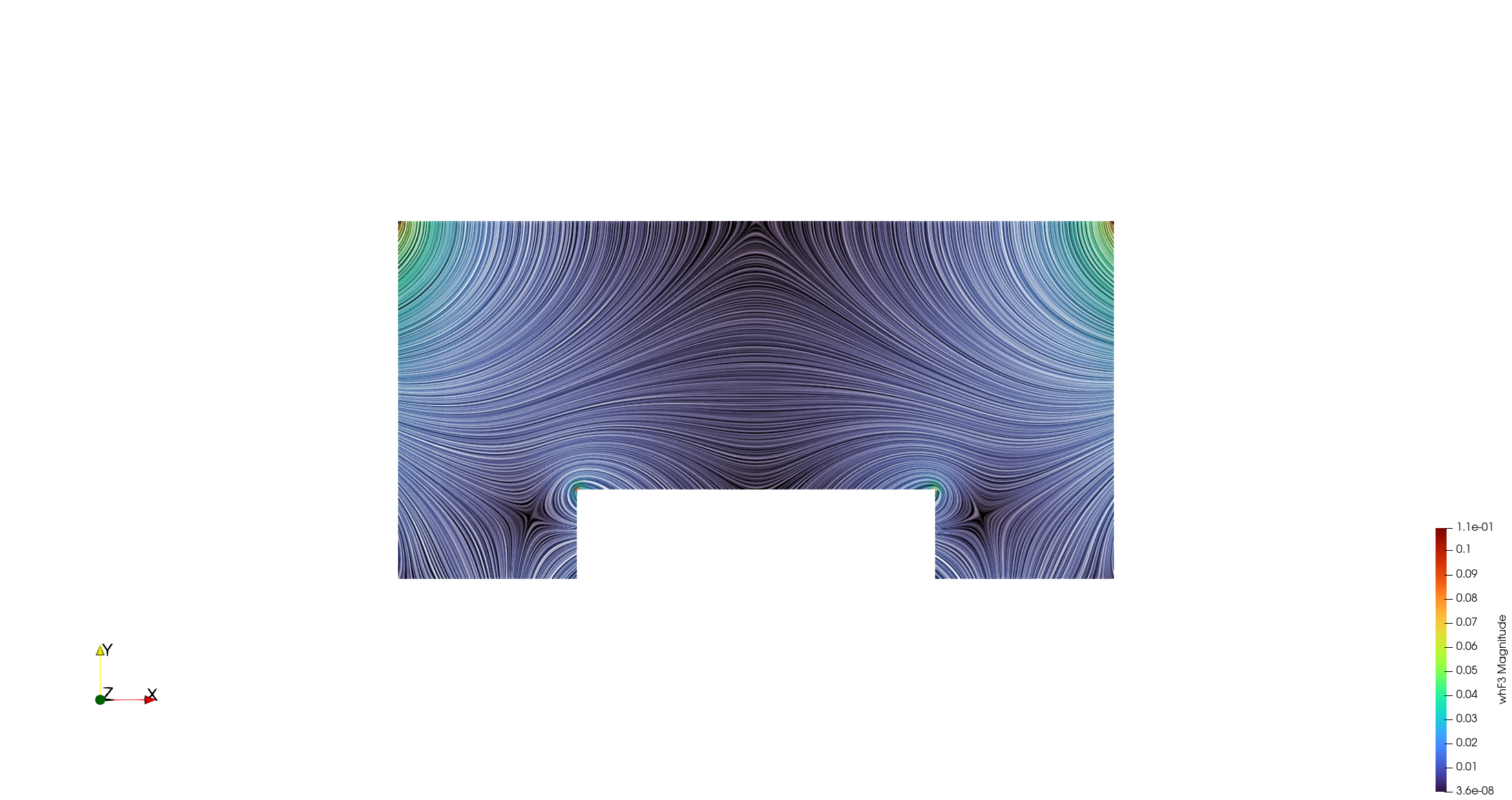}
	\end{minipage}
	\caption{Example \ref{subsec:accuracy-test}. Comparison between the first fourth lowest order computed elasto-acoustic modes on $\Omega_2$. The solid domain have been warped by a sufficiently large factor in order to observe the deformation.}
\end{figure}
\subsection{Fluid-structure interaction in a half-filled barrel}\label{subsec:3Dbarrel}
In this experiment we study the behavior of the scheme when considering a non-polygonal domain in three dimensions.  The solid and fluid sub-domains (see a sketch in Figure \ref{fig:barrel-domain}) are defined as follows 
\begin{gather*}
	\O_{s_1}:=\{(x,y,z)\in\mathbb{R}^3\,:\, y^2+z^2 = 1, x\in[-2,1] \},\\ 
	\O_{s_2}:=\{(x,y,z)\in\mathbb{R}^3\,:\, y^2+z^2 = 0.75^2, x\in[-1.75,0.75] \},\\
	\O_s:=\O_{s_1}\backslash\O_{s_2}, \quad 
	\Omega_f:=\{(x,y,z)\in\mathbb{R}^3\,:\, \sqrt{y^2+z^2} = -\sqrt{0.75}, x\in[-1.75,0.75] \}.
\end{gather*}
The barrel is clamped on the external circular faces and free of stress in the rest.	In order to compute the references eigenfrequencies, we have used a highly refined mesh and the higher order family $\mathbb{P}_3+\mathbb{P}_2 + \mathbb{BDM}_2$  for the solid displacement, solid pressure and fluid displacement, respectively.  The mesh level is defined as $h\approx 1{/} N$. The error history is reported in Table~\ref{tabla:results-3dbarrel}. Approximate solutions are shown in Figure~\ref{fig:3D-sols}. A linear rate of convergence is observed, justified by  the fluid-solid interaction.

\begin{figure}[!t]
	\centering
	\includegraphics[scale=0.1,trim=25cm 8cm 22cm 12cm, clip]{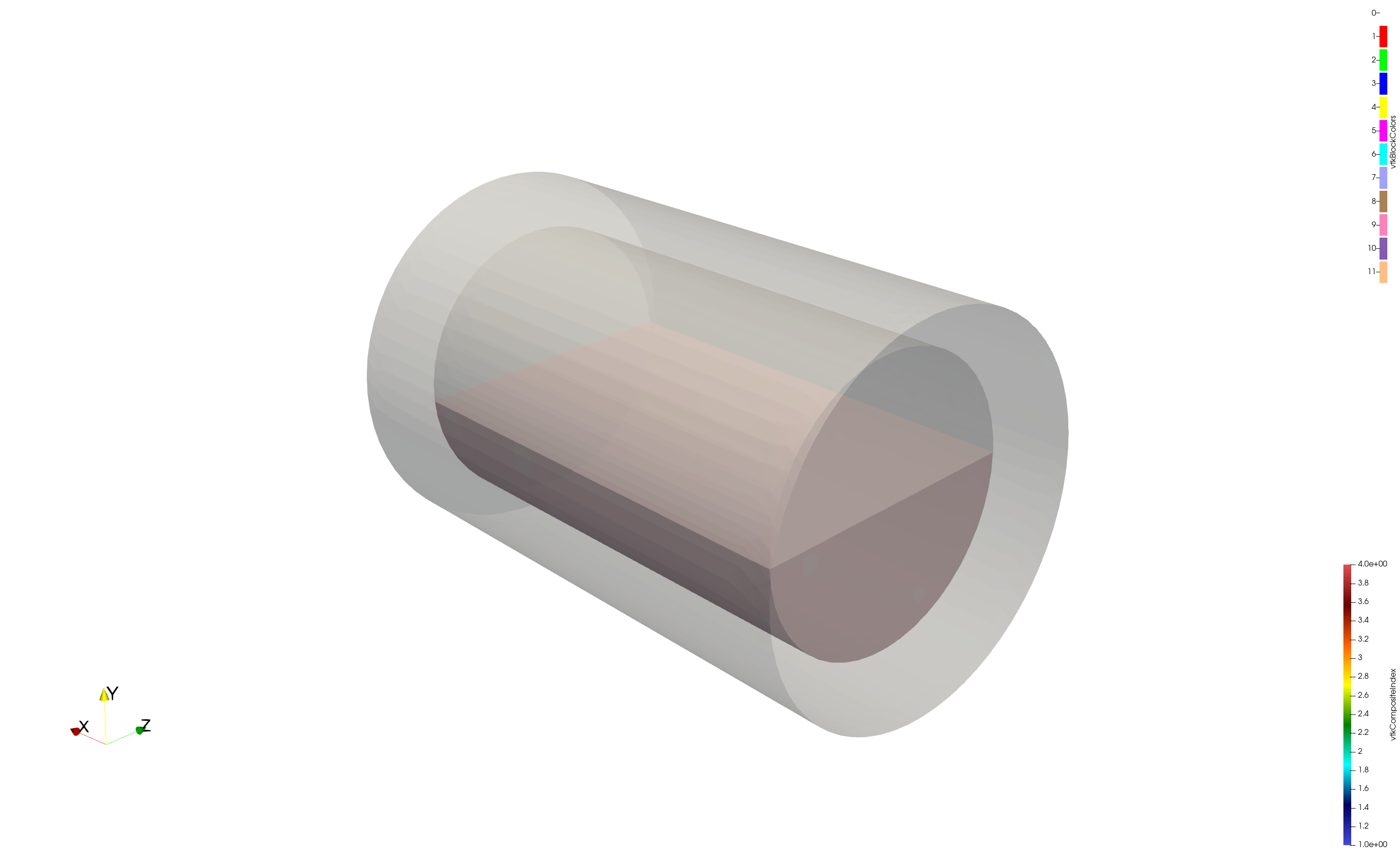}
	\caption{Example \ref{subsec:3Dbarrel}. The computational domain of the half filled barrel.}
	\label{fig:barrel-domain}
\end{figure}

\begin{table}[t!]\centering
	{\footnotesize\setlength{\tabcolsep}{9.5pt}
		\caption{Example \ref{subsec:3Dbarrel}. Lowest computed eigenvalues using different combinations of FE families in the half filled barrel domain.}
		\label{tabla:results-3dbarrel}
		\begin{tabular}{|r r r r |c| r| }
			\hline
			\hline
			$N=8$             &  $N=10$         &   $N=12$         & $N=14$ & Order & $\omega_{extr}$  \\ 
			\hline
			\multicolumn{6}{|c|}{MINI-element + $\mathbb{BDM}_1$}  \\
			\hline
			752.0366  &   750.1402  &   735.6936  &   731.3101  & 1.11 &   712.2499  \\
			837.3585  &   833.0175  &   807.9182  &   799.7536  & 1.17 &   768.1502  \\
			1072.1009  &  1066.3581  &  1022.3635  &  1009.2311  & 1.30 &   963.4901  \\
			1137.4735  &  1133.2683  &  1101.1474  &  1091.0507  & 1.30 &  1057.3622  \\
			\hline
			\multicolumn{6}{|c|}{Taylor--Hood + $\mathbb{BDM}_1$}  \\
			\hline
			716.0246  &   715.3147  &   714.2296  &   713.8362  & 1.29 &   712.2499  \\
			772.4976  &   771.5265  &   770.3104  &   769.8766  & 1.37 &   768.1502  \\
			968.7337  &   967.6326  &   966.1352  &   965.6204  & 1.34 &   963.4901  \\
			1061.6803  &  1060.8890  &  1059.6190  &  1059.1303  & 1.32 &  1057.3622  \\
			\hline
			\hline
		\end{tabular}
}
\end{table}

\begin{figure}[!hbt]
\centering
\begin{minipage}{0.24\linewidth}\centering
	{\footnotesize $\bu_{h,1}$}\\
	\includegraphics[scale=0.102,trim=22cm 4cm 20cm 6cm,clip]{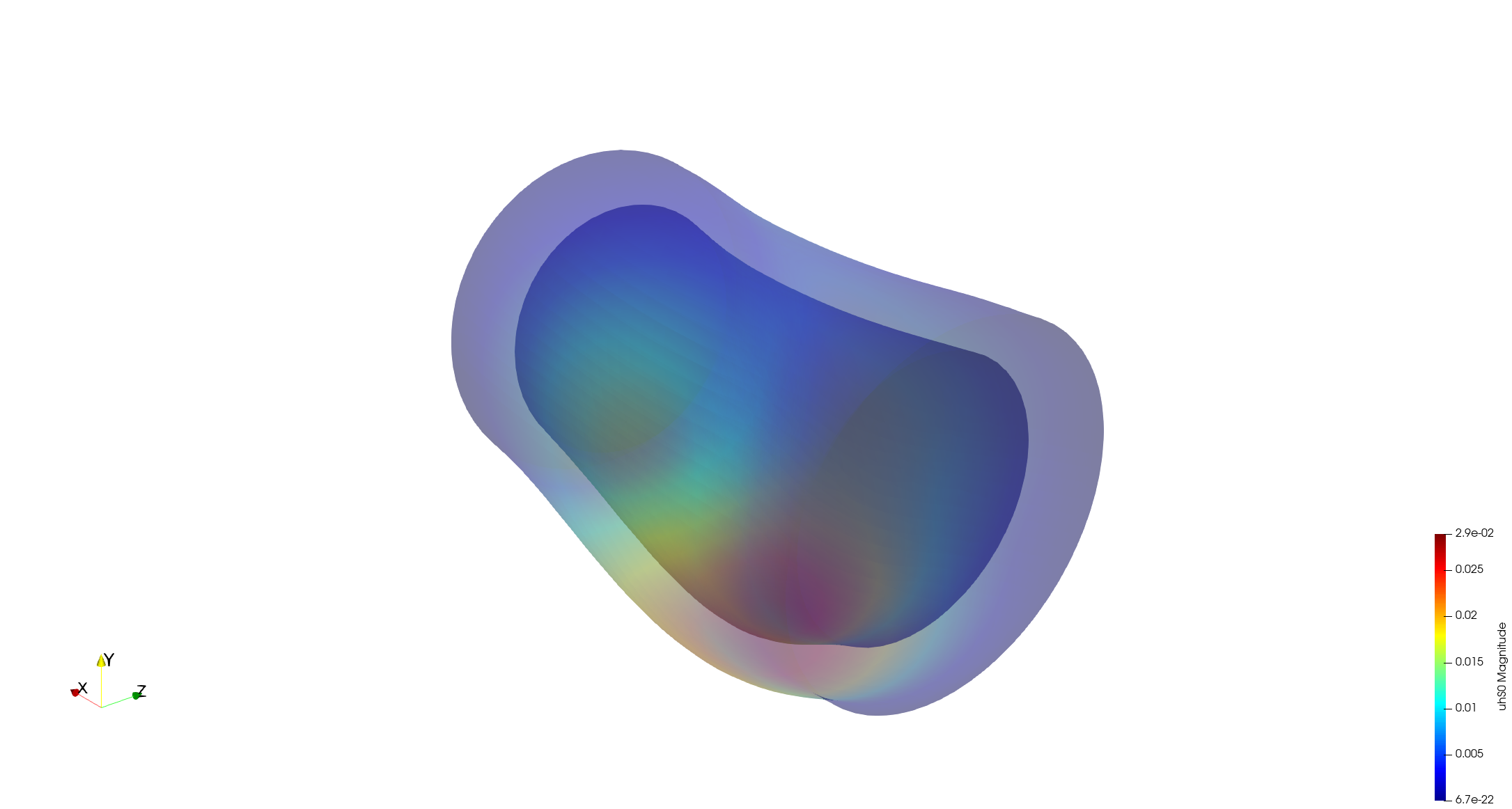}
\end{minipage}
\begin{minipage}{0.24\linewidth}\centering
	{\footnotesize $\bu_{h,2}$}\\
	\includegraphics[scale=0.102,trim=22cm 4cm 20cm 6cm,clip]{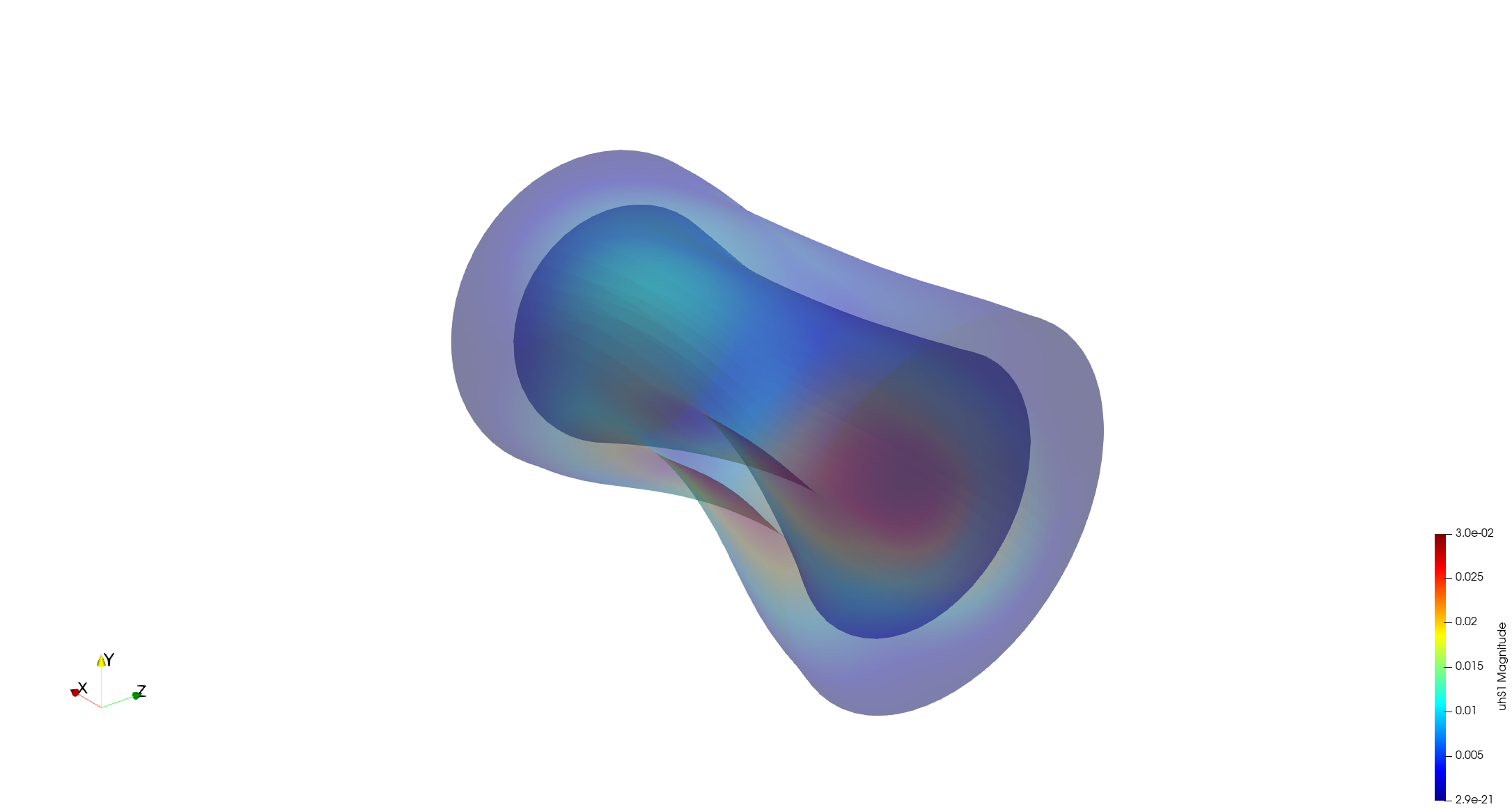}
\end{minipage}
\begin{minipage}{0.24\linewidth}\centering
	{\footnotesize $\bu_{h,3}$}\\
	\includegraphics[scale=0.102,trim=22cm 4cm 20cm 6cm,clip]{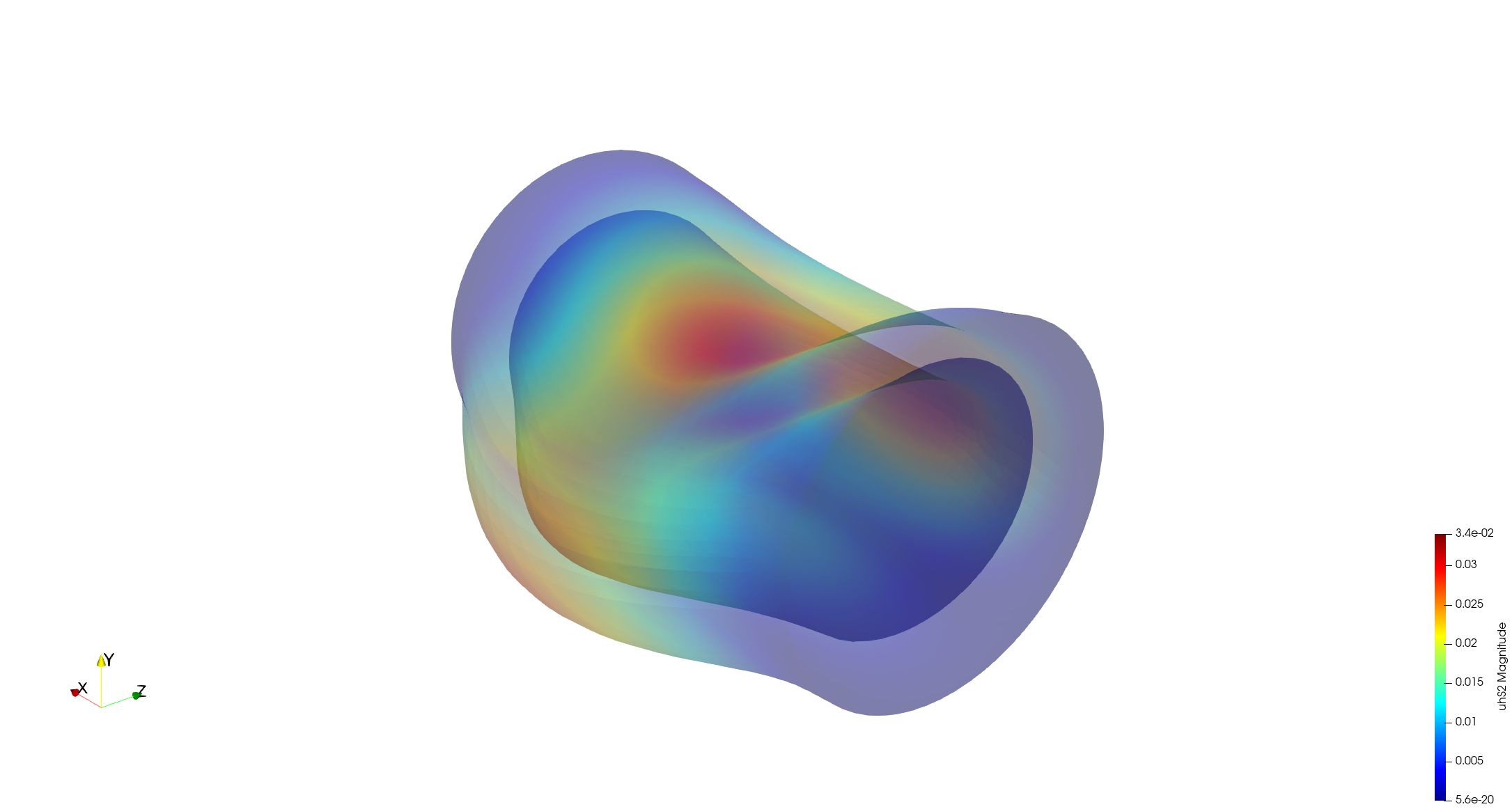}
\end{minipage}
\begin{minipage}{0.24\linewidth}\centering
	{\footnotesize $\bu_{h,4}$}\\
	\includegraphics[scale=0.102,trim=22cm 4cm 20cm 6cm,clip]{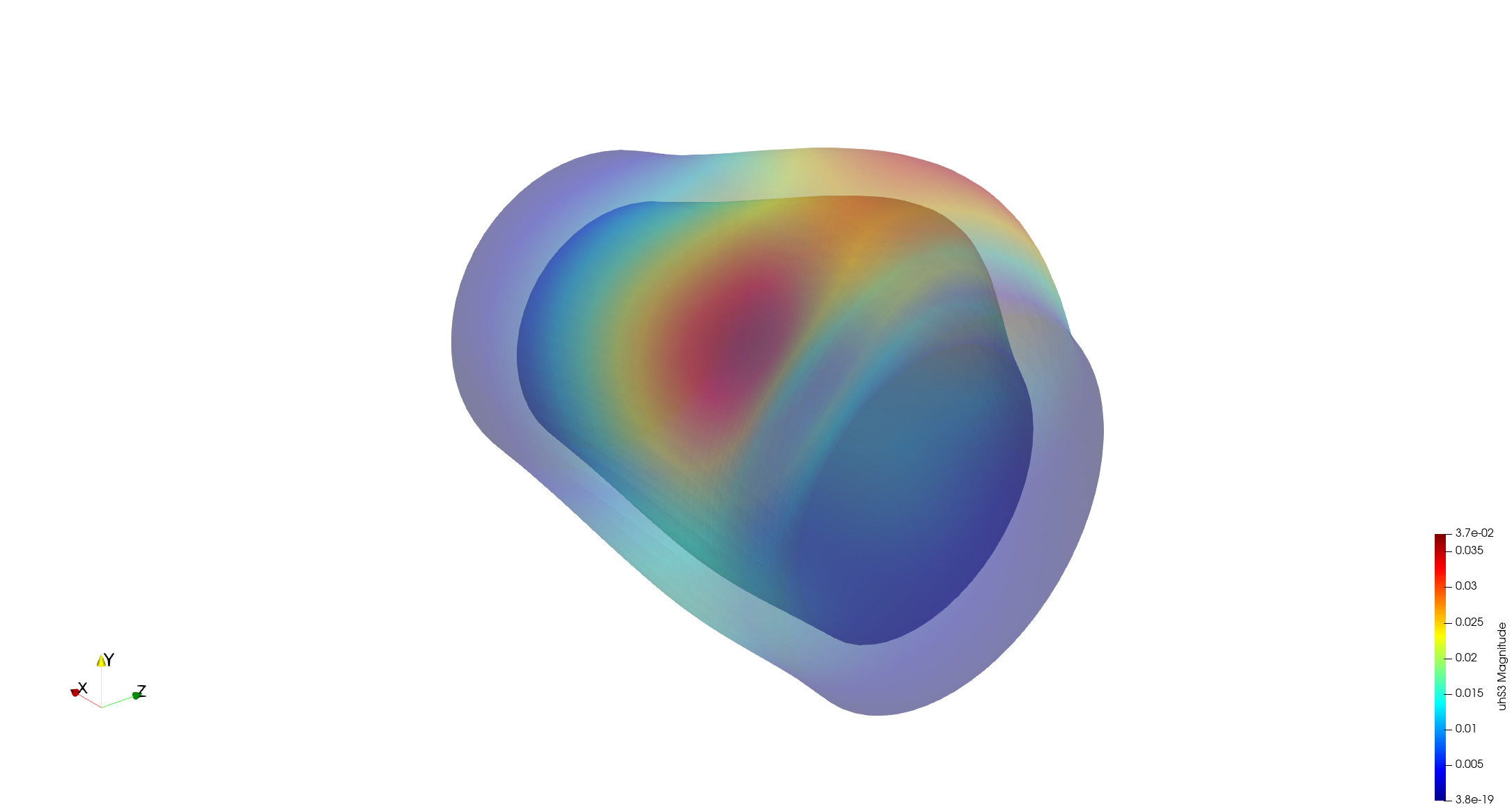}
\end{minipage}
\begin{minipage}{0.24\linewidth}\centering
	{\footnotesize $p_{h,1}$}\\
	\includegraphics[scale=0.102,trim=22cm 4cm 20cm 6cm,clip]{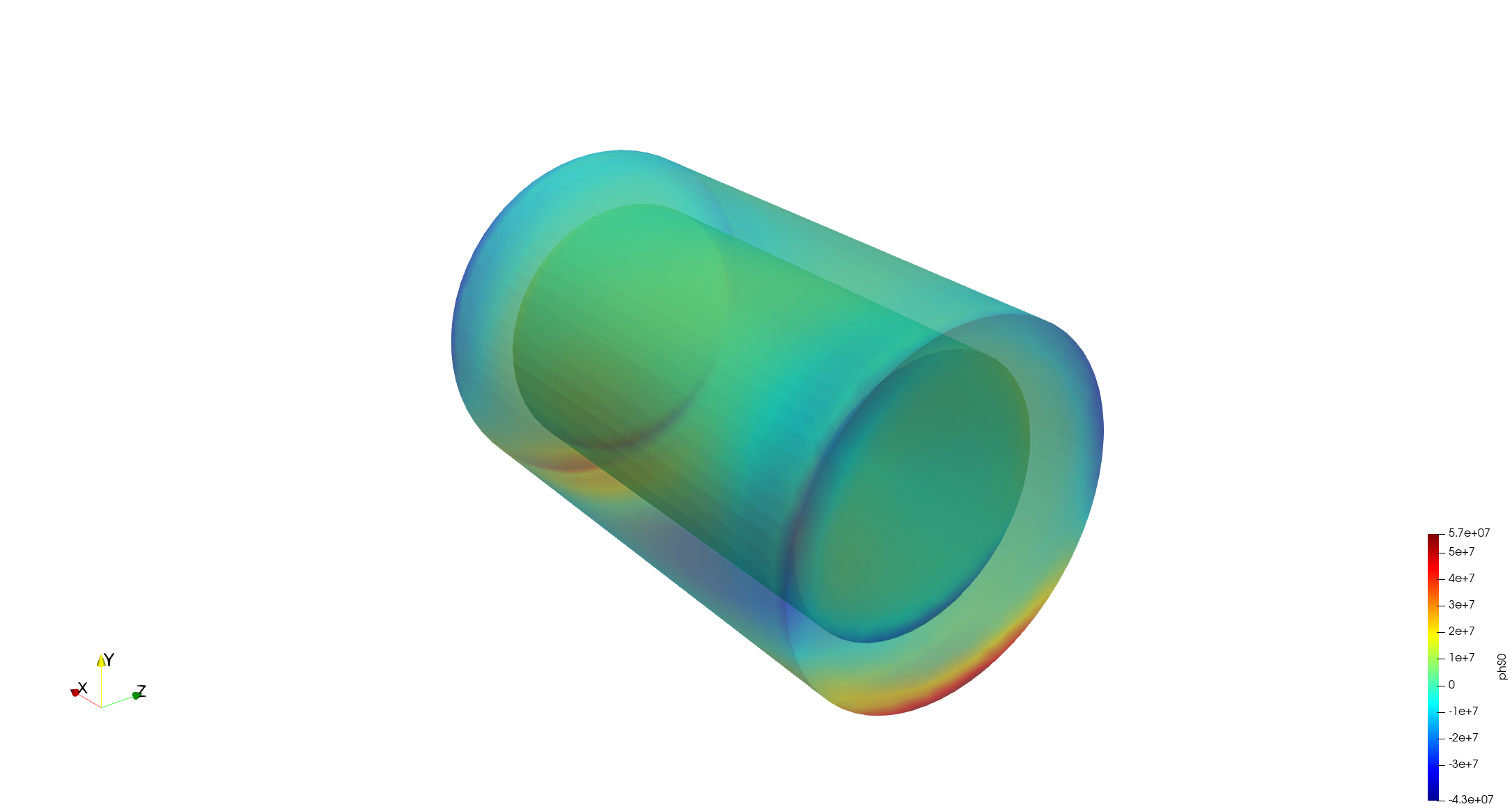}
\end{minipage}
\begin{minipage}{0.24\linewidth}\centering
	{\footnotesize $p_{h,2}$}\\
	\includegraphics[scale=0.102,trim=22cm 4cm 20cm 6cm,clip]{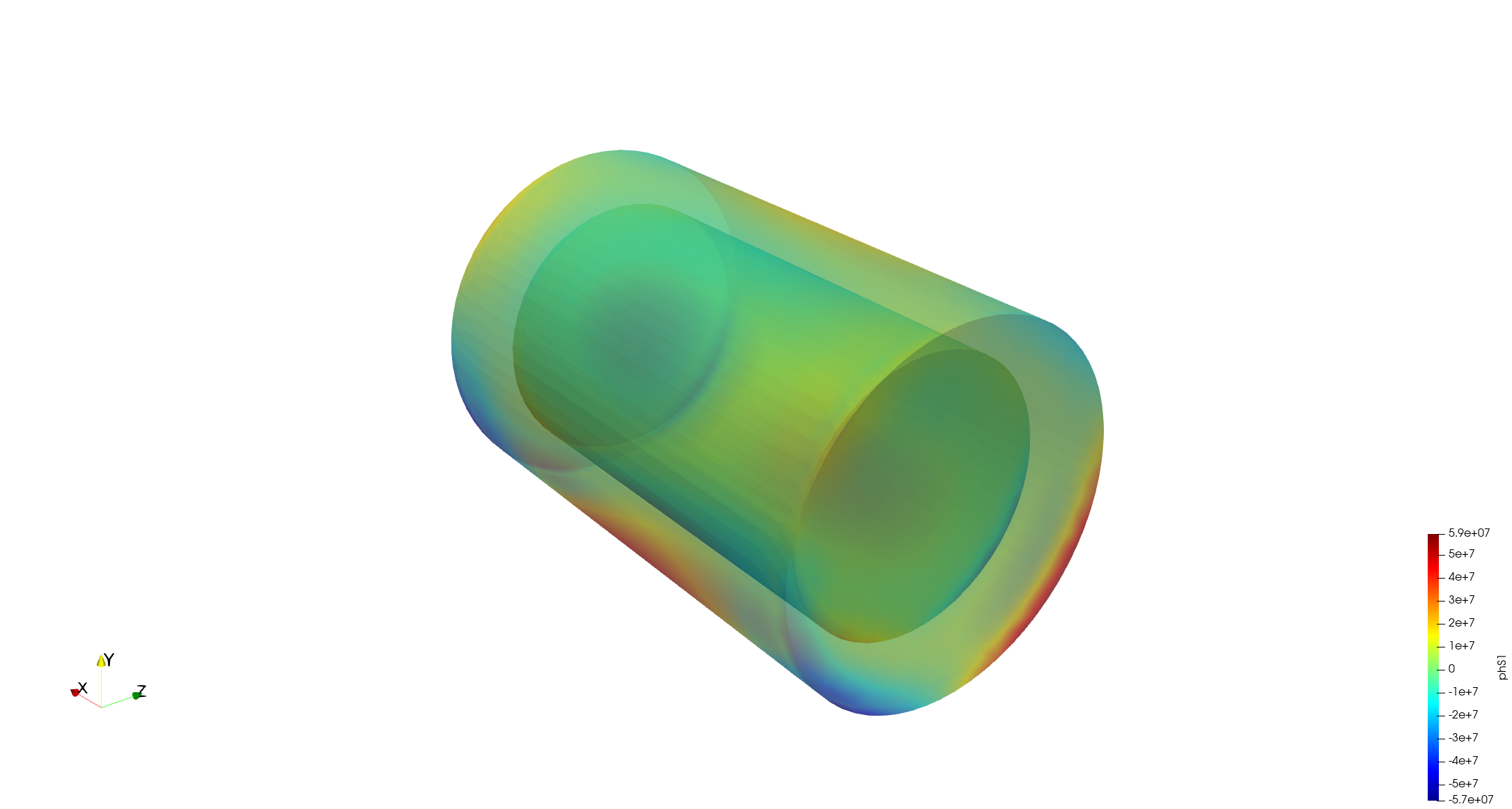}
\end{minipage}
\begin{minipage}{0.24\linewidth}\centering
	{\footnotesize $p_{h,3}$}\\
	\includegraphics[scale=0.102,trim=22cm 4cm 20cm 6cm,clip]{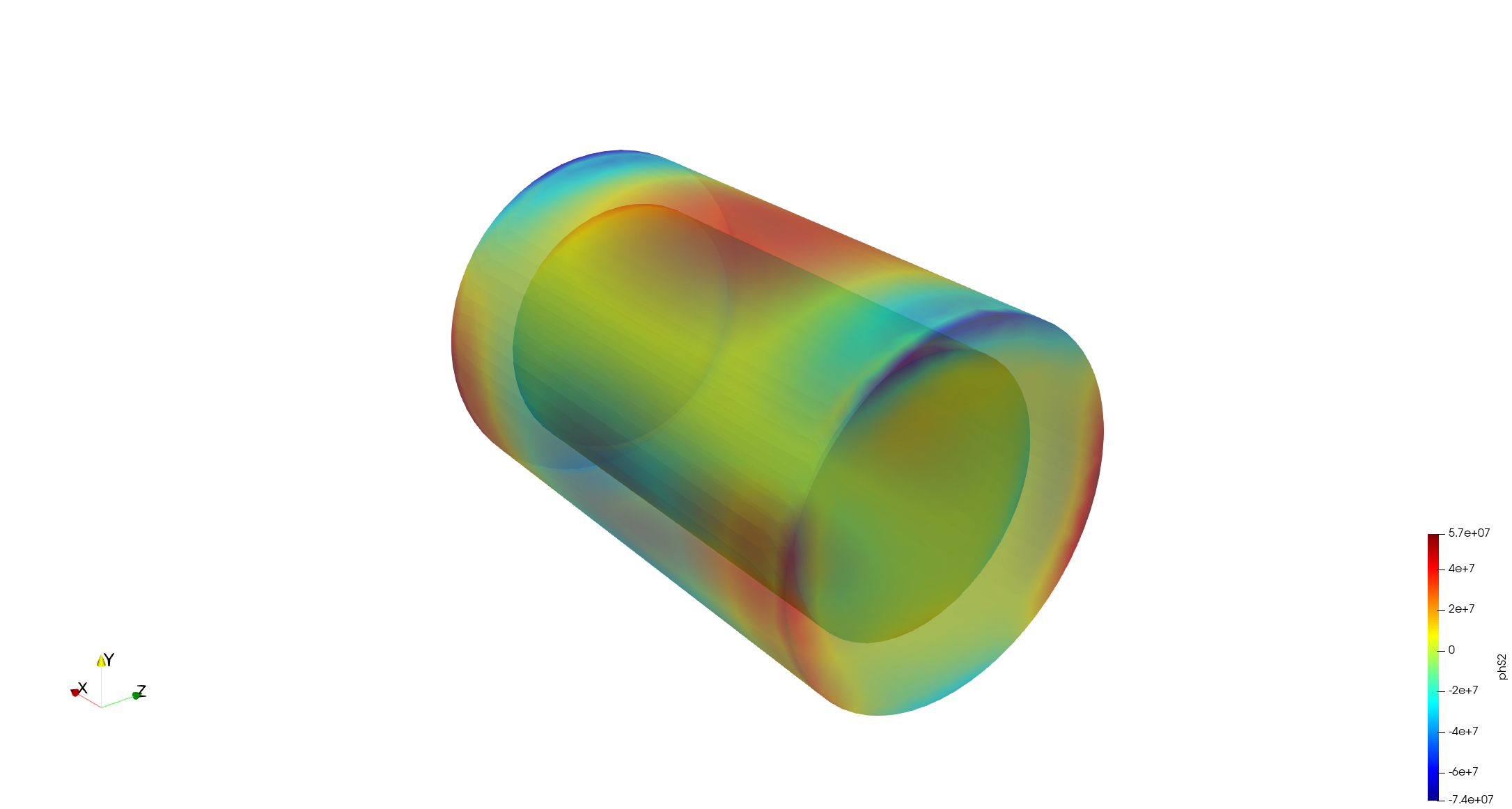}
\end{minipage}
\begin{minipage}{0.24\linewidth}\centering
	{\footnotesize $p_{h,4}$}\\
	\includegraphics[scale=0.102,trim=22cm 4cm 20cm 6cm,clip]{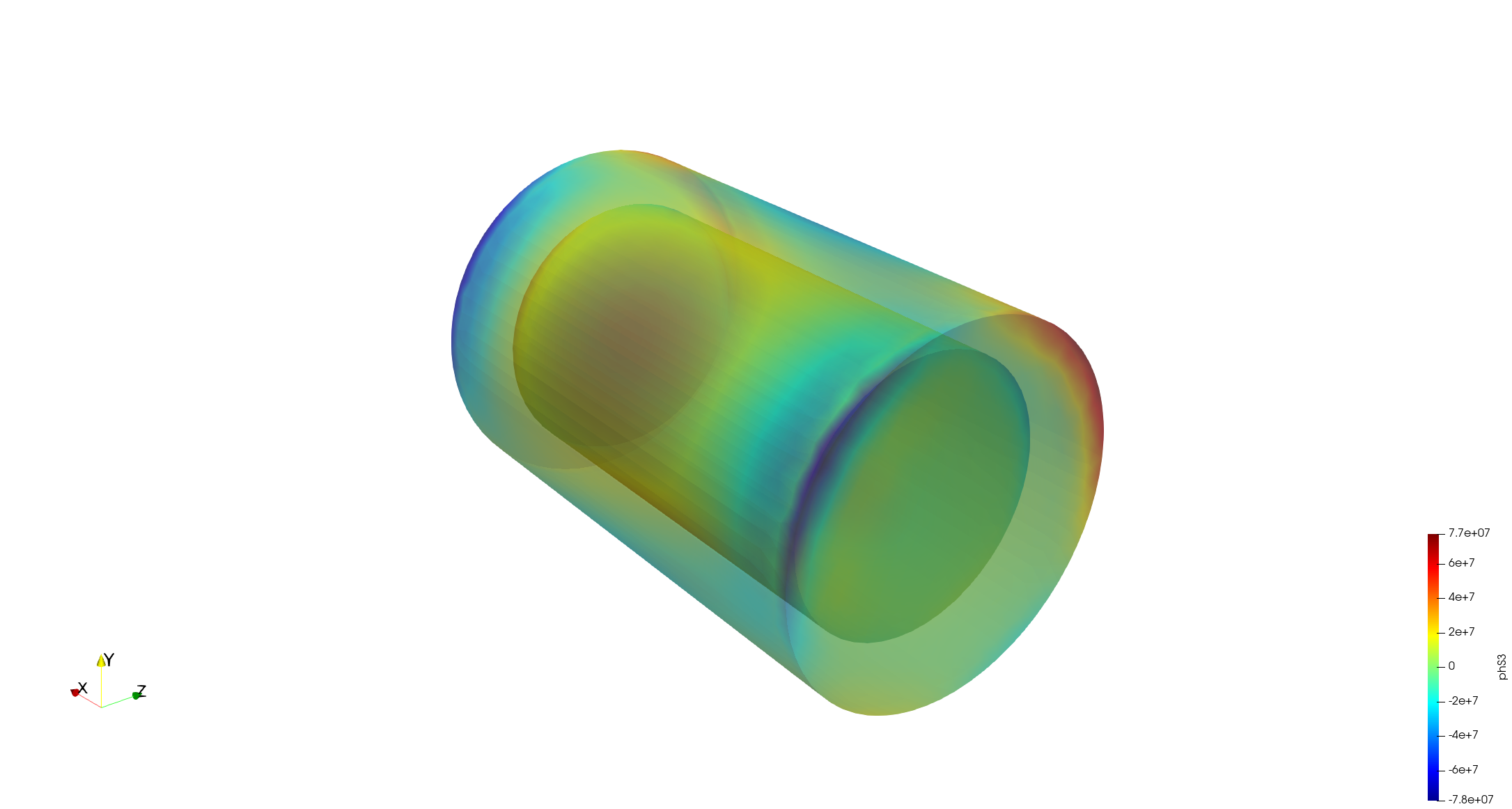}
\end{minipage}
\begin{minipage}{0.24\linewidth}\centering
	{\footnotesize $\bw_{h,1}$}\\
	\includegraphics[scale=0.102,trim=22cm 4cm 20cm 6cm,clip]{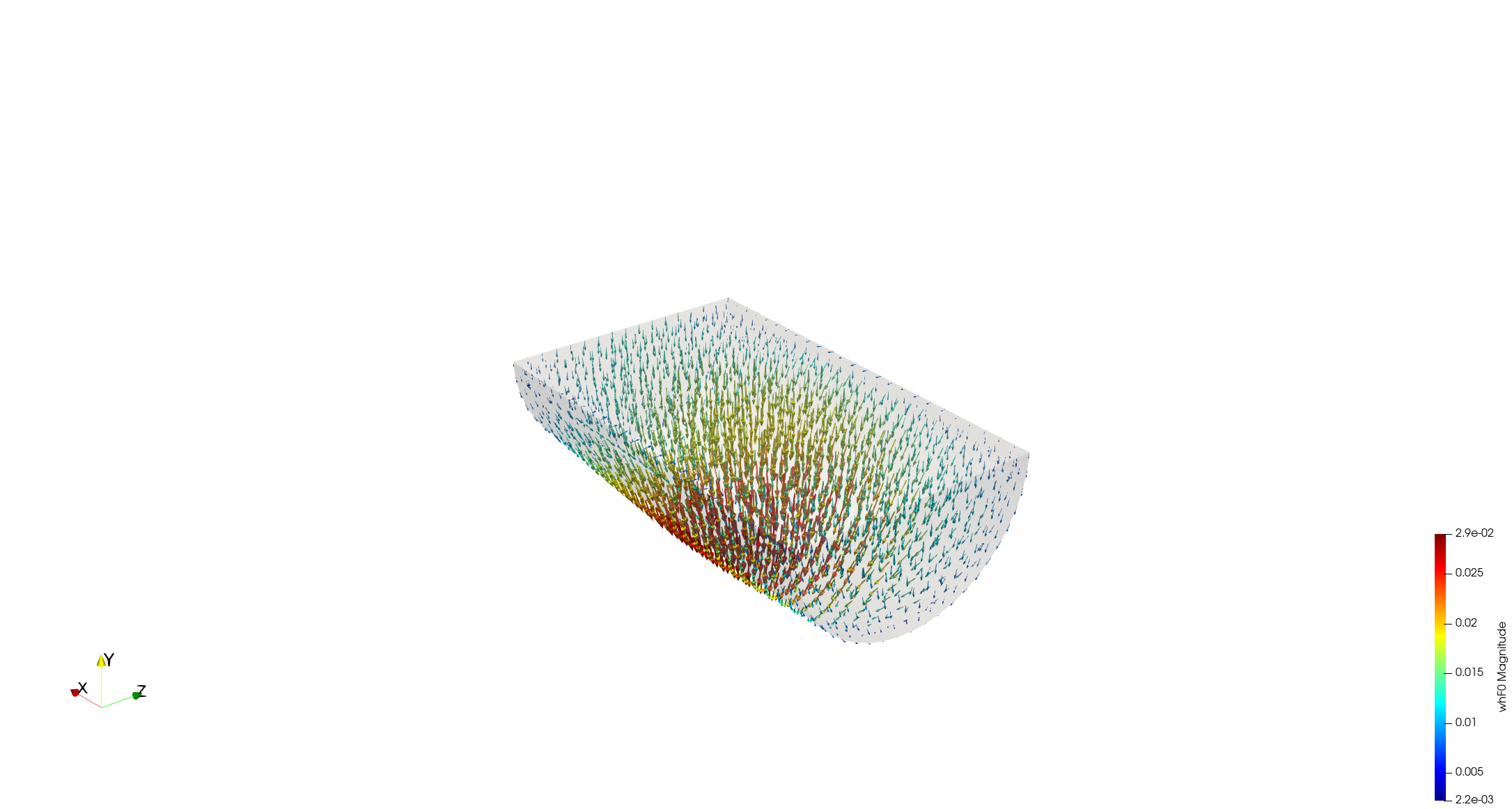}
\end{minipage}
\begin{minipage}{0.24\linewidth}\centering
	{\footnotesize $\bw_{h,2}$}\\
	\includegraphics[scale=0.102,trim=22cm 4cm 20cm 6cm,clip]{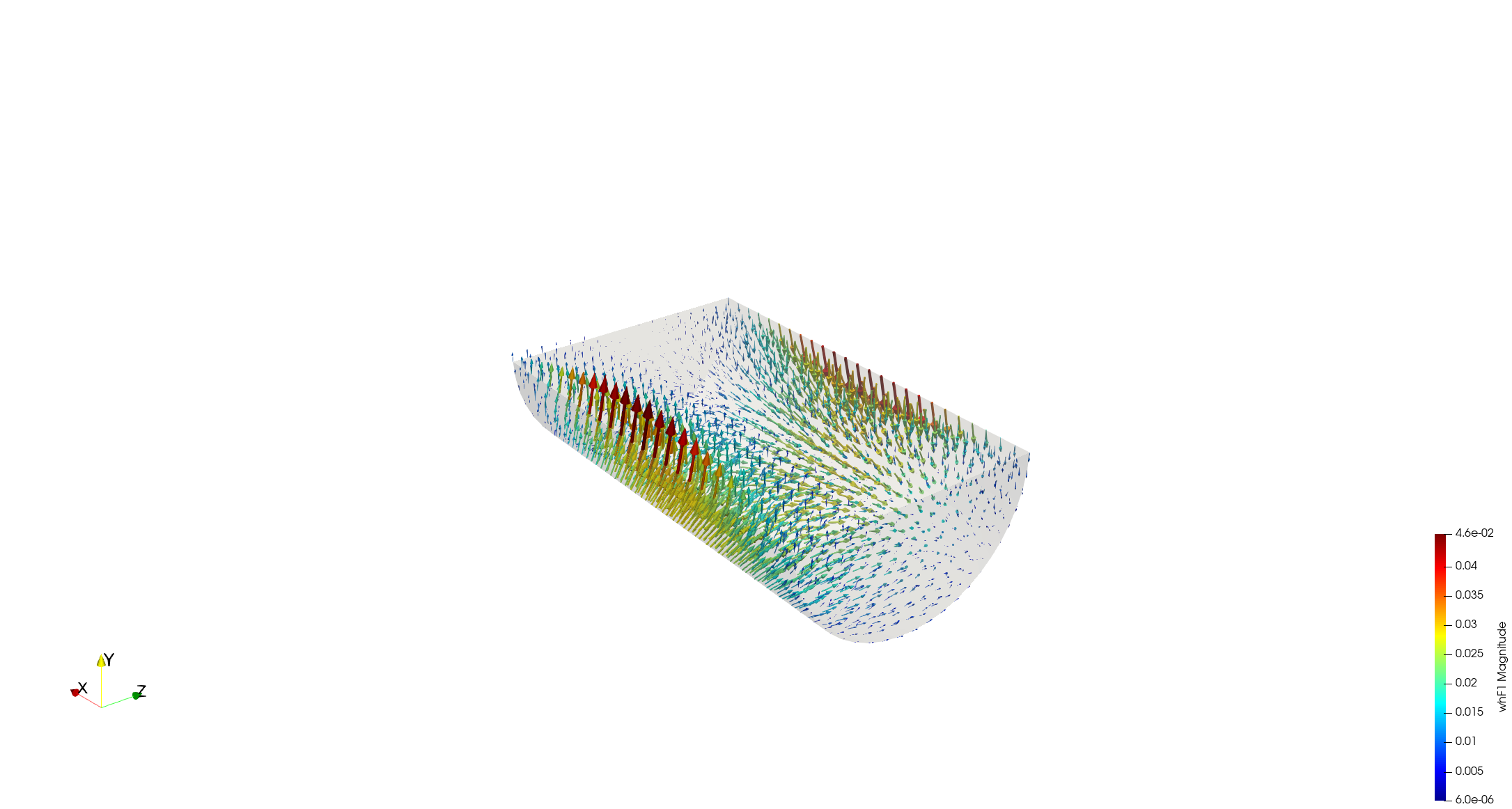}
\end{minipage}
\begin{minipage}{0.24\linewidth}\centering
	{\footnotesize $\bw_{h,3}$}\\
	\includegraphics[scale=0.102,trim=22cm 4cm 20cm 6cm,clip]{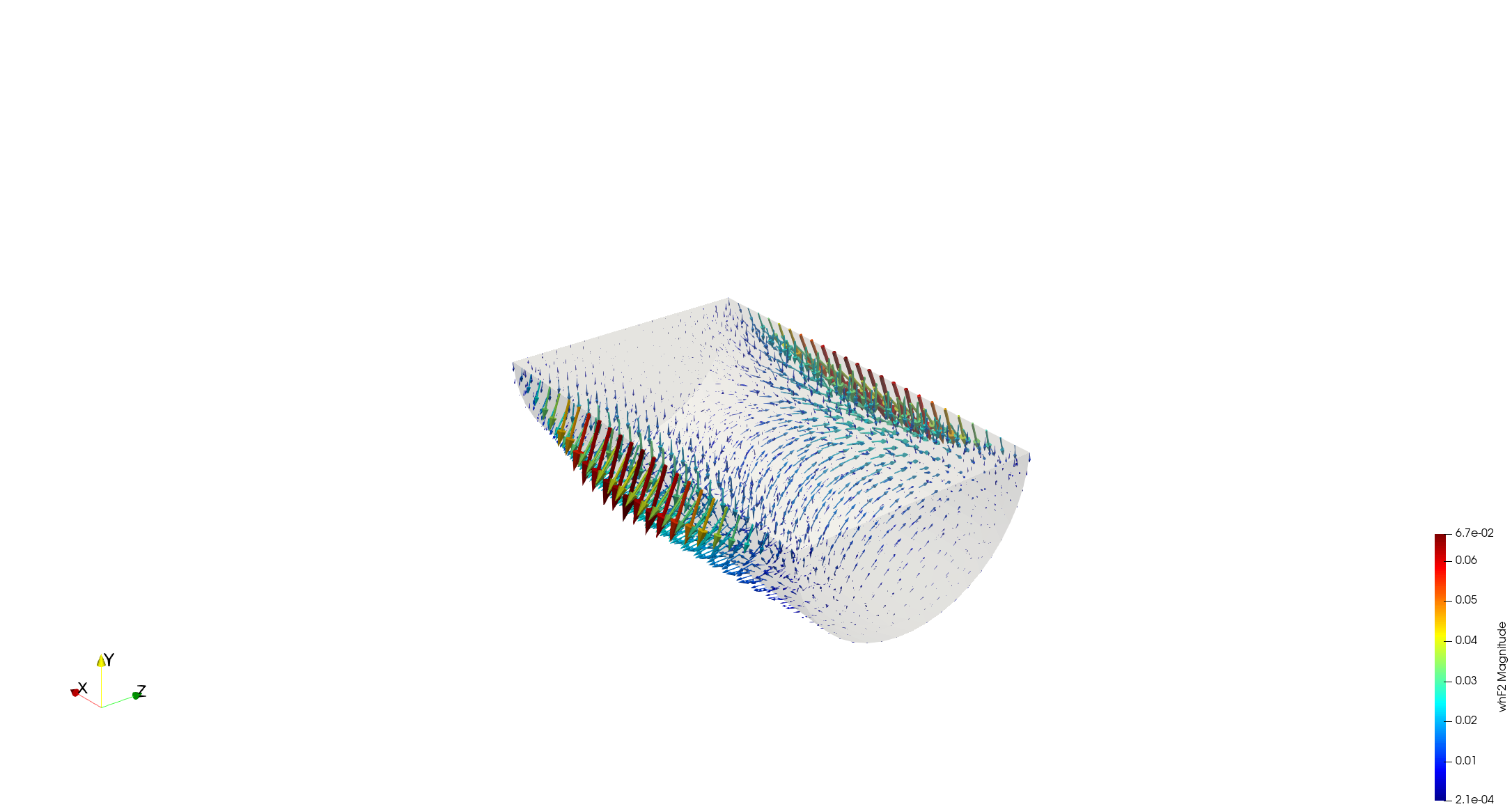}
\end{minipage}
\begin{minipage}{0.24\linewidth}\centering
	{\footnotesize $\bw_{h,4}$}\\
	\includegraphics[scale=0.102,trim=22cm 4cm 20cm 6cm,clip]{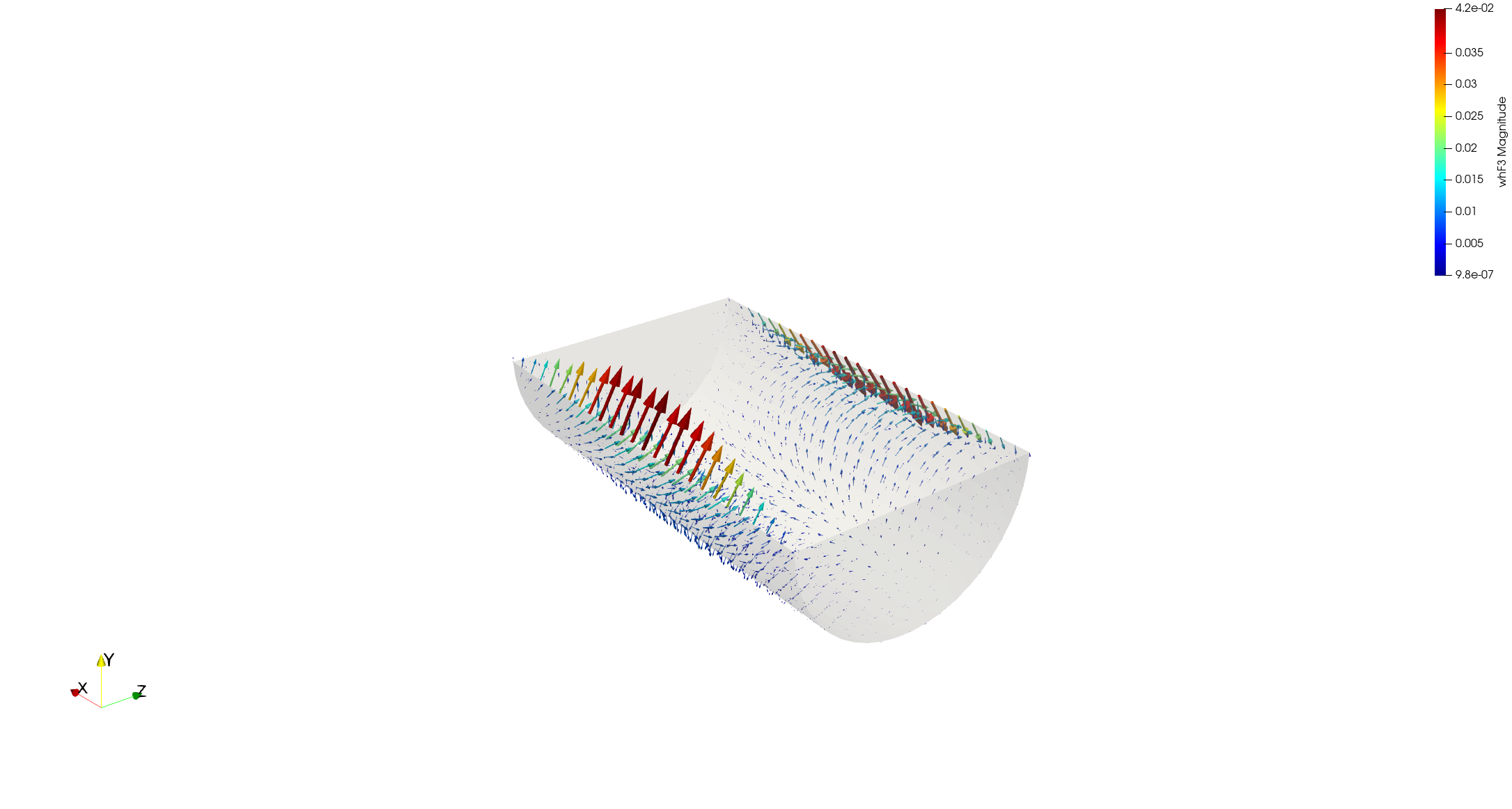}
\end{minipage}
\caption{Example \ref{subsec:3Dbarrel}. Sample numerical solutions of the first four elasto-acoustic modes on the barrel domain. The solid subdomain {has} been warped by a sufficiently large factor in order to observe the deformation.}\label{fig:3D-sols}
\end{figure}

\subsection{A posteriori error estimation}\label{subsec:afem}
With the aim of assessing the performance of the  a posteriori error estimator,  we consider domains   with singularities in two and three dimensions.  On each adaptive iteration, we use the blue-green marking strategy to refine each $T'\in \CT_{h}$ whose indicator $\eta_{T'}$ satisfies
$
\eta_{T'}\geq 0.5\max\{\eta_{T}\,:\,T\in\CT_{h} \}.
$
The effectivity indexes with respect to $\eta$ and the eigenvalue $\omega_i$ are defined by
$$
\eff(\omega_i):={\err(\omega_i)}/{\eta^2}. 		$$
For simplicity, the adaptive algorithm uses the MINI-element family for the solid displacement and pressure.

\subsubsection{2D example}
From Section \ref{subsec:accuracy-test} we observed that elasto-acoustic modes converges with suboptimal rate because of the point singularities in the reentrant corner. In this example we study the convergence on the adaptive process when this configuration is considered. The domain under consideration is the same as in Figure \ref{fig:2D-meshed-domain}(right). We will analyze the convergence of the first eigenvalue for several values of $\nu$.  
This is observed in the intermediate meshes portrayed in Figure \ref{fig:afem-2D-nu049-050}. Also, we note that the refinements are concentrated where the solid pressure $p_{h,1}$ becomes singular, namely, near the re-entrant corner and where the boundary conditions change from Dirichlet to Neumann type. 

In order to further examine the locking-free property of our scheme, we proceed to change the Poisson ratio but keeping the rest of the physical parameters unmodified. Adaptive meshes for $\nu=0.49$ and $\nu=0.5$ are also depicted in Figure \ref{fig:afem-2D-nu049-050}. Here, we observe that less elements are used the closer we get to $\nu=0.5$.   The error decay is presented in Figure \ref{fig:error-2D-adaptive} for different values of $\nu$. We observe that optimal rates are attained for some $h_0$ as predicted in Section \ref{sec:apost}. We end the experiment by investigating the effectivity curves, reported in Figure \ref{fig:efficiency-2D-adaptive}. Properly bounded effectivity values are observed for the studied eigenvalues with the selected values for $\nu$.	

\begin{figure}[!t]
\centering
\begin{minipage}{0.32\linewidth}\centering
	\text{$\omega_{h,1}$, \texttt{dof}= 11970}\\
	\includegraphics[scale=0.08,trim=14cm 0cm 14cm 0cm ,clip]{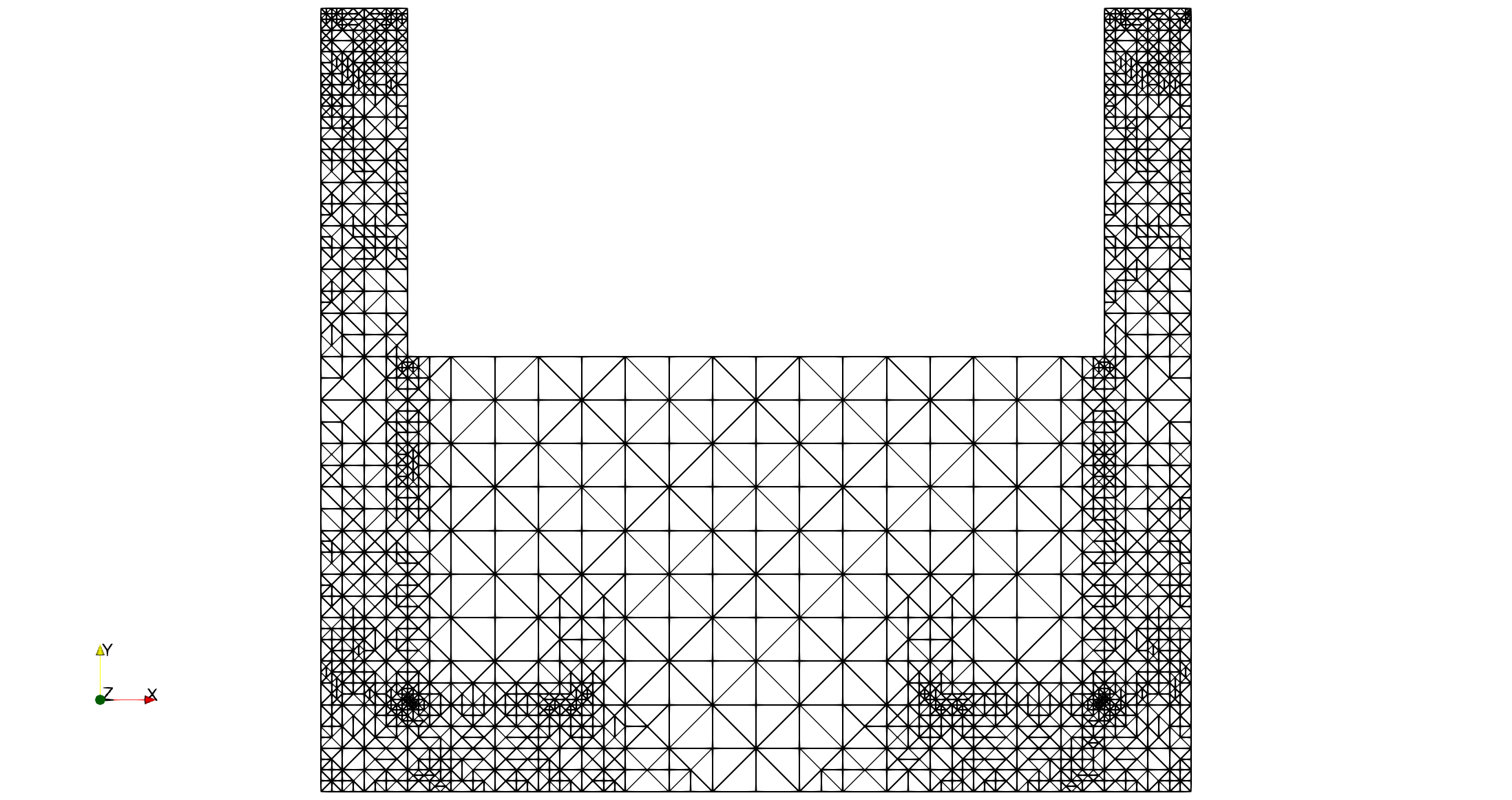}
\end{minipage}
\begin{minipage}{0.32\linewidth}\centering
	\text{$\omega_{h,1}$, \texttt{dof}= 50709}\\
	\includegraphics[scale=0.08,trim=14cm 0cm 14cm 0cm ,clip]{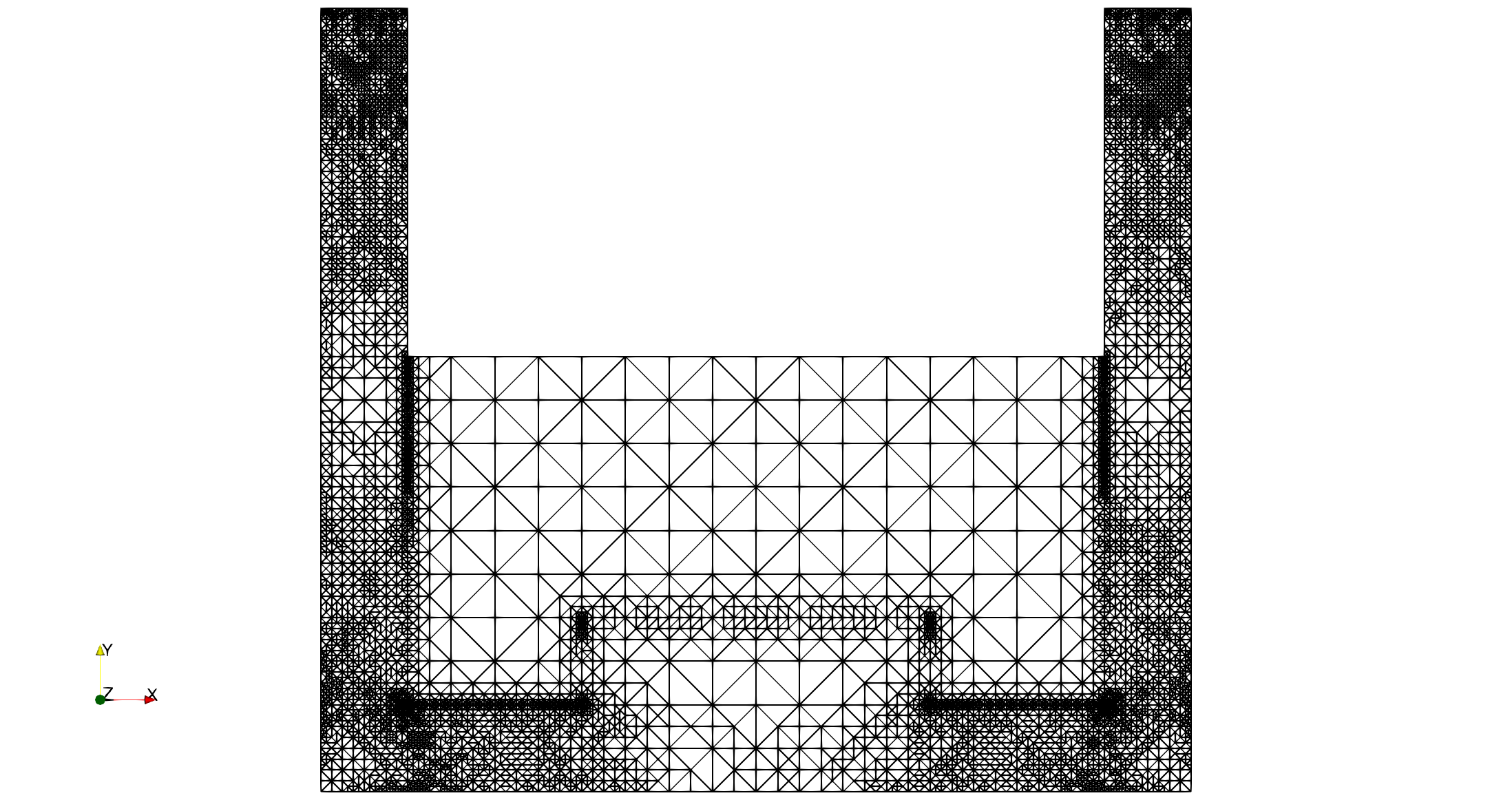}
\end{minipage}
\begin{minipage}{0.32\linewidth}\centering
	\text{$\omega_{h,1}$, \texttt{dof}= 104871}\\
	\includegraphics[scale=0.08,trim=14cm 0cm 14cm 0cm ,clip]{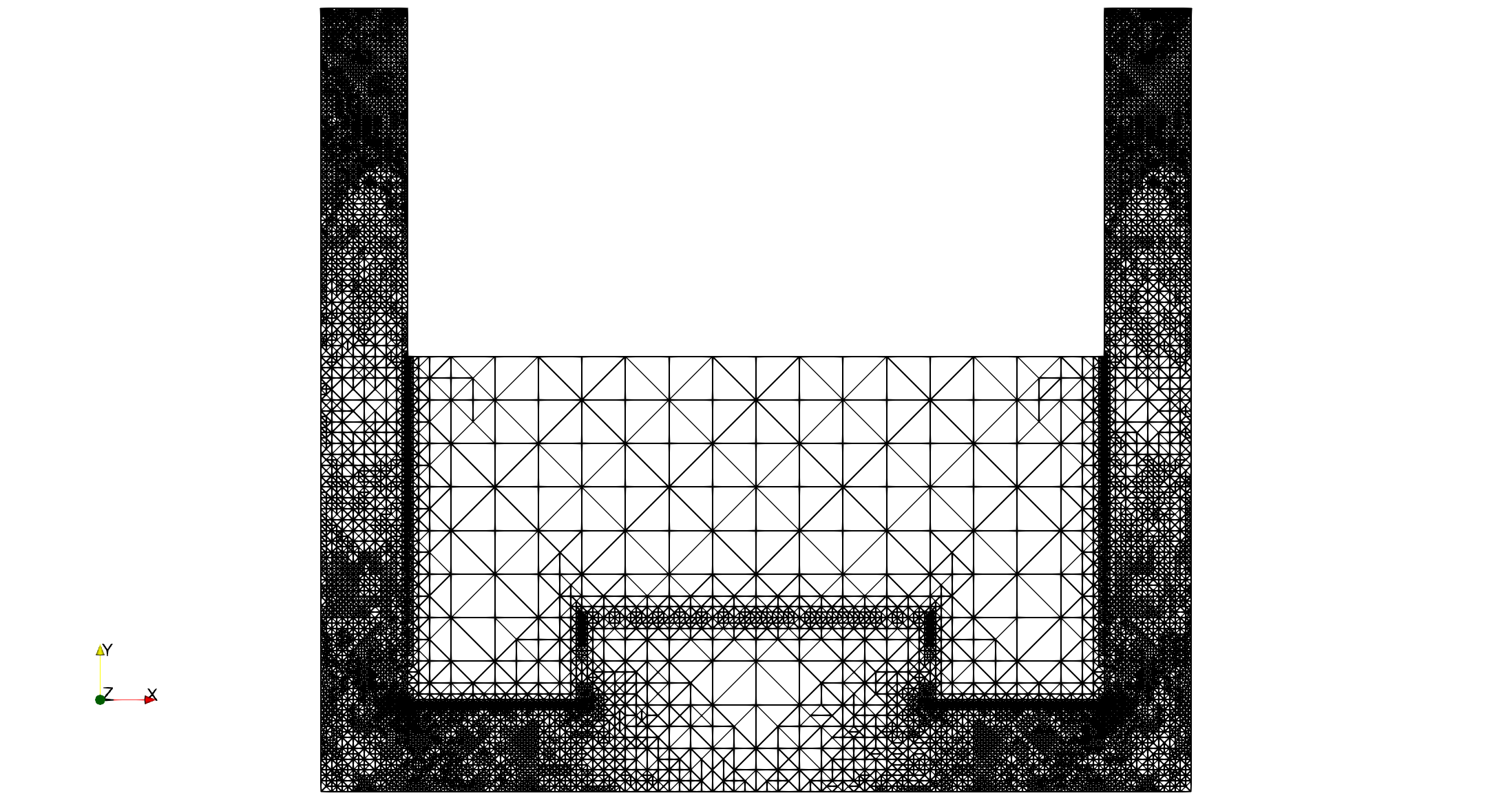}
\end{minipage}\\
\begin{minipage}{0.32\linewidth}\centering
	\text{$\omega_{h,2}$, \texttt{dof}= 13564}\\
	\includegraphics[scale=0.08,trim=14cm 0cm 14cm 0cm ,clip]{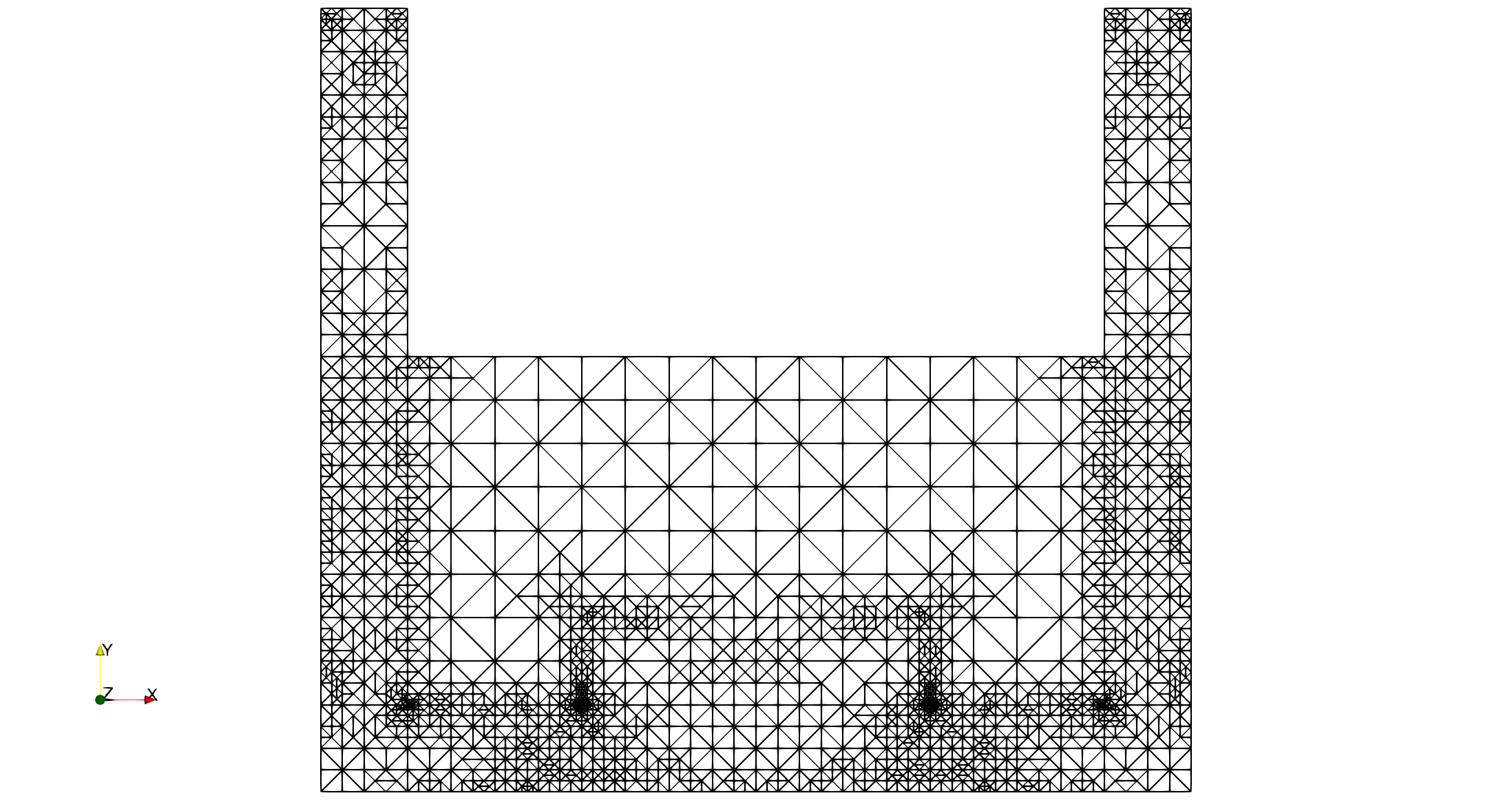}
\end{minipage}
\begin{minipage}{0.32\linewidth}\centering
	\text{$\omega_{h,2}$, \texttt{dof}= 60933}\\
	\includegraphics[scale=0.08,trim=14cm 0cm 14cm 0cm ,clip]{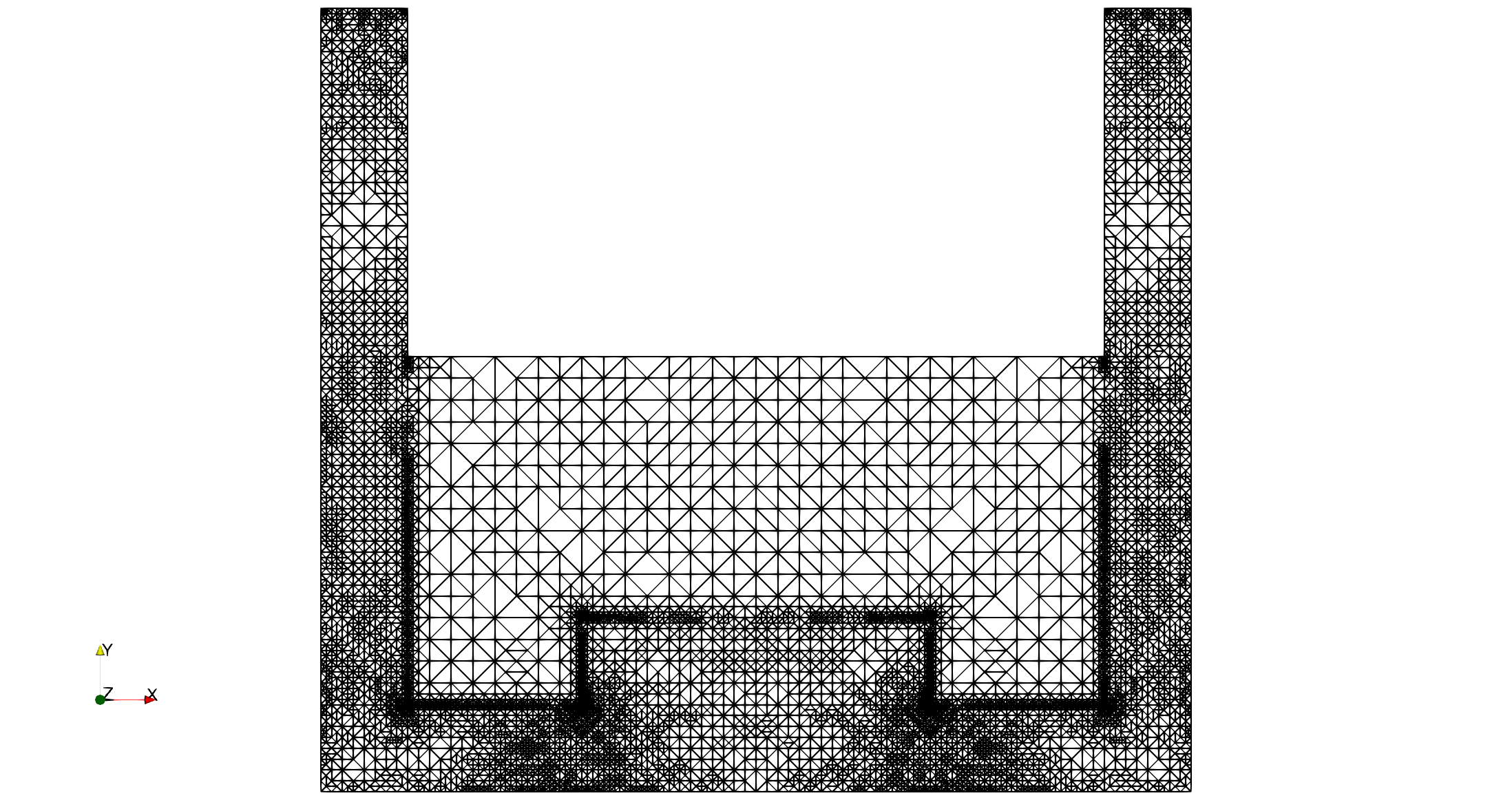}
\end{minipage}
\begin{minipage}{0.32\linewidth}\centering
	\text{$\omega_{h,2}$, \texttt{dof}= 139219}\\
	\includegraphics[scale=0.08,trim=14cm 0cm 14cm 0cm ,clip]{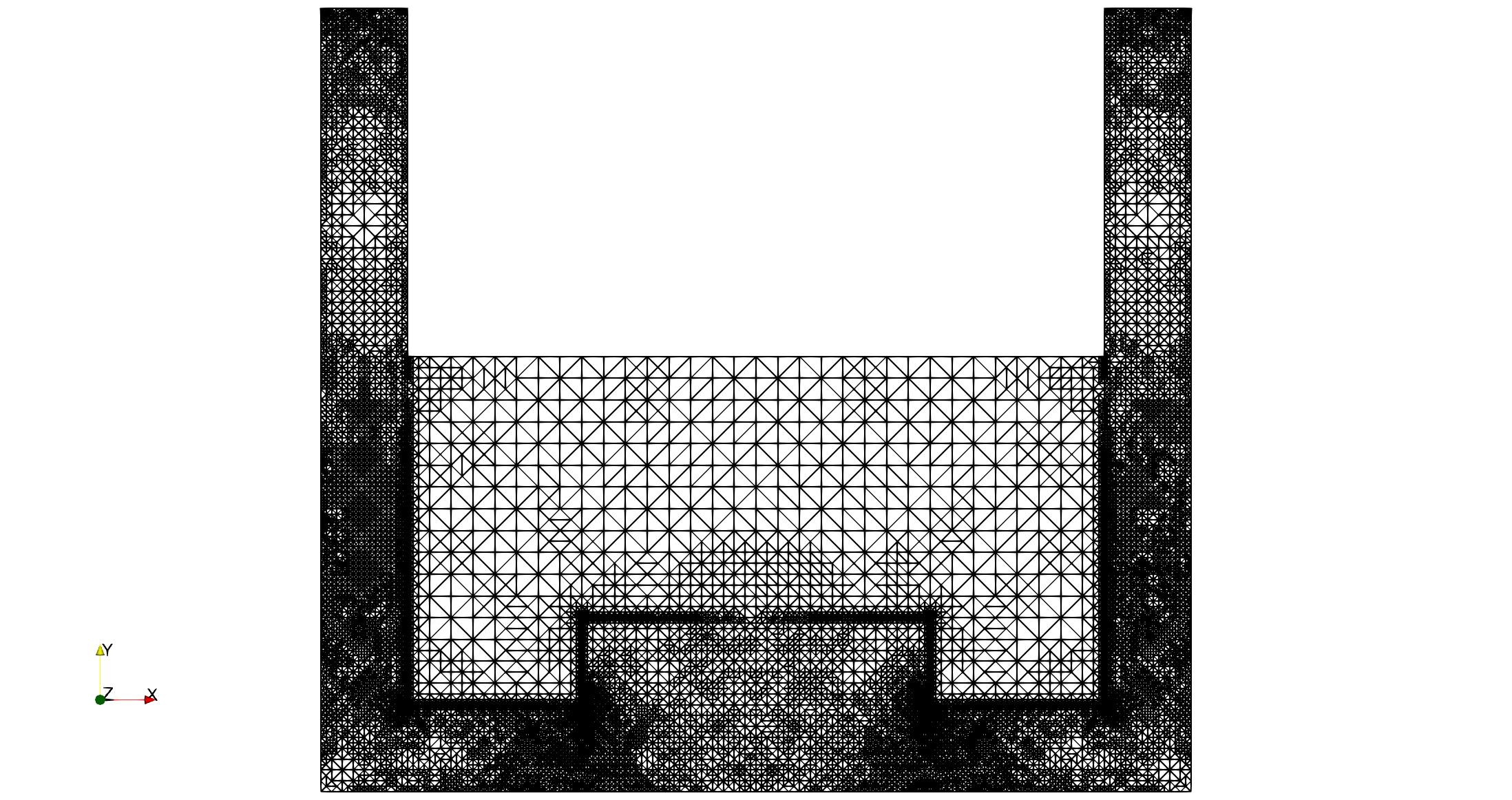}
\end{minipage}\\
\begin{minipage}{0.32\linewidth}\centering
	\text{$\omega_{h,4}$, \texttt{dof}= 13709}\\
	\includegraphics[scale=0.08,trim=14cm 0cm 14cm 0cm ,clip]{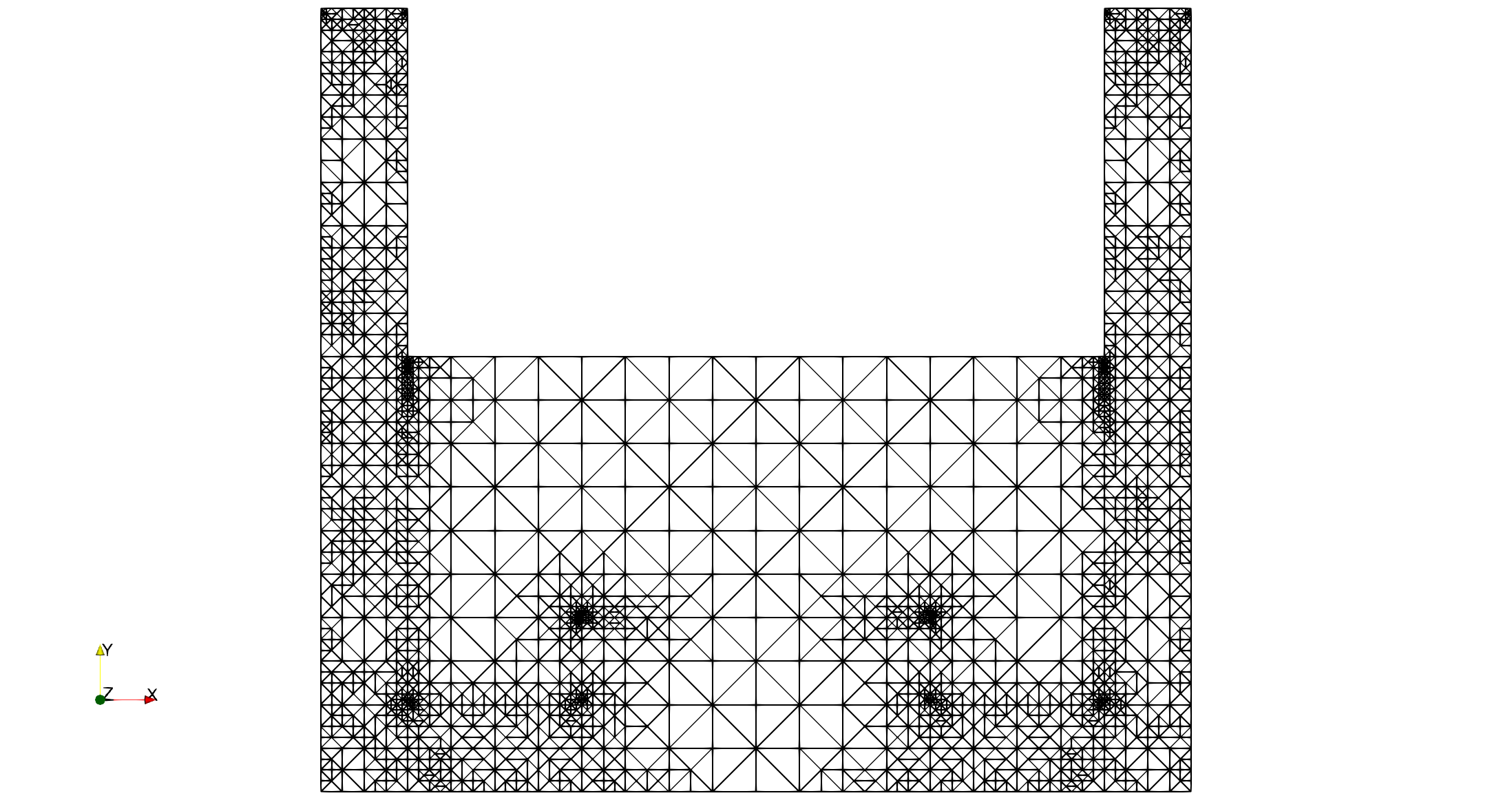}
\end{minipage}
\begin{minipage}{0.32\linewidth}\centering
	\text{$\omega_{h,4}$, \texttt{dof}= 75728} \\
	\includegraphics[scale=0.08,trim=14cm 0cm 14cm 0cm ,clip]{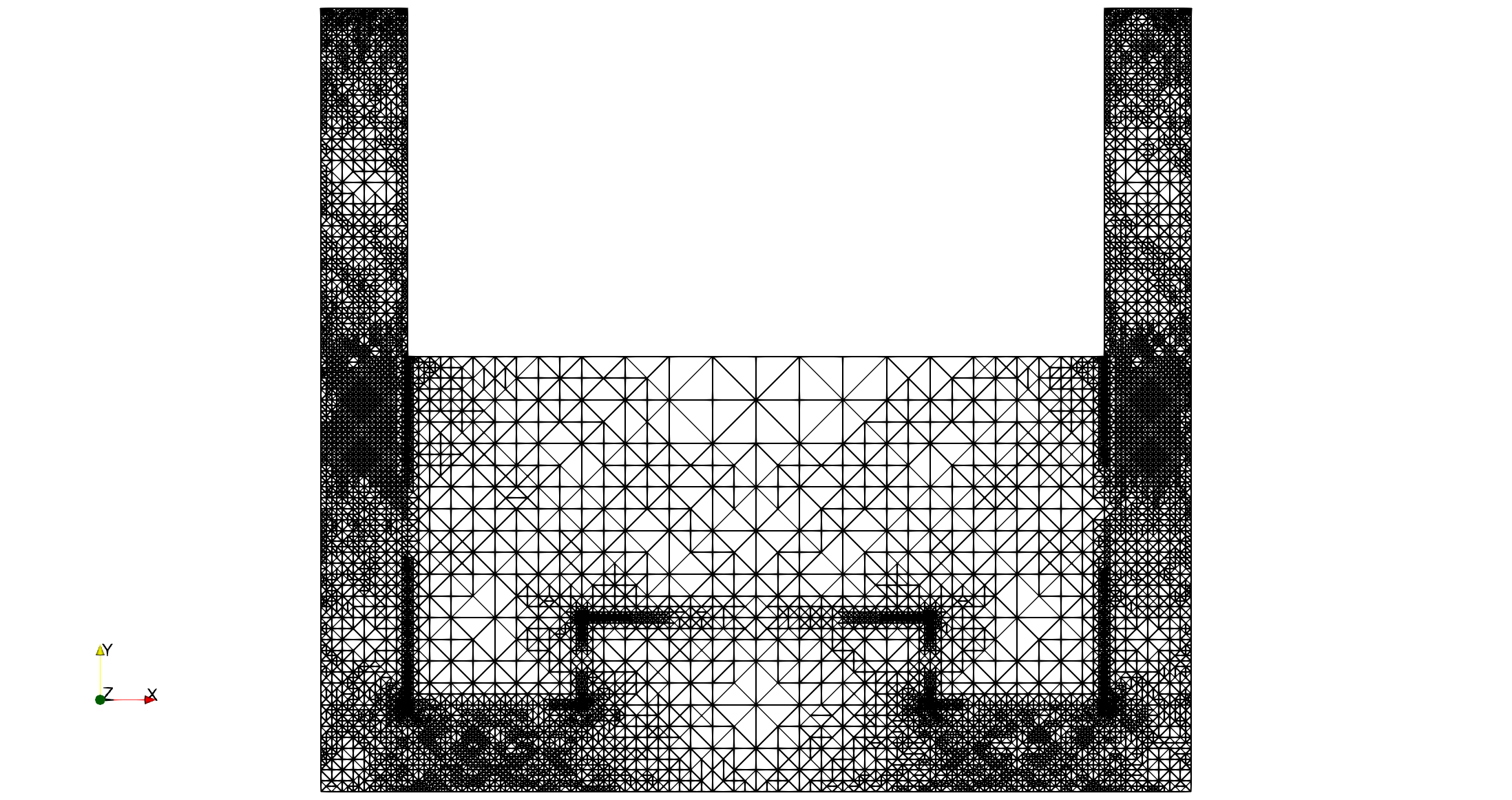}
\end{minipage}
\begin{minipage}{0.32\linewidth}\centering
	\text{$\omega_{h,4}$, \texttt{dof}= 154326}\\
	\includegraphics[scale=0.08,trim=14cm 0cm 14cm 0cm ,clip]{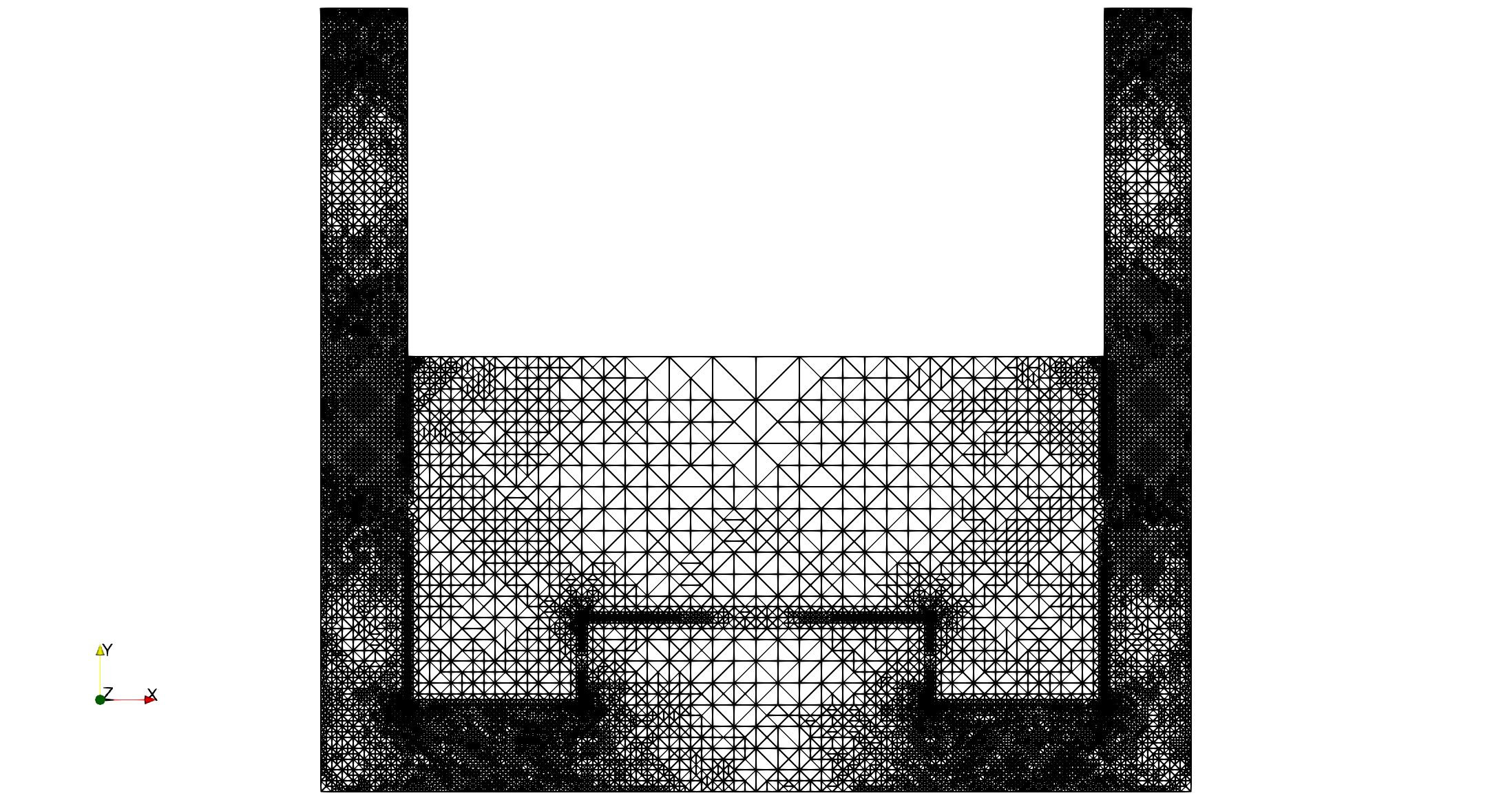}
\end{minipage}
\caption{Example \ref{subsec:afem}. Comparison between intermediate meshes on the adaptive process for the first, second and fourth eigenfrequencies in $\Omega_2$ with $\nu=0.35$.}
\label{fig:afem-2D-nu049-050}
\end{figure}

\begin{figure}[!t]\centering
\begin{minipage}{0.32\linewidth}\centering
	\includegraphics[scale=0.26]{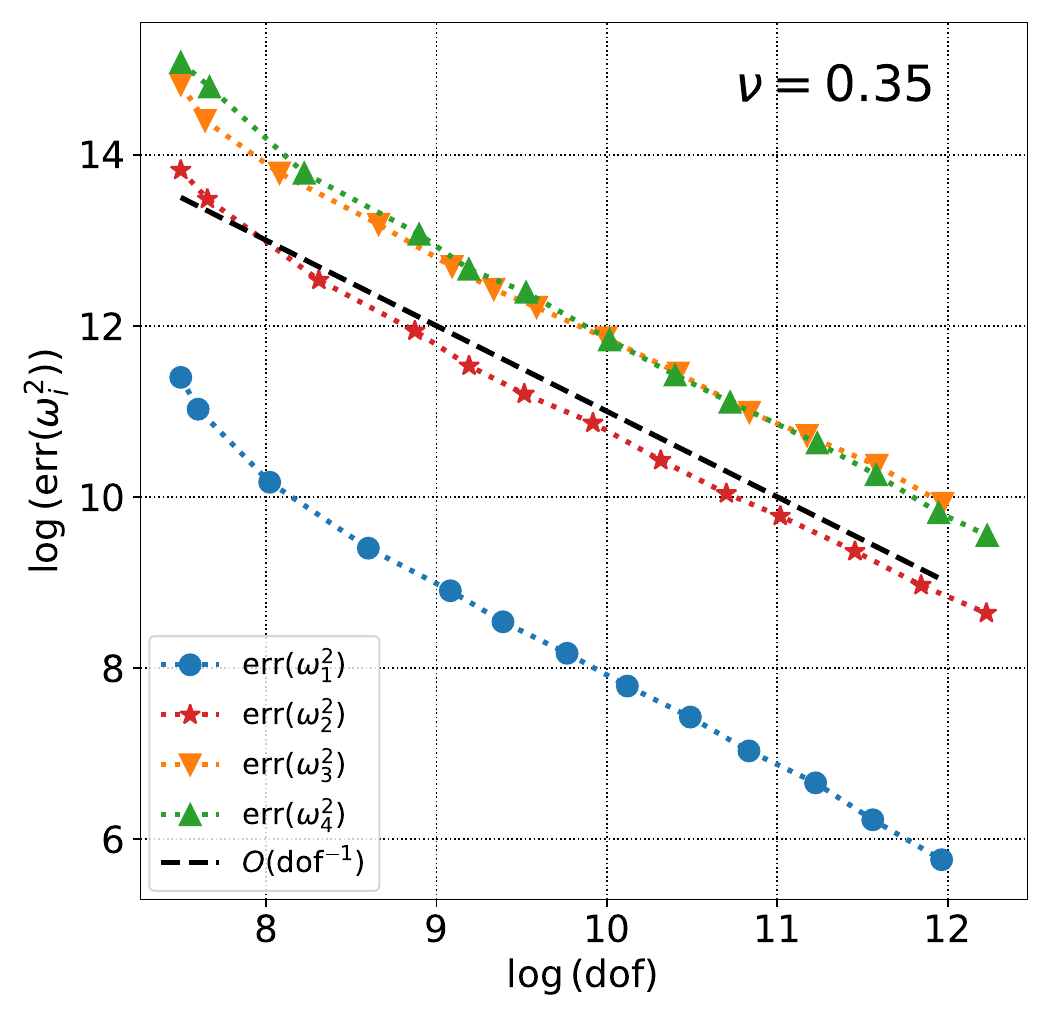}
\end{minipage}
\begin{minipage}{0.32\linewidth}\centering
	\includegraphics[scale=0.26]{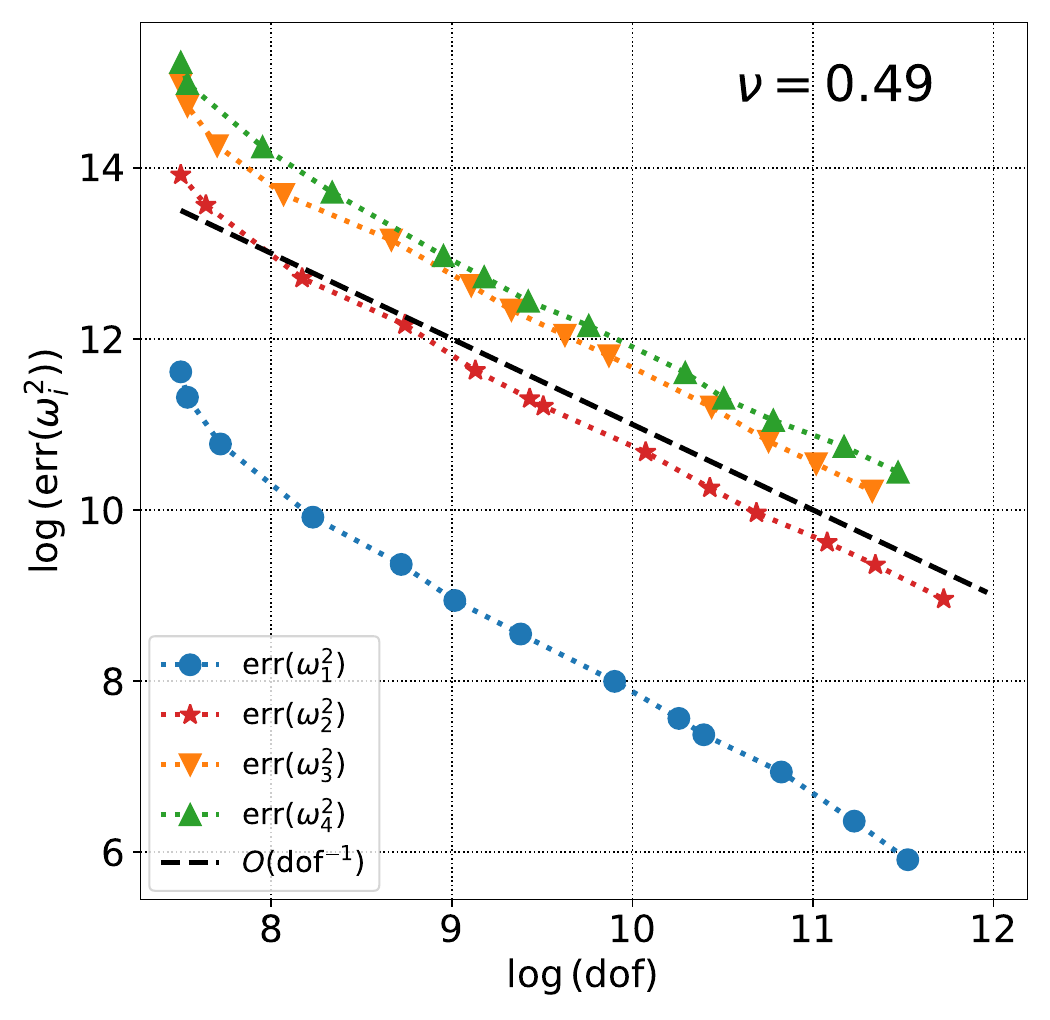}
\end{minipage}
\begin{minipage}{0.32\linewidth}\centering
	\includegraphics[scale=0.26]{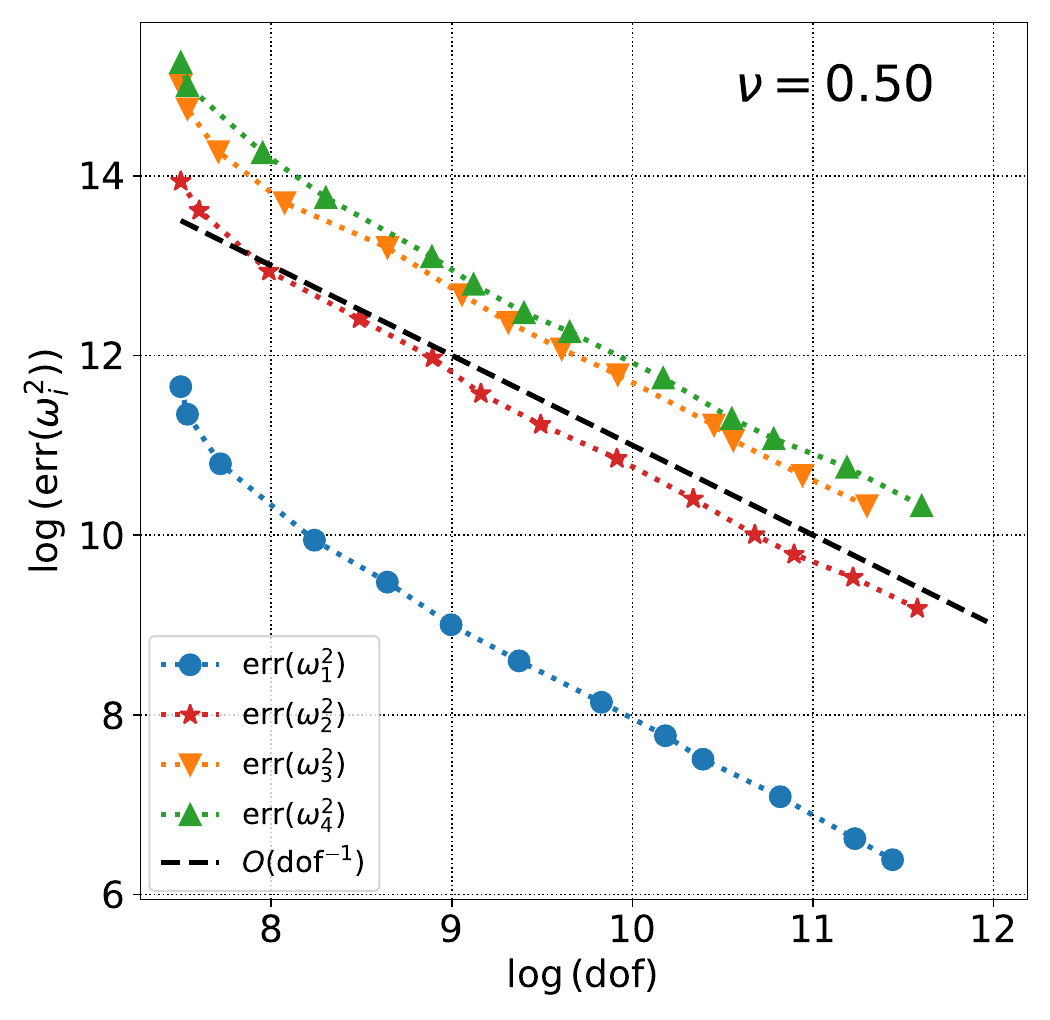}
\end{minipage}
\caption{Example \ref{subsec:afem}. Error history for the first four lowest computed frequencies in the adaptive algorithm on $\Omega_{2}$ with different values of $\nu$.}
\label{fig:error-2D-adaptive}
\end{figure}

\begin{figure}[!hbt]\centering
\begin{minipage}{0.32\linewidth}\centering
	\includegraphics[scale=0.26]{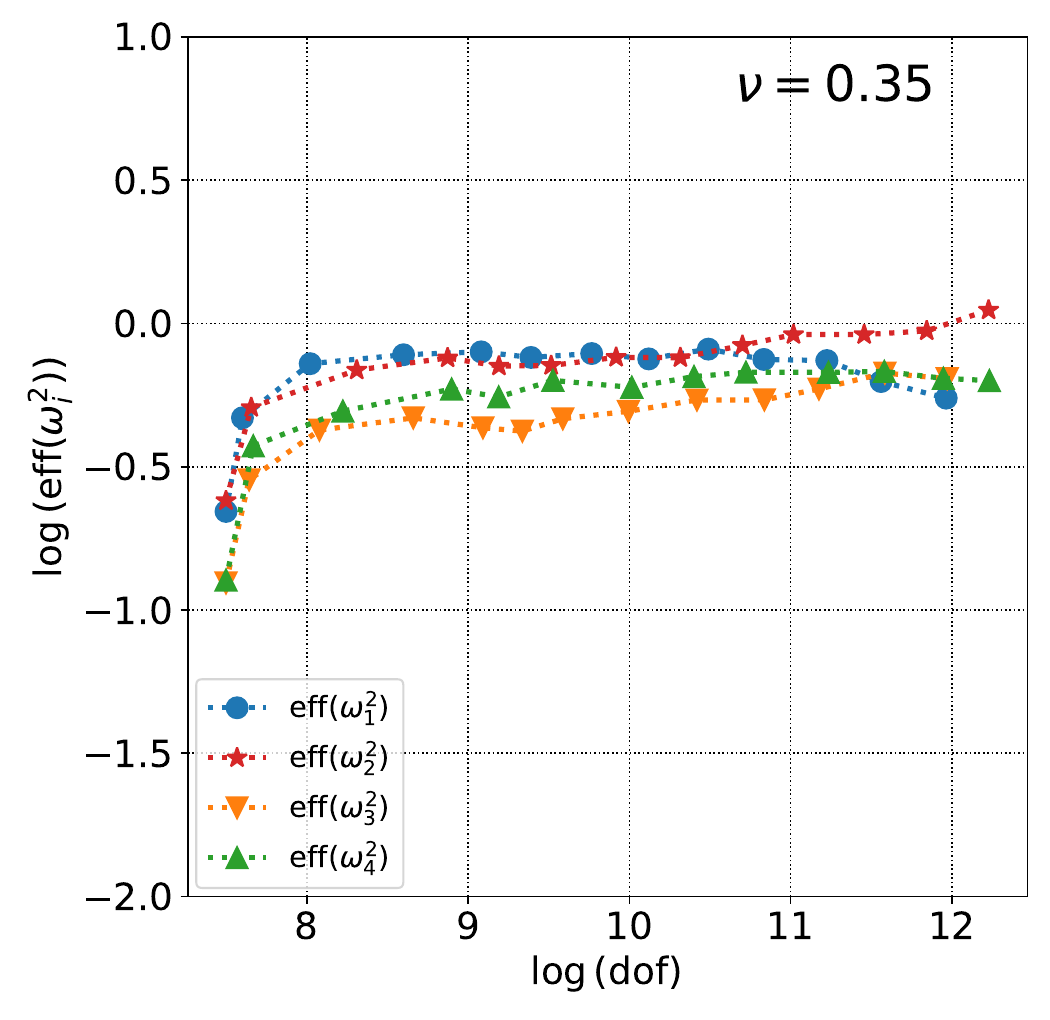}
\end{minipage}
\begin{minipage}{0.32\linewidth}\centering
	\includegraphics[scale=0.26]{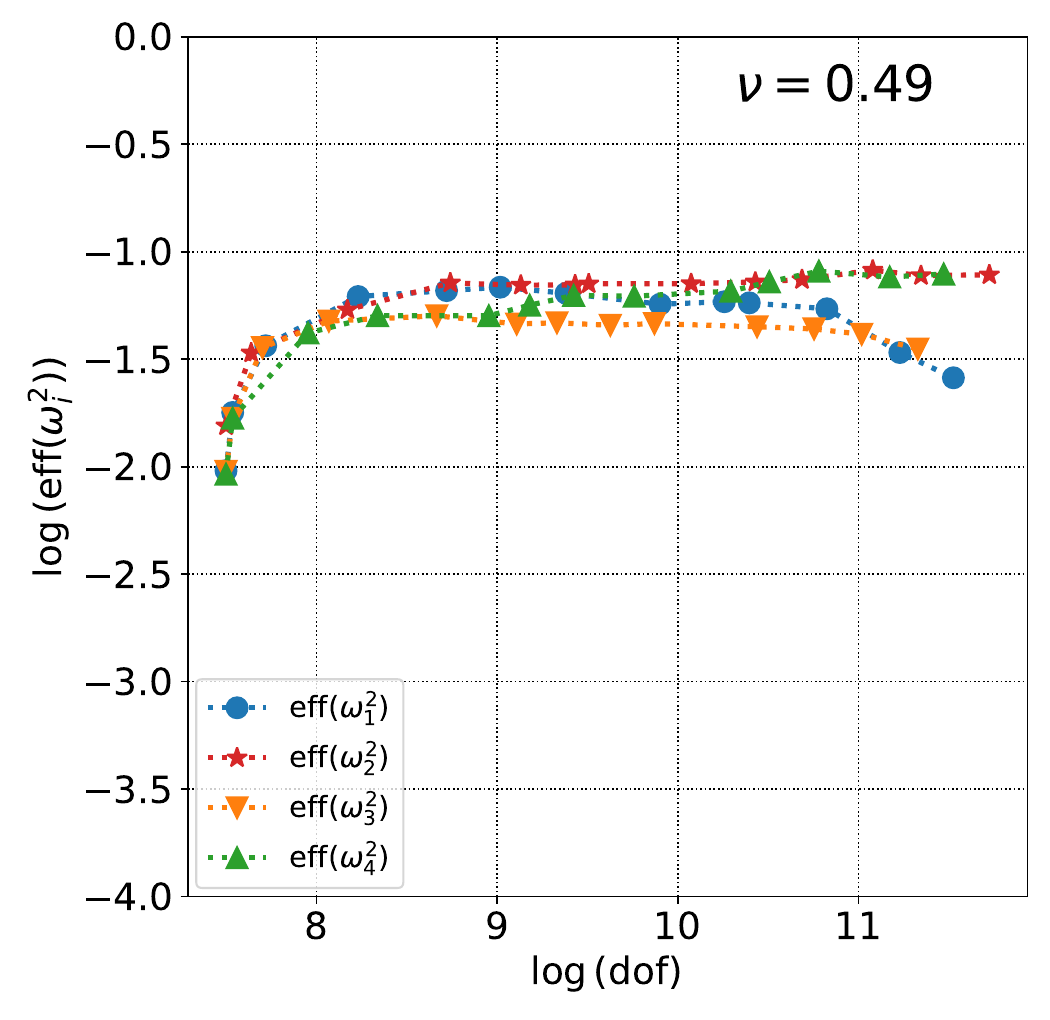}
\end{minipage}
\begin{minipage}{0.32\linewidth}\centering
	\includegraphics[scale=0.26]{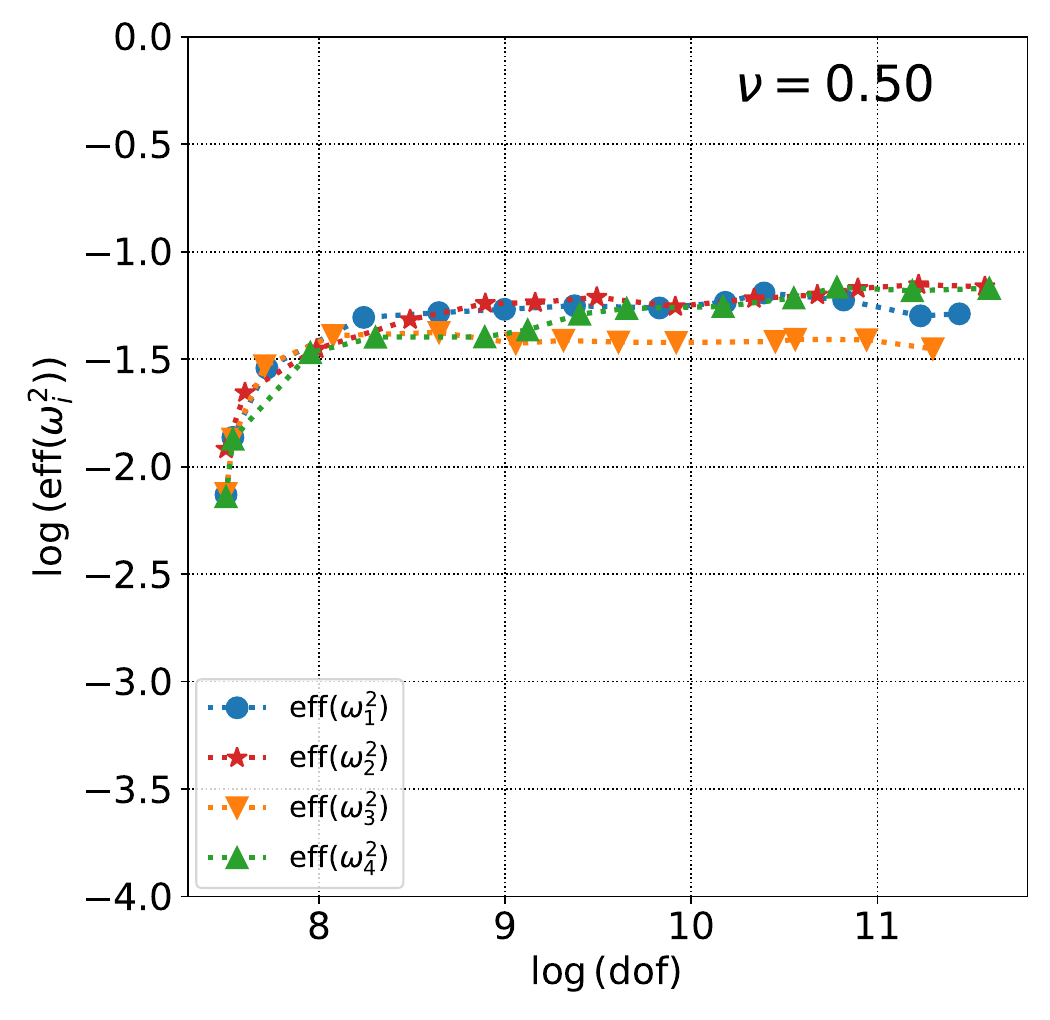}
\end{minipage}
\caption{Example \ref{subsec:afem}. Effectivity indexes for the first four lowest computed frequencies in the adaptive algorithm on $\Omega_{2}$ with different values of $\nu$.}
\label{fig:efficiency-2D-adaptive}
\end{figure}

\subsection{3D example}\label{subsec:3D-afem}
In this experiment we test our estimator on a geometry that has dihedral as well as trihedral singularities. More precisely, we consider the domain defined by $\Omega_{CF}:=\Omega_s\cup\Omega_f$, where
$$
\begin{aligned}
&\Omega_s:=(-1,1)^3,\\ 
&\Omega_f:= (-0.75,-0.3)\times(-0.75,-0.2)\times(-0.75,-0.1)\backslash((-0.5,-0.3)\times(-0.5,-0.2)\times(-0.5,-0.1)).
\end{aligned}
$$
Here, $\Omega_f$ is a Fichera-like domain. The physical parameters are the same as those of Section \ref{subsec:accuracy-test}. A sample of this domain is depicted in Figure \ref{fig:sample-afem-3D}. The refinement level is such that $1/N\approx h$. We assume that $\Omega_{CF}$ is clamped in the $xz$-plane with $y=-1$ and consider $\Gamma_0=\emptyset$.

\begin{figure}[!t]\centering
\includegraphics[scale=0.1,trim=19cm 0cm 19cm 0cm,clip]{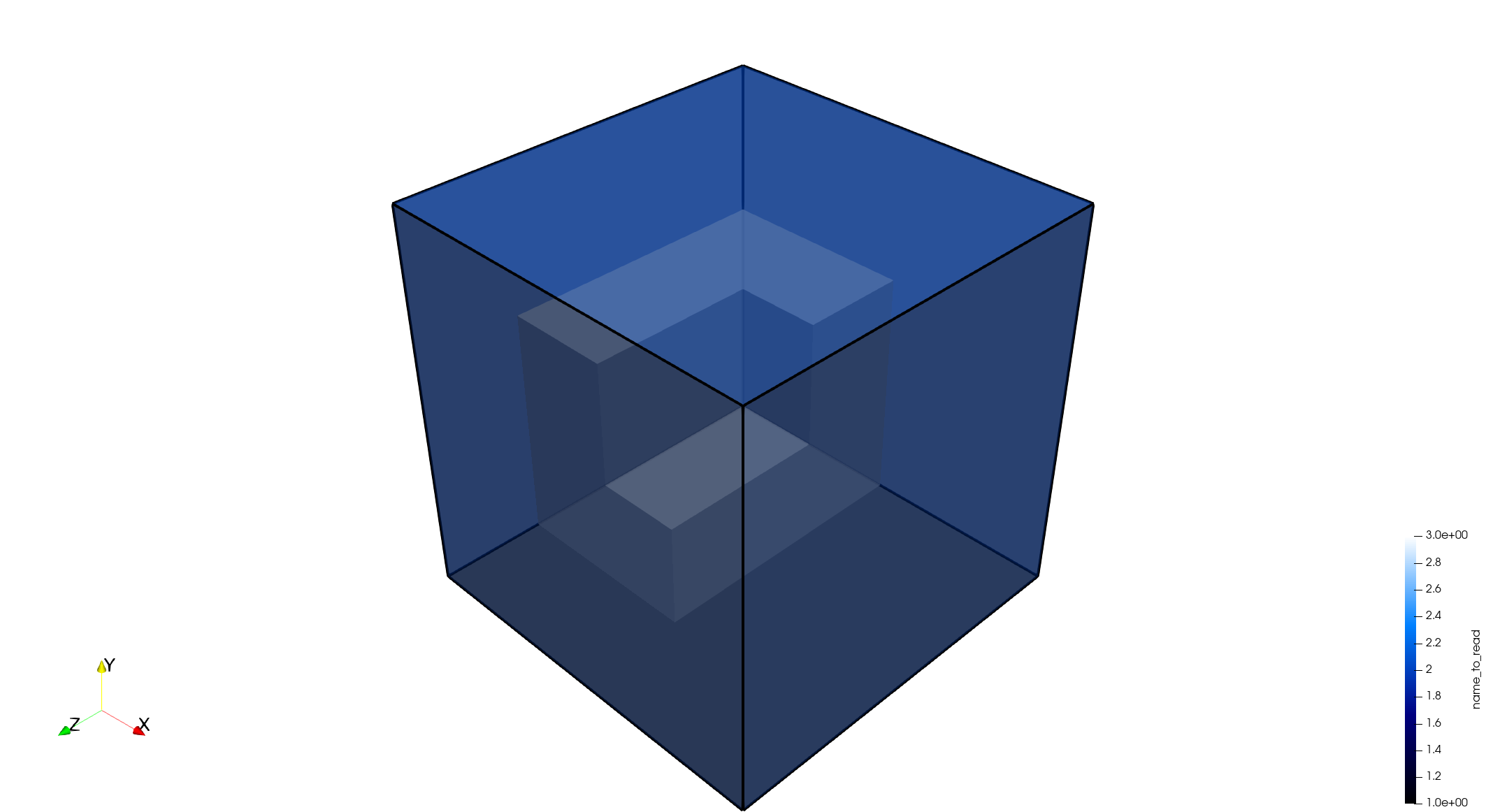}
\caption{Example \ref{subsec:3D-afem}. Sketch of the computational domain $\Omega_{CF}$.}
\label{fig:sample-afem-3D}
\end{figure}

We start by presenting the lowest computed elasto-acoustic frequencies in Table \ref{tabla:results-3D-cube-fichera}. We observe a linear rate of convergence in all the cases, denoting a strong coupling between solid and fluid. This was also observed in \cite{bermudez1994finite} for the two dimensional case. Samples of the eigenmodes are portrayed in Figure \ref{fig:3D-sols}. Here we note the effects of having a bottom clamped cube and a singular fluid domain. The pressure modes shows high gradients near $z=0$ and on the dihedral singularities from the solid domain side. On the fluid part, we observe that the third and fourth modes have high gradients across the Fichera singularities.

\begin{table}[t!]\centering
{\footnotesize\setlength{\tabcolsep}{9.5pt}
		\caption{Example \ref{subsec:3D-afem}. Lowest computed eigenvalues using different combinations of FE families in the cube-fichera domain.}
		\label{tabla:results-3D-cube-fichera}
		\begin{tabular}{|r r r r |c| r| }
			\hline
			\hline
			$N=8$             &  $N=16$         &   $N=24$         & $N=32$ & Order & $\omega_{extr}$  \\ 
			\hline
			\multicolumn{6}{|c|}{$\nu=0.35$}  \\
			\hline
			\multicolumn{6}{|c|}{MINI-element + $\mathbb{BDM}_1$}  \\
			\hline
			2611.7912  &  2533.9313  &  2506.0877  &  2494.1838  & 1.03 &  2468.8218  \\
			2704.2740  &  2646.4223  &  2626.2310  &  2617.1568  & 1.06 &  2599.8235  \\
			3921.0288  &  3804.9109  &  3772.5662  &  3759.9489  & 1.37 &  3741.8883  \\
			6100.5965  &  5981.2481  &  5938.0631  &  5918.0167  & 1.03 &  5879.1927  \\
			\hline
			\multicolumn{6}{|c|}{Taylor--Hood + $\mathbb{BDM}_1$}  \\
			\hline
			2488.4968  &  2477.4893  &  2473.7092  &  2472.1574  & 1.09 &  2468.9218  \\
			2611.8620  &  2604.6753  &  2602.3378  &  2601.3331  & 1.13 &  2599.4235  \\
			3747.8091  &  3742.5131  &  3740.6487  &  3739.9089  & 1.07 &  3738.2751  \\
			5903.2999  &  5888.5302  &  5883.9274  &  5881.8062  & 1.16 &  5878.1927  \\
			\hline
			\multicolumn{6}{|c|}{$\nu=0.49$}  \\
			\hline
			\multicolumn{6}{|c|}{MINI-element + $\mathbb{BDM}_1$}  \\
			\hline
			2648.6718  &  2556.4203  &  2522.2810  &  2507.4712  & 1.01 &  2476.1459  \\
			2724.4364  &  2653.7656  &  2626.9209  &  2614.6971  & 1.04 &  2591.6592  \\
			3755.3424  &  3641.8972  &  3609.7036  &  3597.0934  & 1.24 &  3574.6994  \\
			6175.0254  &  6050.0035  &  5995.0476  &  5970.3920  & 0.96 &  5920.2144  \\
			\hline
			\multicolumn{6}{|c|}{Taylor--Hood + $\mathbb{BDM}_1$}  \\
			\hline
			2497.1383  &  2484.6311  &  2480.2831  &  2478.4651  & 1.07 &  2474.6225  \\
			2605.5115  &  2596.3564  &  2593.3240  &  2591.9988  & 1.11 &  2589.4592  \\
			3584.0209  &  3579.0341  &  3577.2017  &  3576.4690  & 1.01 &  3574.6994  \\
			5943.3247  &  5930.3454  &  5925.9587  &  5923.9550  & 1.06 &  5919.9416  \\
			\hline
			\multicolumn{6}{|c|}{$\nu=0.5$}  \\
			\hline
			\multicolumn{6}{|c|}{MINI-element + $\mathbb{BDM}_1$}  \\
			\hline
			2657.6245  &  2561.6934  &  2526.2274  &  2510.8320  & 1.00 &  2476.0557  \\
			2731.1055  &  2657.5435  &  2629.3030  &  2616.5388  & 0.93 &  2585.4029  \\
			3745.9572  &  3631.9913  &  3599.5220  &  3586.8093  & 1.34 &  3567.8426  \\
			6181.2730  &  6060.4663  &  6003.3623  &  5977.8556  & 1.01 &  5932.4767  \\
			\hline
			\multicolumn{6}{|c|}{Taylor--Hood + $\mathbb{BDM}_1$}  \\
			\hline
			2499.4545  &  2486.7744  &  2482.3658  &  2480.5197  & 1.07 &  2476.6248  \\
			2606.4714  &  2597.1113  &  2594.0098  &  2592.6527  & 1.11 &  2590.0569  \\
			3573.4731  &  3568.4840  &  3566.6458  &  3565.9107  & 1.01 &  3564.1405  \\
			5948.8517  &  5935.9581  &  5931.5641  &  5929.5595  & 1.05 &  5925.4767  \\
			\hline
			\hline
		\end{tabular}
}
\end{table}

Considering the extrapolated values from Table \ref{tabla:results-3D-cube-fichera}, we now perform a total of 12 adaptive refinement for different values of $\nu$. The results on the error and efficiency are depicted in figures \ref{fig:error-3D-adaptive} and \ref{fig:efficiency-3D-adaptive}. We observe that the fourth eigenmode is the one who takes more iterations to converge optimally, and it is the one with less marked elements. This behavior increases when we approach to the incompressible limit. In all the cases, optimal rates of $\mathcal{O}(h^2)\approx \mathcal{O}(\texttt{dof}^{-1})$ is observed for some $h_0$. Also, the estimator remains reliable and efficient.

Sample of the adaptive meshes for $\nu=0.35$ are displayed in Figure \ref{fig:sample-afem-3D}. Here we present a bottom view of the solid domain in order to observe the refinements on the edges where boundary conditions change. It is clear that the algorithm detects and refine near the singularities according to the computed eigenmode. The pressure modes in Figure \ref{fig:3D-sols} gives a clue of the zones to be refined. It notes that a dihedral singularity for the solid is not necessary a singularity for the fluid.

\begin{figure}[!t]\centering
\begin{minipage}{0.24\linewidth}\centering
	{\footnotesize $\bu_{h,1}$}\\
	\includegraphics[scale=0.093,trim=21cm 0cm 16cm 0cm,clip]{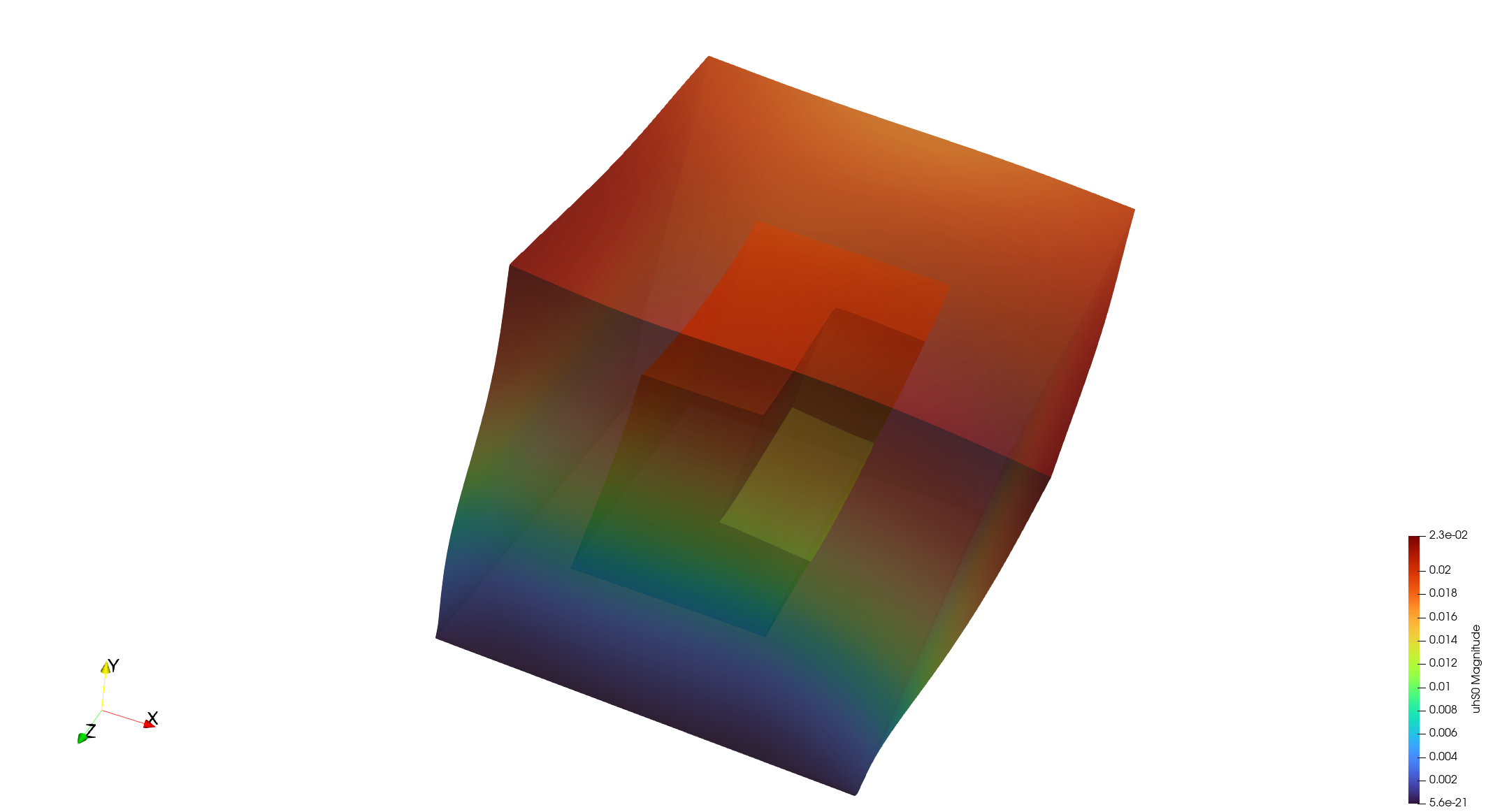}
\end{minipage}
\begin{minipage}{0.24\linewidth}\centering
	{\footnotesize $\bu_{h,2}$}\\
	\includegraphics[scale=0.093,trim=17cm 0cm 20cm 0cm,clip]{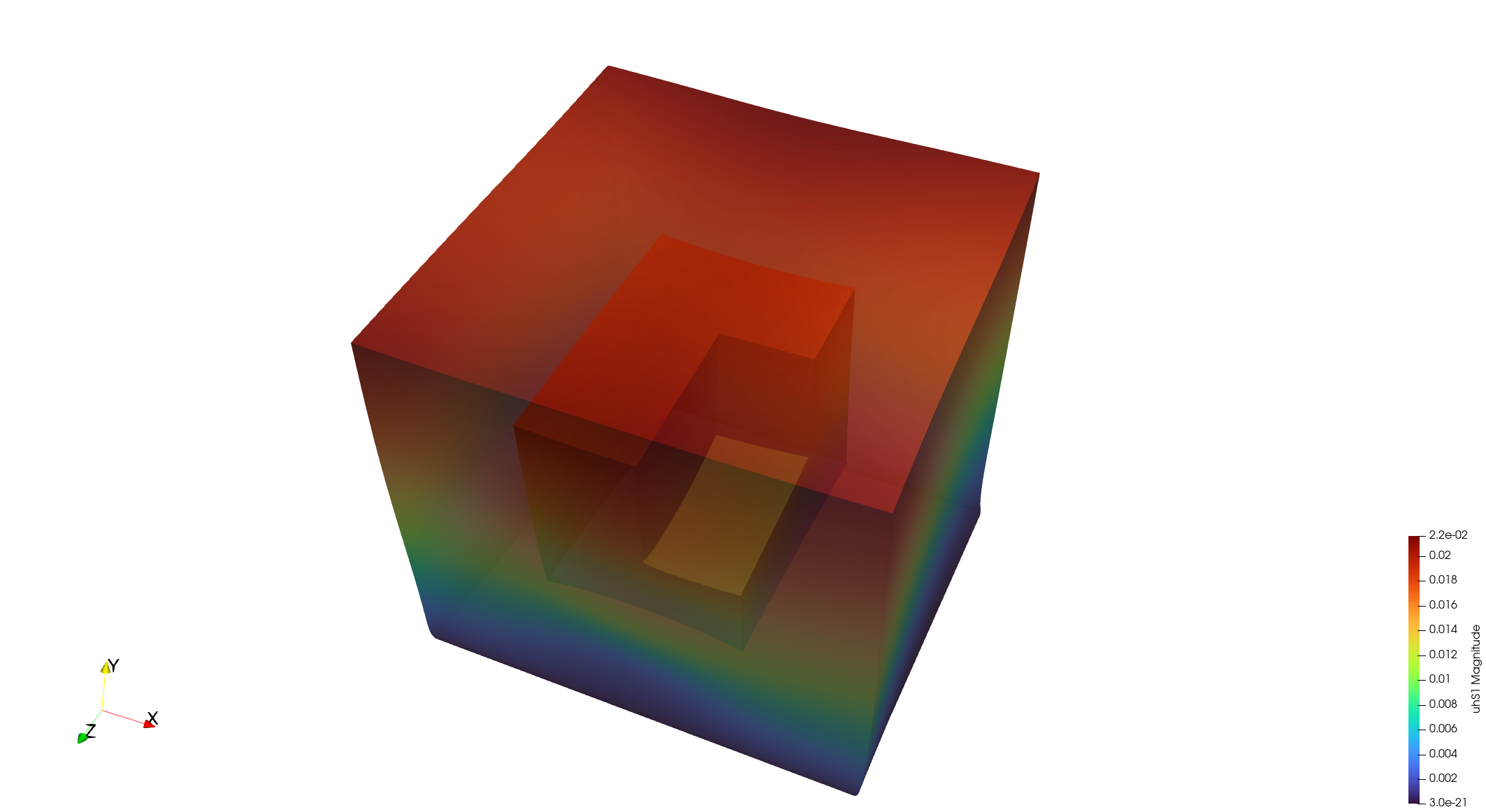}
\end{minipage}
\begin{minipage}{0.24\linewidth}\centering
	{\footnotesize $\bu_{h,3}$}\\
	\includegraphics[scale=0.093,trim=18cm 0cm 16cm 0cm,clip]{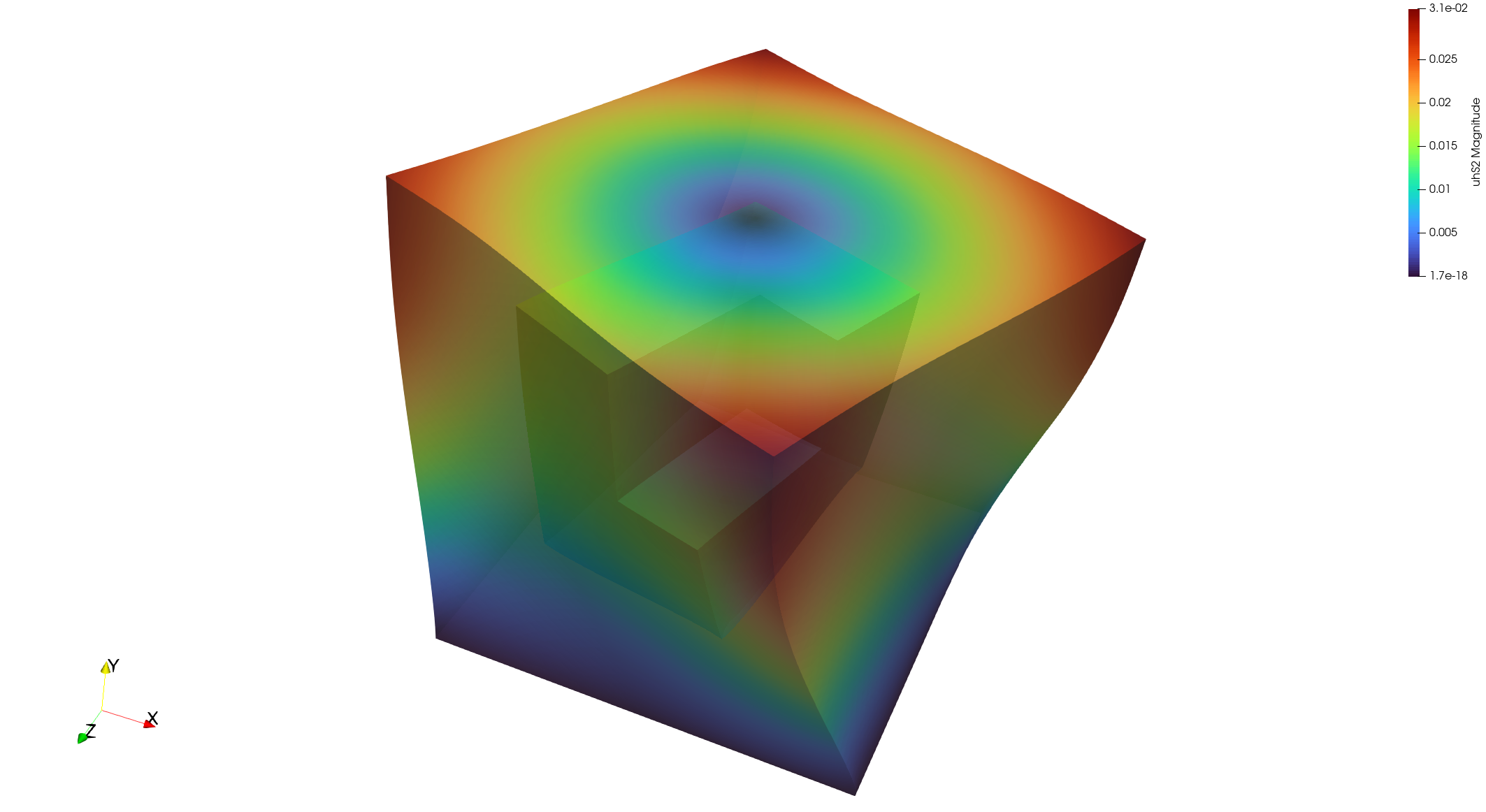}
\end{minipage}
\begin{minipage}{0.24\linewidth}\centering
	{\footnotesize $\bu_{h,4}$}\\
	\includegraphics[scale=0.093,trim=19cm 0cm 18cm 0cm,clip]{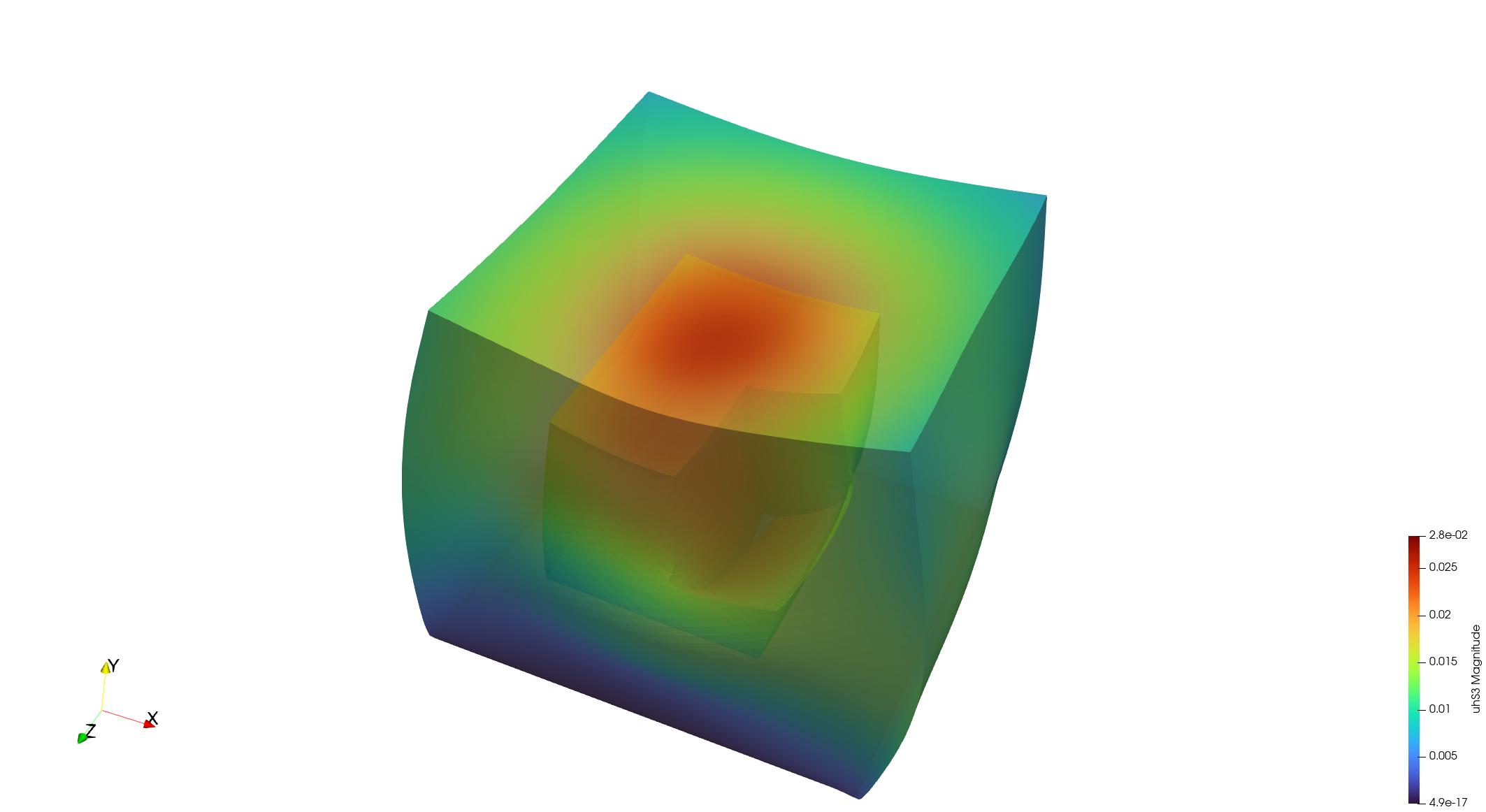}
\end{minipage}
\begin{minipage}{0.24\linewidth}\centering
	{\footnotesize $p_{h,1}$}\\
	\includegraphics[scale=0.093,trim=19cm 0cm 19cm 0cm,clip]{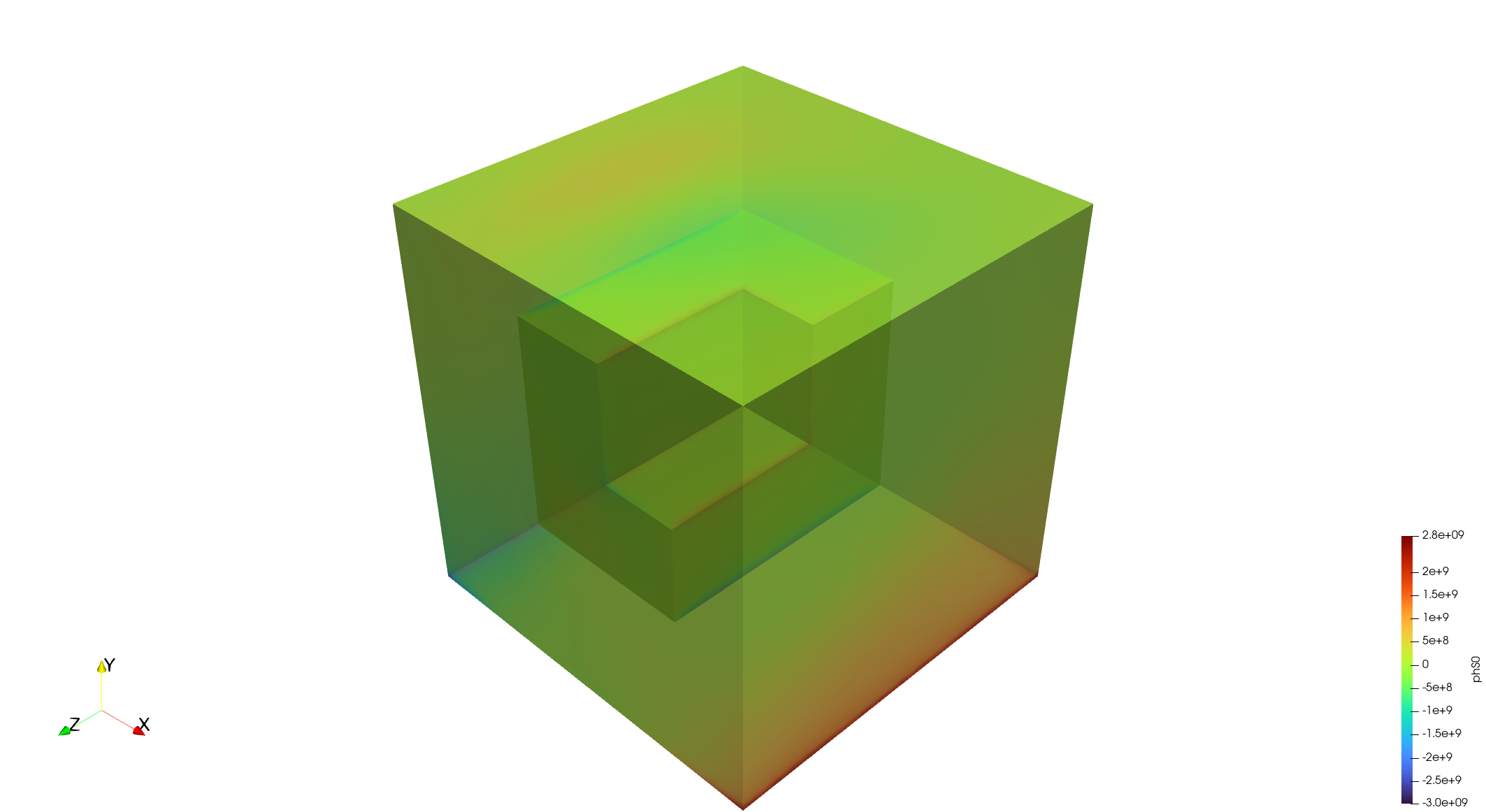}
\end{minipage}
\begin{minipage}{0.24\linewidth}\centering
	{\footnotesize $p_{h,2}$}\\
	\includegraphics[scale=0.093,trim=19cm 0cm 19cm 0cm,clip]{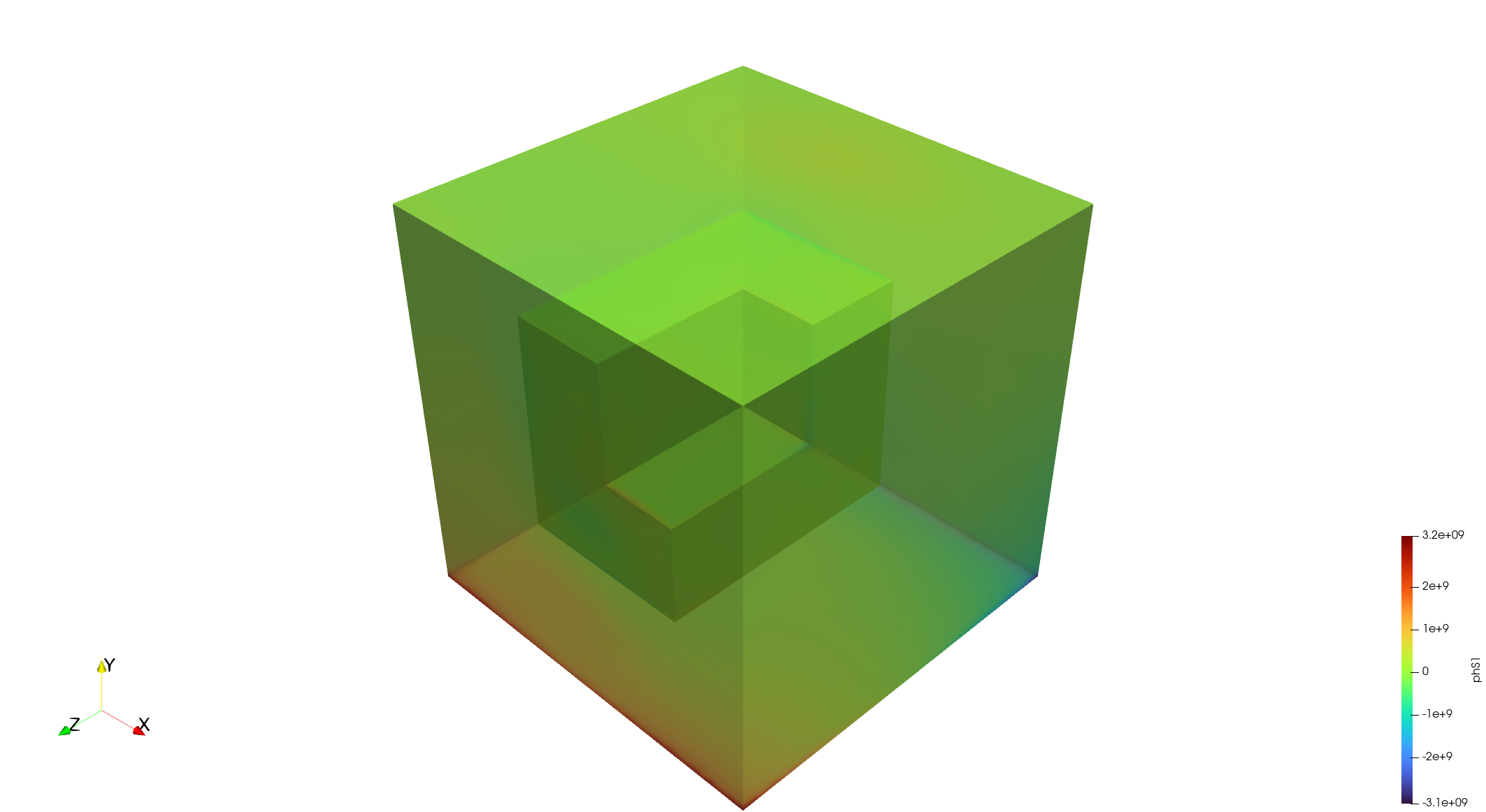}
\end{minipage}
\begin{minipage}{0.24\linewidth}\centering
	{\footnotesize $p_{h,3}$}\\
	\includegraphics[scale=0.093,trim=19cm 0cm 19cm 0cm,clip]{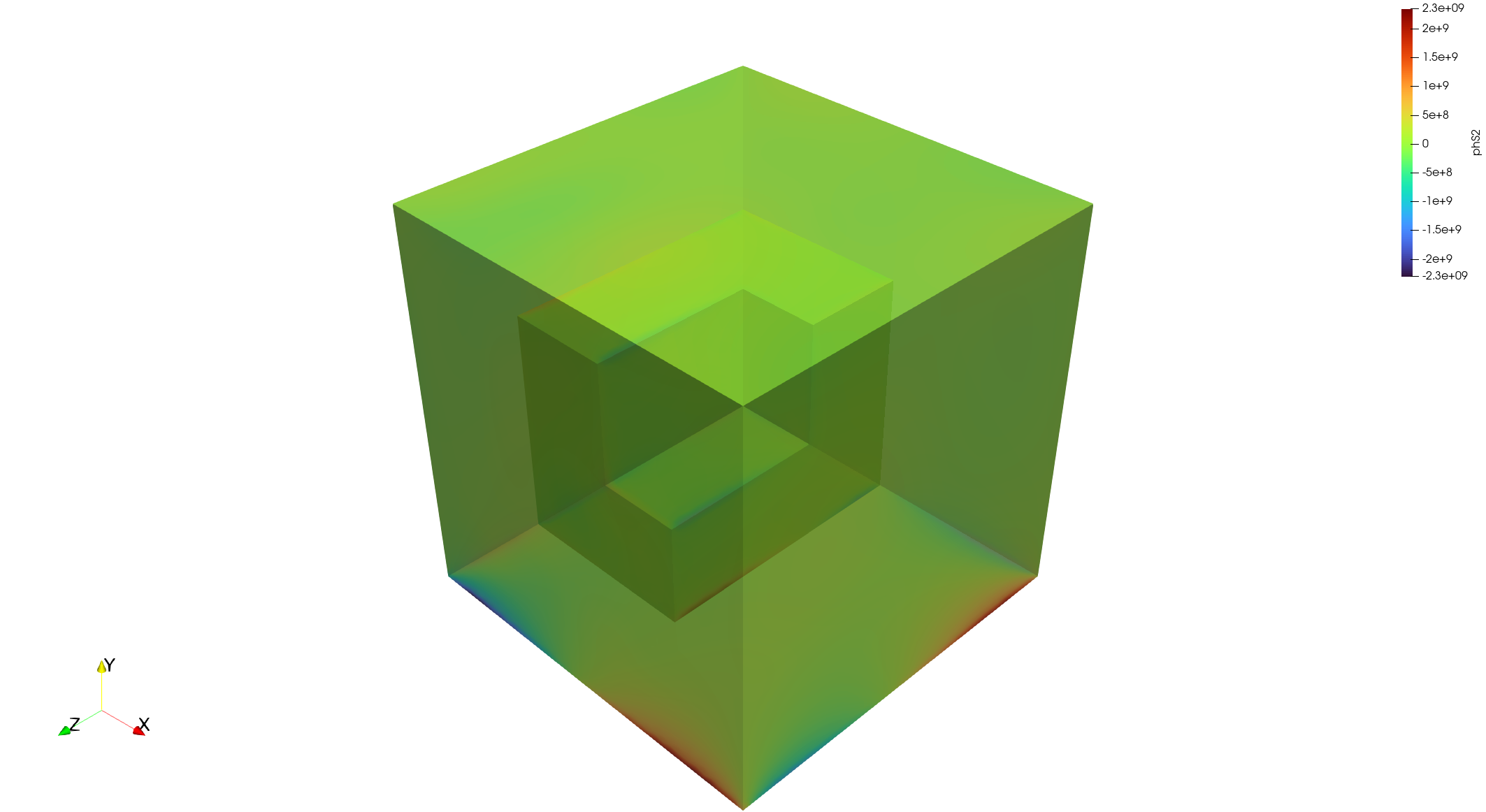}
\end{minipage}
\begin{minipage}{0.24\linewidth}\centering
	{\footnotesize $p_{h,4}$}\\
	\includegraphics[scale=0.093,trim=19cm 0cm 19cm 0cm,clip]{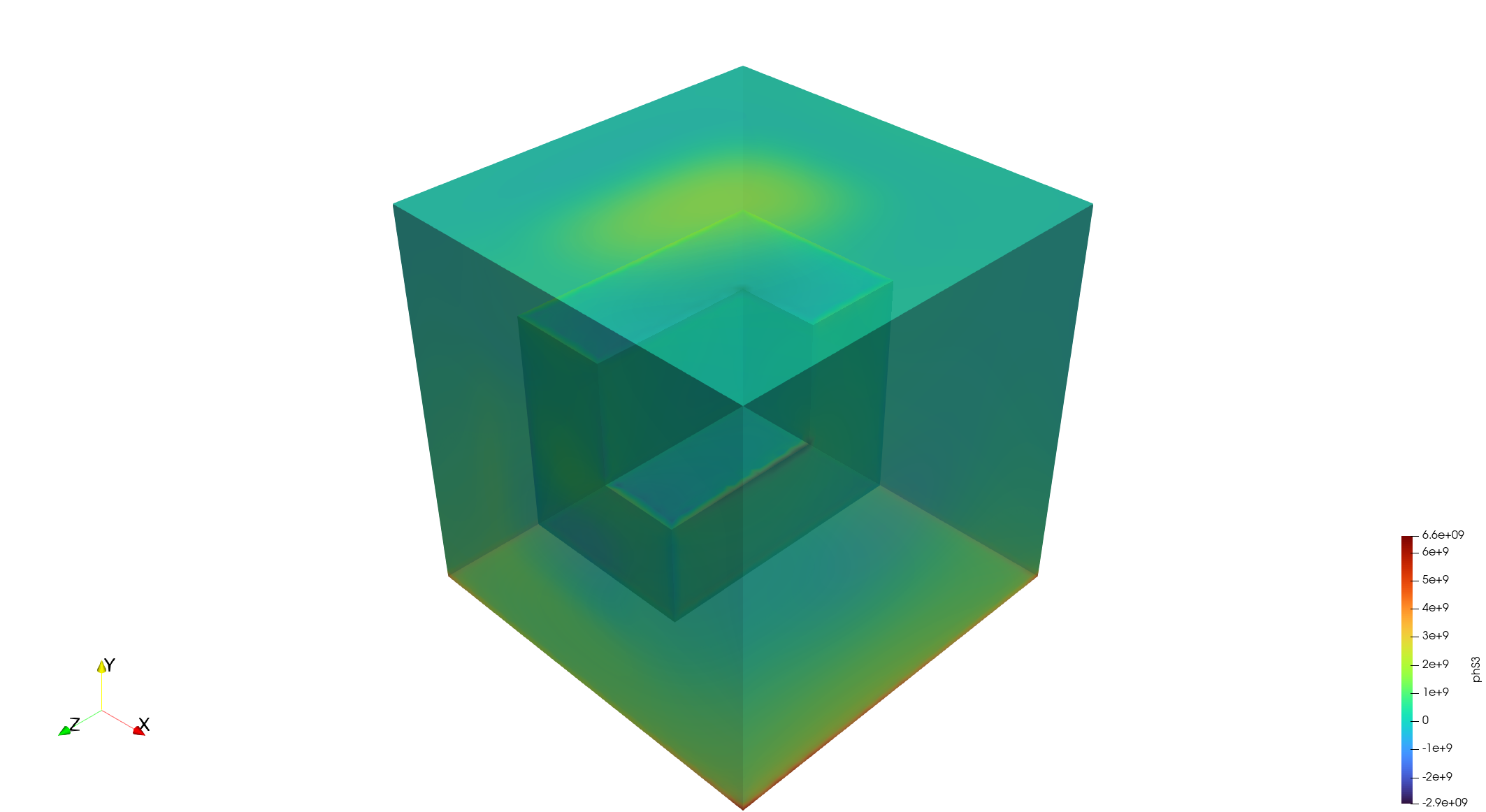}
\end{minipage}
\begin{minipage}{0.24\linewidth}\centering
	{\footnotesize $\bw_{h,1}$}\\
	\includegraphics[scale=0.093,trim=19cm 0cm 19cm 0cm,clip]{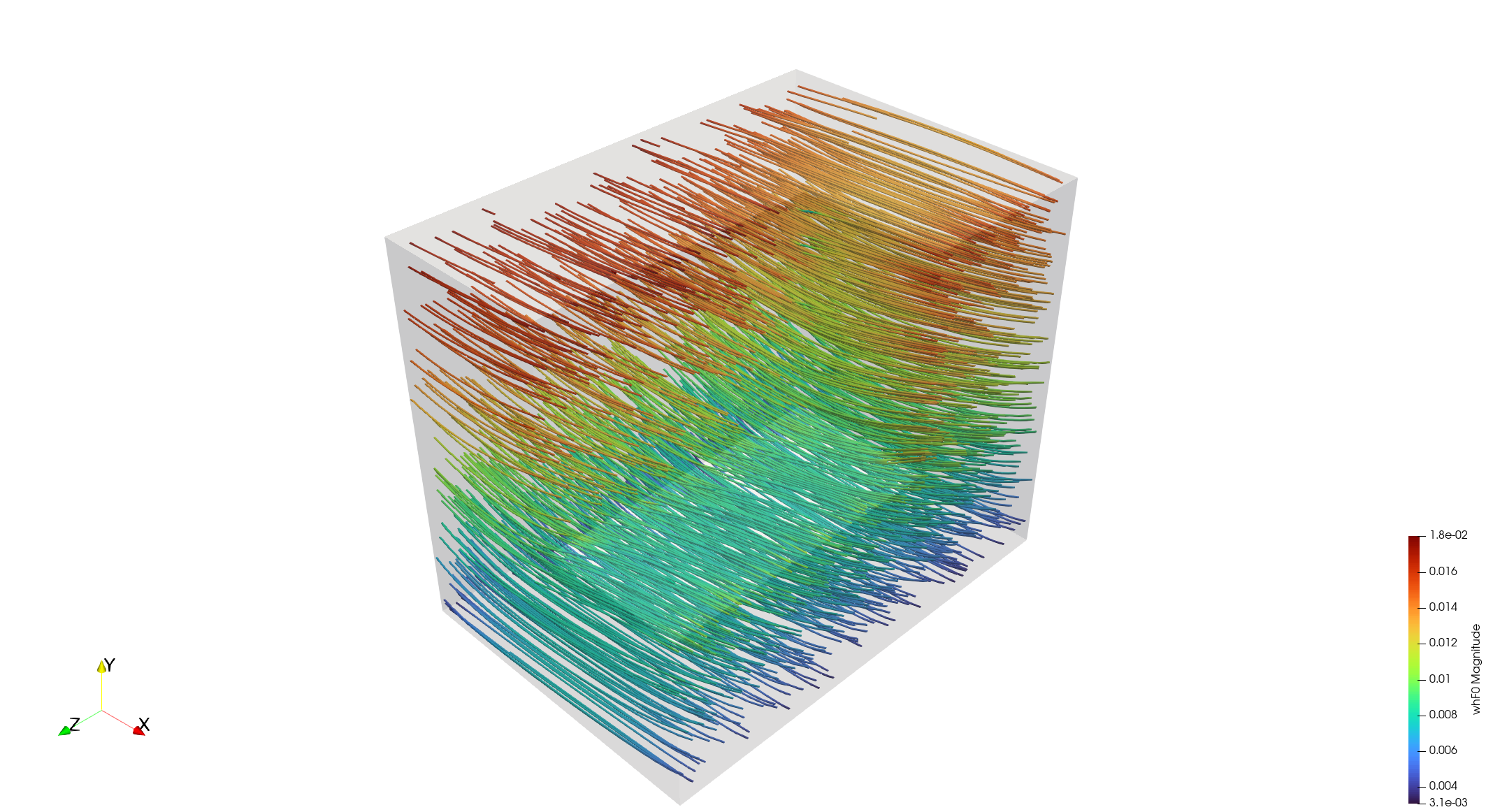}
\end{minipage}
\begin{minipage}{0.24\linewidth}\centering
	{\footnotesize $\bw_{h,2}$}\\
	\includegraphics[scale=0.093,trim=19cm 0cm 19cm 0cm,clip]{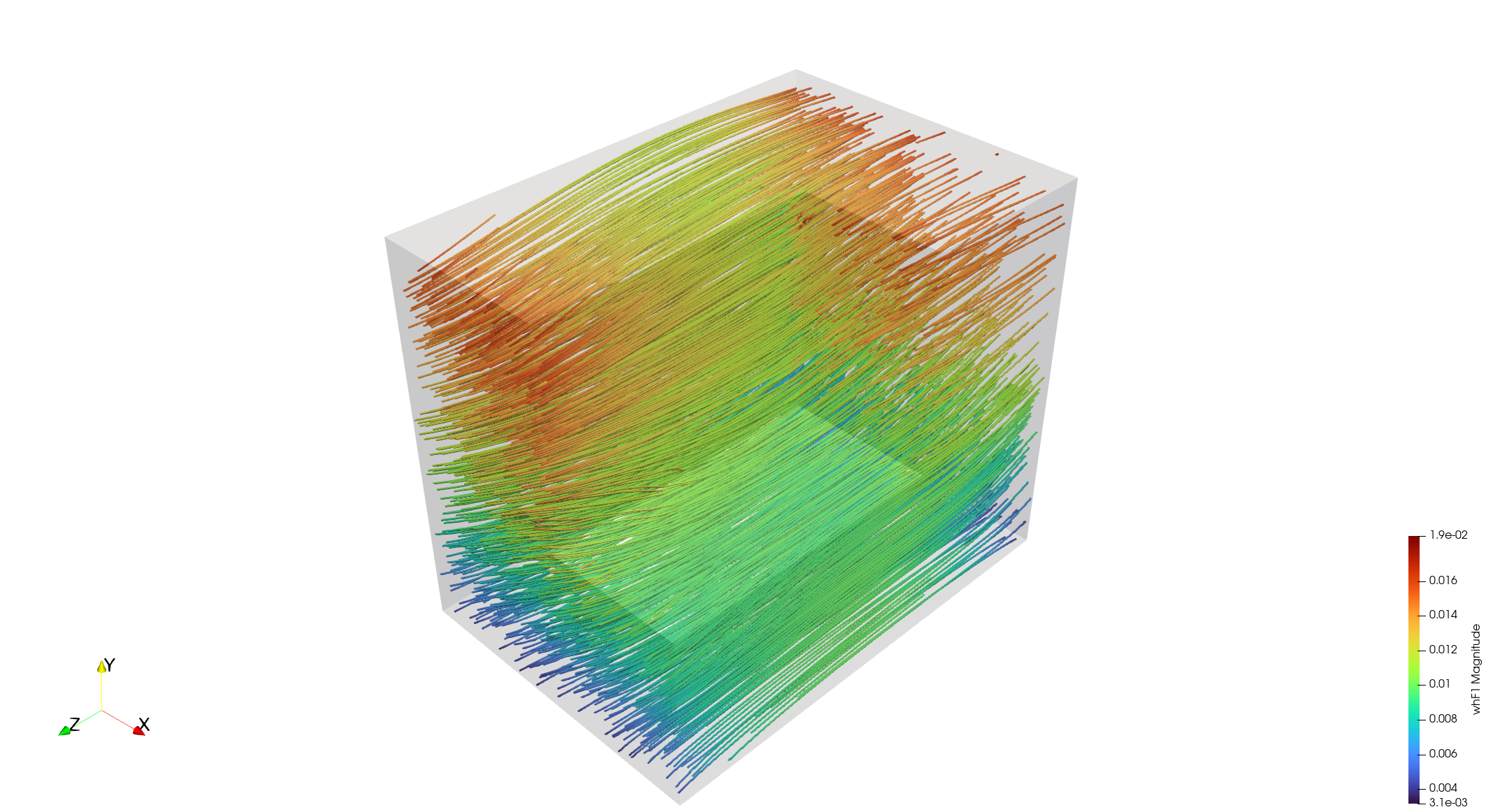}
\end{minipage}
\begin{minipage}{0.24\linewidth}\centering
	{\footnotesize $\bw_{h,3}$}\\
	\includegraphics[scale=0.093,trim=19cm 0cm 19cm 0cm,clip]{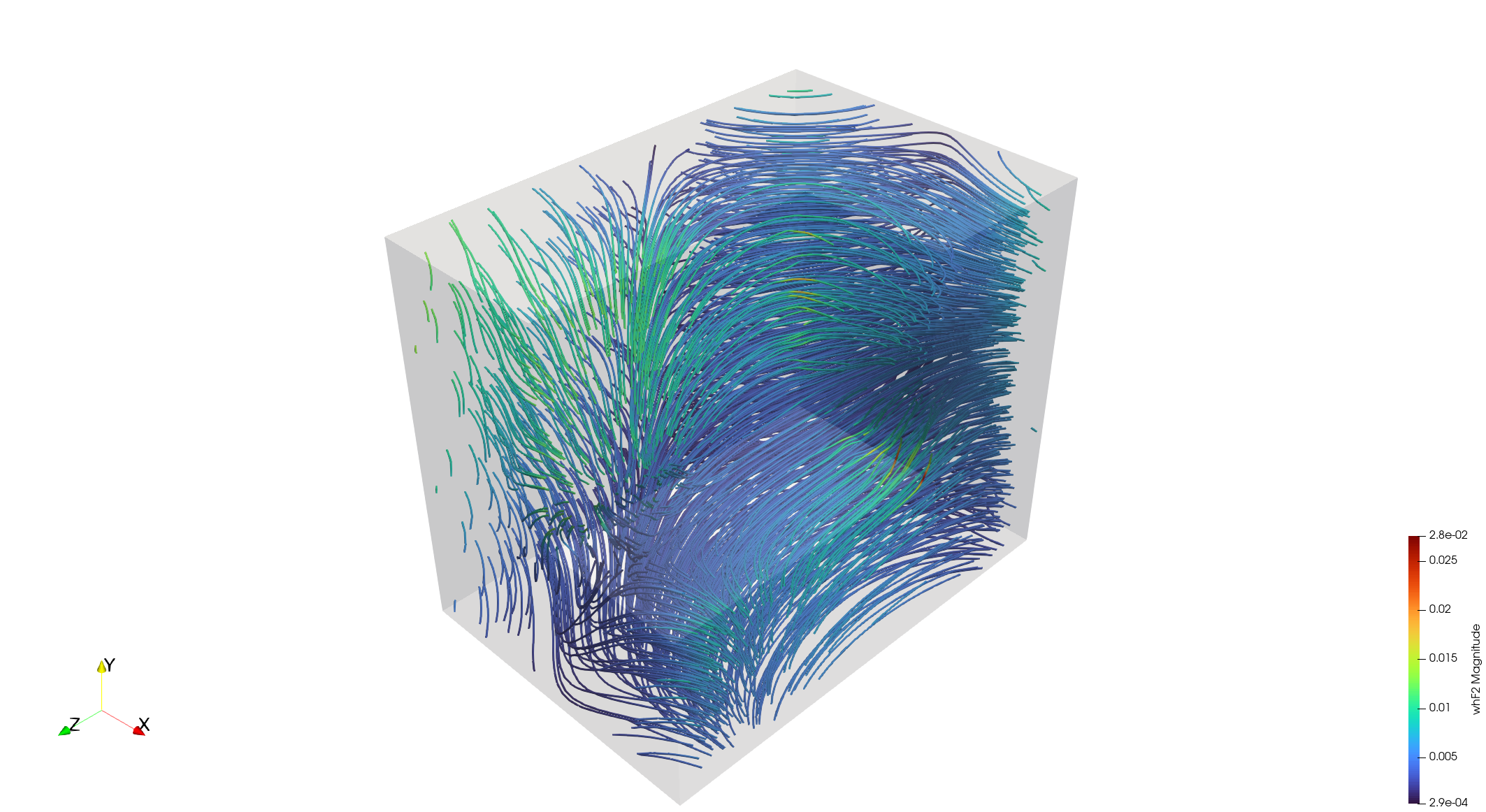}
\end{minipage}
\begin{minipage}{0.24\linewidth}\centering
	{\footnotesize $\bw_{h,4}$}\\
	\includegraphics[scale=0.093,trim=19cm 0cm 19cm 0cm,clip]{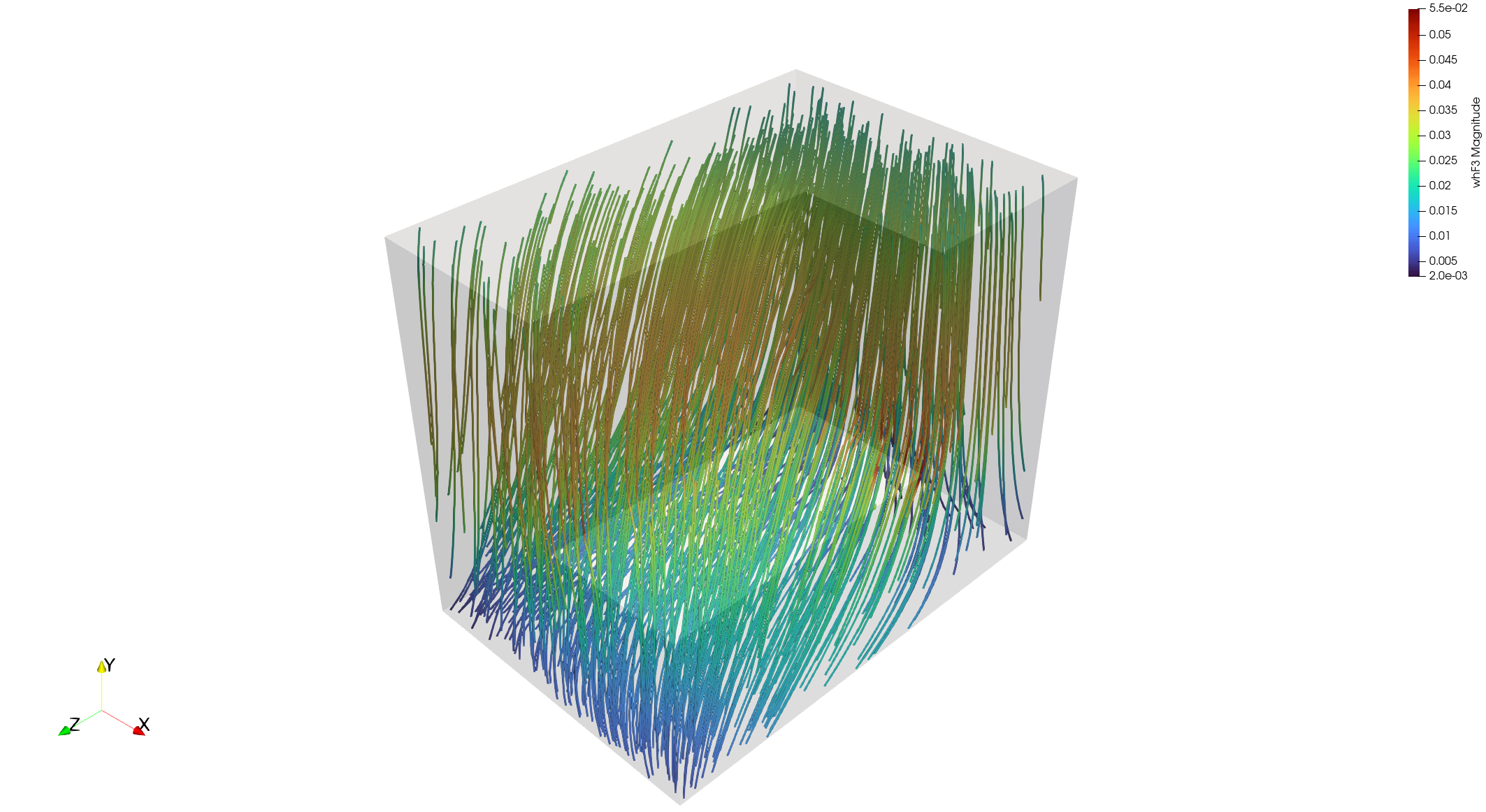}
\end{minipage}
\caption{Example \ref{subsec:3D-afem}. Comparison between the first fourth lowest order computed elasto-acoustic modes on $\Omega_{CF}$. The solid domain have been warped by a sufficiently large factor in order to observe the deformation.}
\end{figure}

\begin{figure}[!t]\centering
\begin{minipage}{0.24\linewidth}\centering
	{\text{$\Omega_s$, $\omega_{h,1}$, \texttt{dof}= 128136}\\ }
	\includegraphics[scale=0.09,trim=18cm 0cm 18cm 0cm,clip]{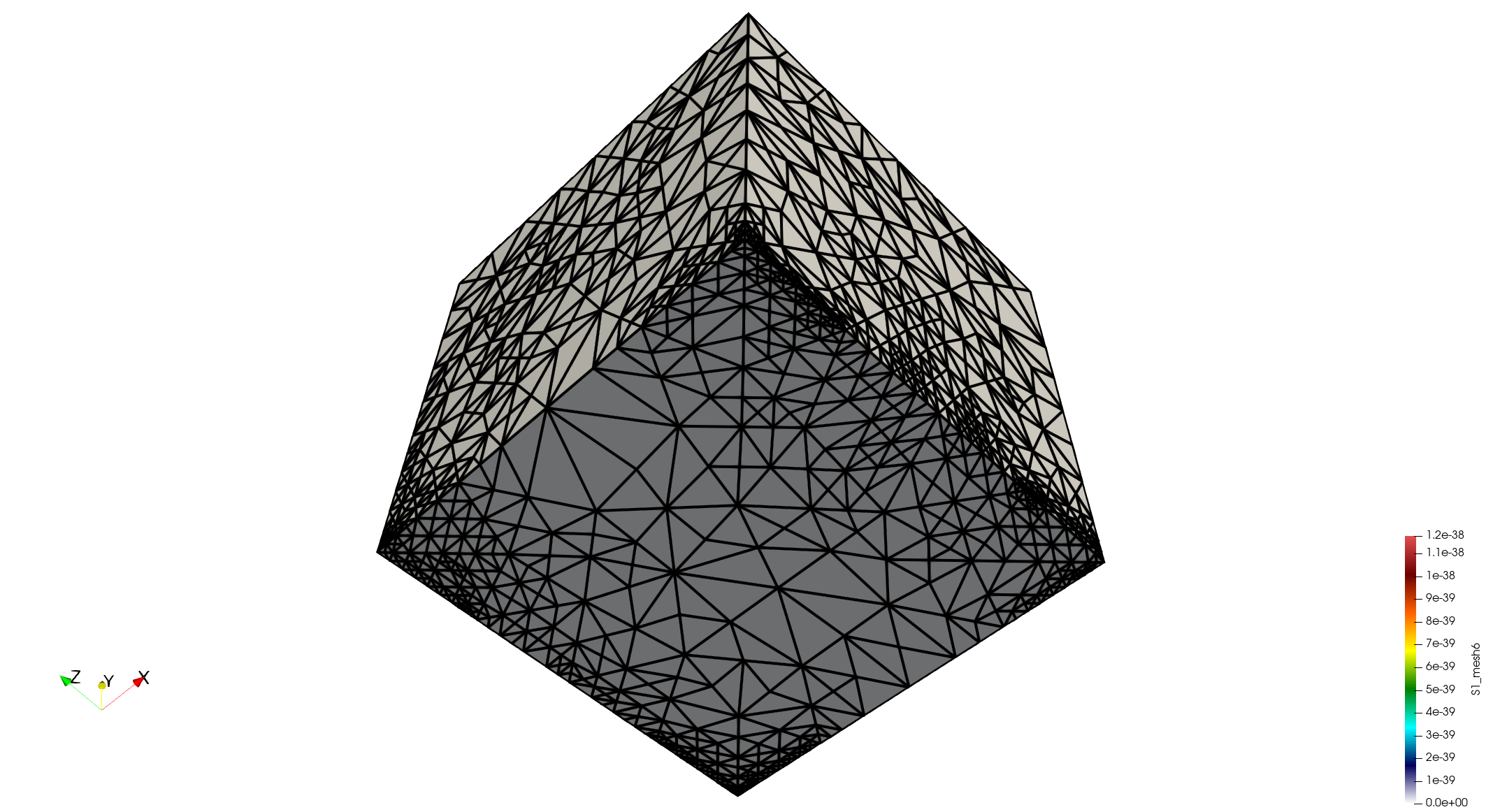}
\end{minipage}
\begin{minipage}{0.24\linewidth}\centering
	{\text{$\Omega_f$, $\omega_{h,1}$, \texttt{dof}= 128136}\\ }
	\includegraphics[scale=0.09,trim=18cm 0cm 18cm 0cm,clip]{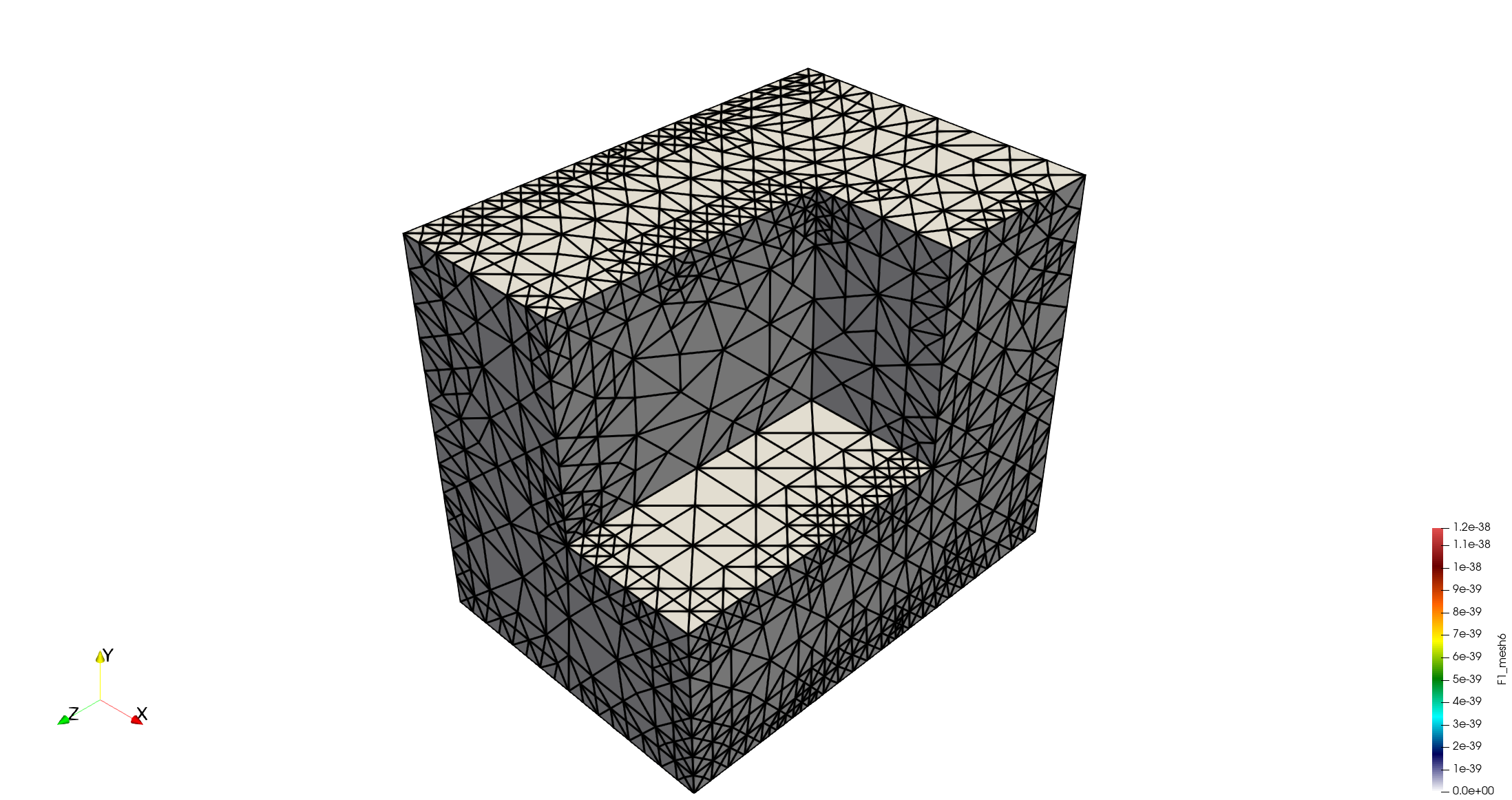}
\end{minipage}
\begin{minipage}{0.24\linewidth}\centering
	{\text{$\Omega_s$, $\omega_{h,1}$, \texttt{dof}= 1736770}\\ }
	\includegraphics[scale=0.09,trim=18cm 0cm 18cm 0cm,clip]{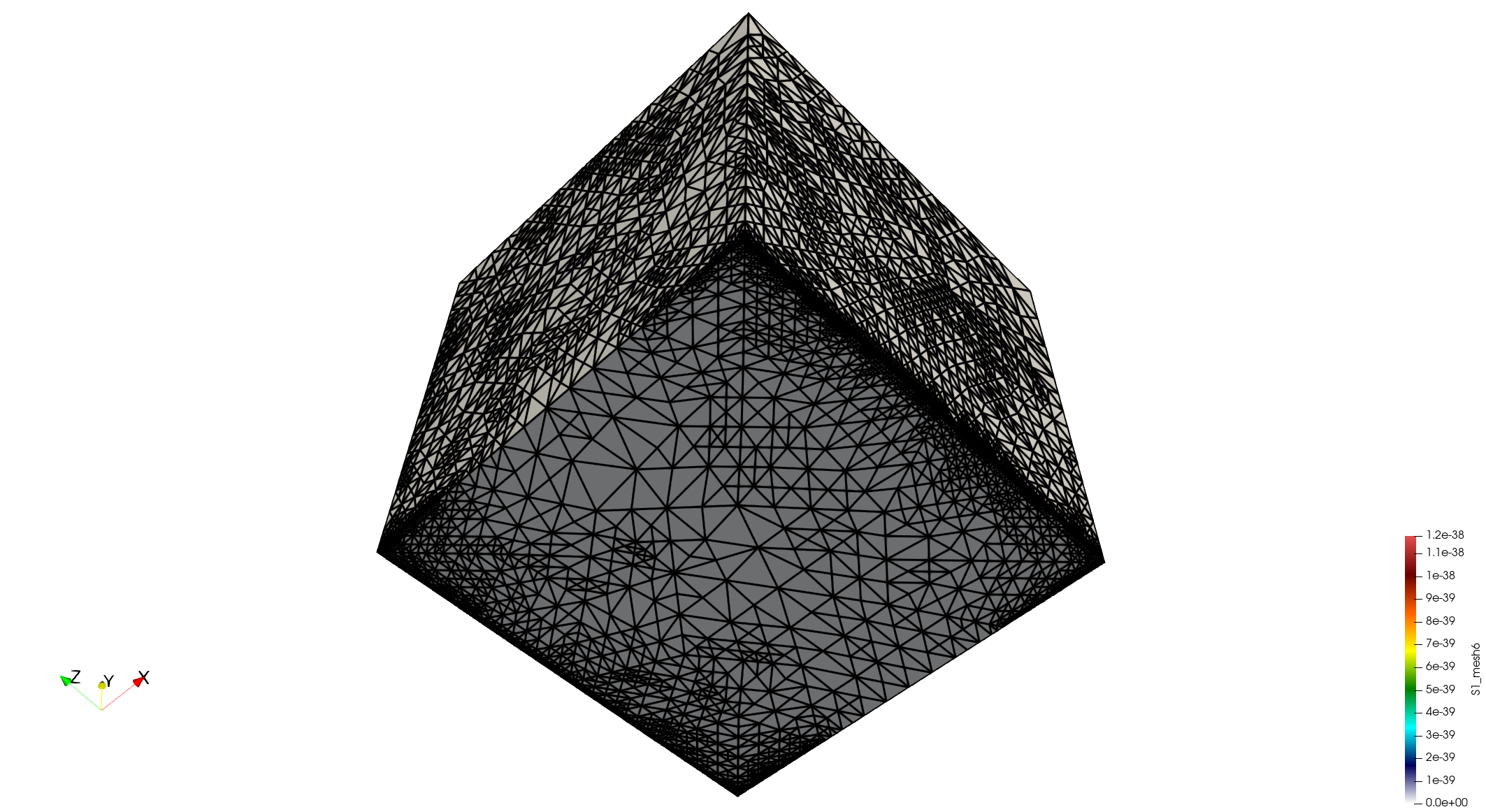}
\end{minipage}
\begin{minipage}{0.24\linewidth}\centering
	{\text{$\Omega_f$, $\omega_{h,1}$, \texttt{dof}= 1736770}\\ }
	\includegraphics[scale=0.09,trim=18cm 0cm 18cm 0cm,clip]{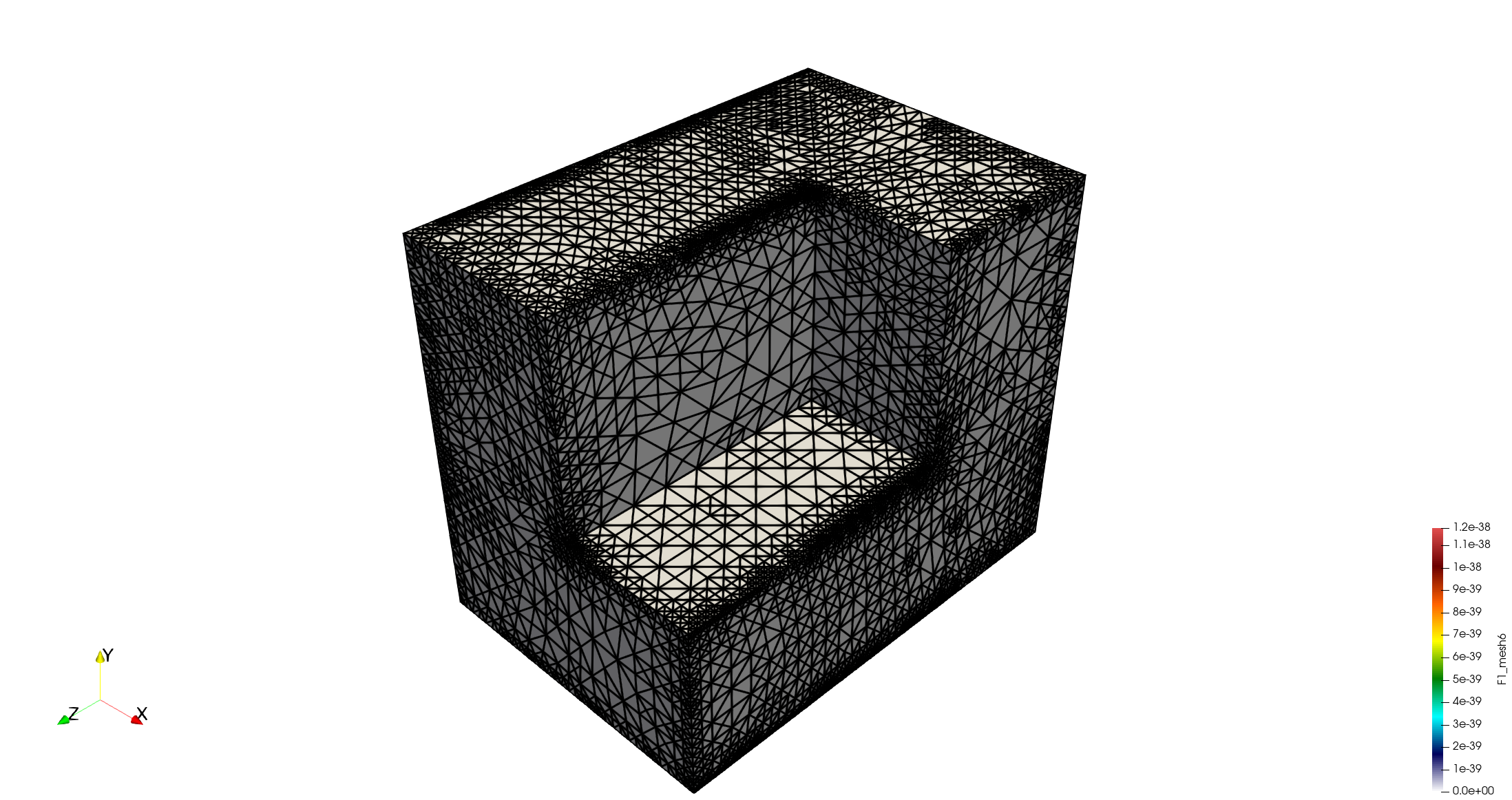}
\end{minipage}
\begin{minipage}{0.24\linewidth}\centering
	{\text{$\Omega_s$, $\omega_{h,3}$, \texttt{dof}= 331900}\\ }
	\includegraphics[scale=0.09,trim=18cm 0cm 18cm 0cm,clip]{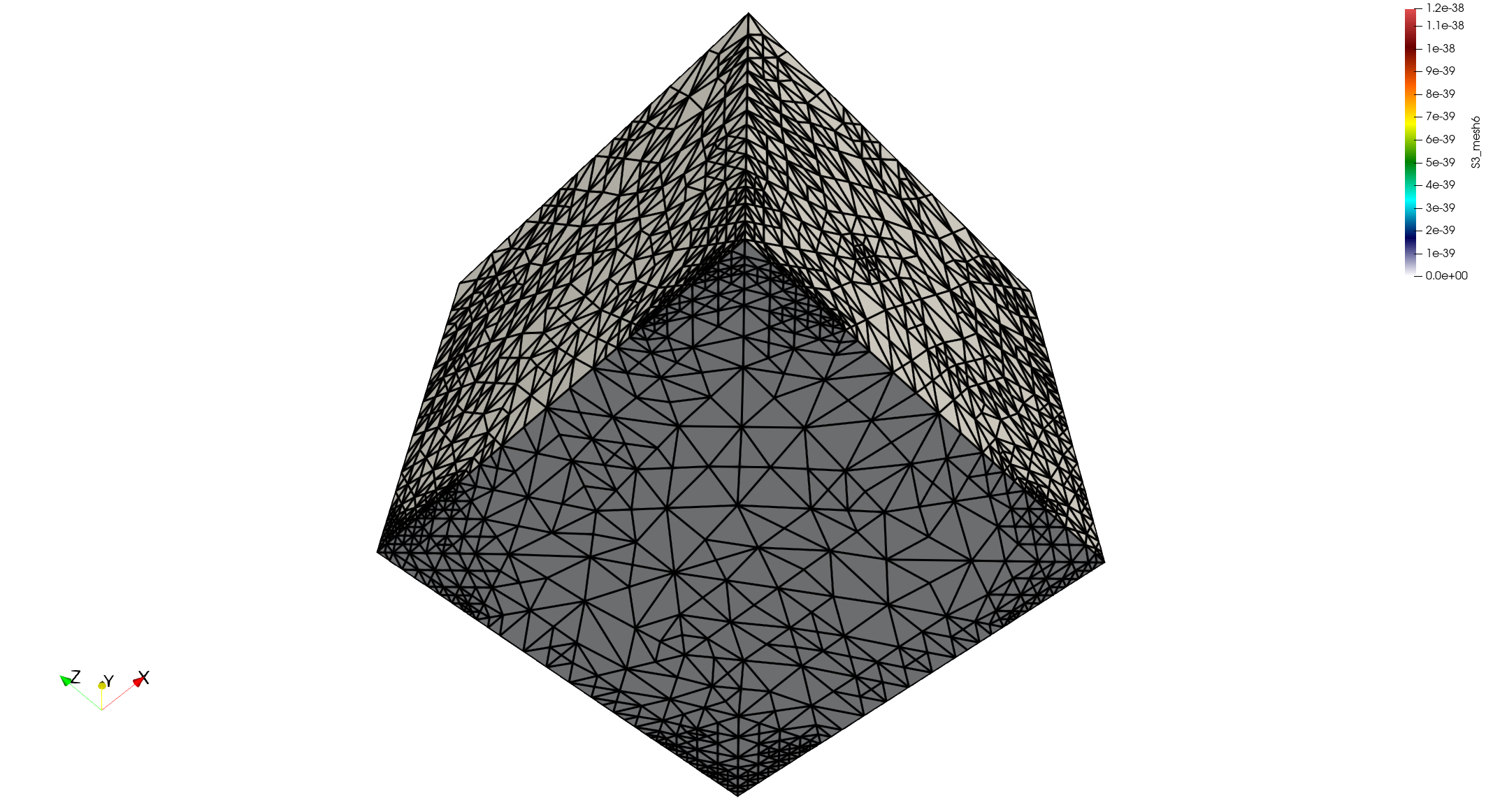}
\end{minipage}
\begin{minipage}{0.24\linewidth}\centering
	{\text{$\Omega_f$, $\omega_{h,3}$, \texttt{dof}= 331900}\\ }
	\includegraphics[scale=0.09,trim=18cm 0cm 18cm 0cm,clip]{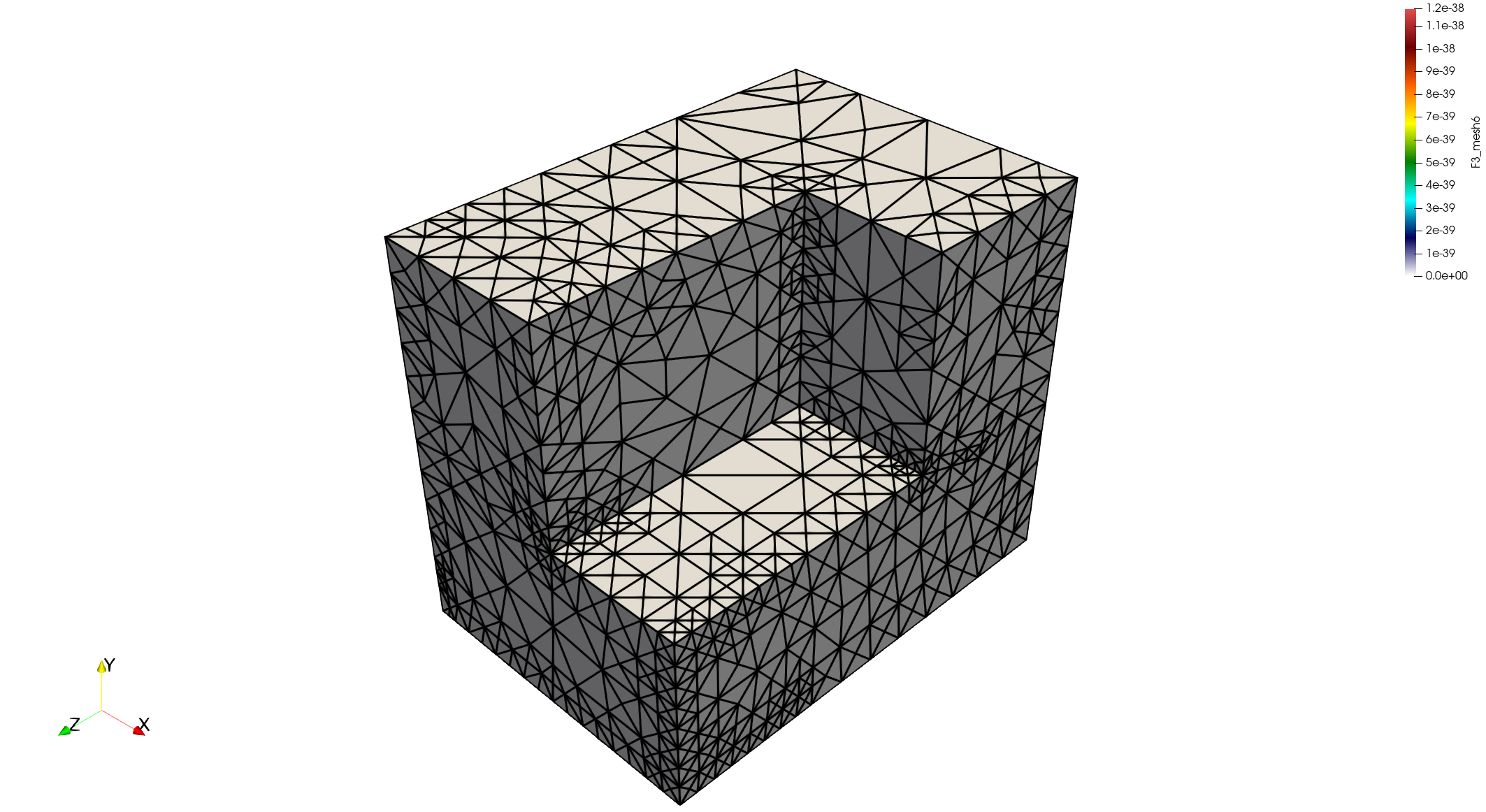}
\end{minipage}
\begin{minipage}{0.24\linewidth}\centering
	{\text{$\Omega_s$, $\omega_{h,3}$, \texttt{dof}= 1566206}\\ }
	\includegraphics[scale=0.09,trim=18cm 0cm 18cm 0cm,clip]{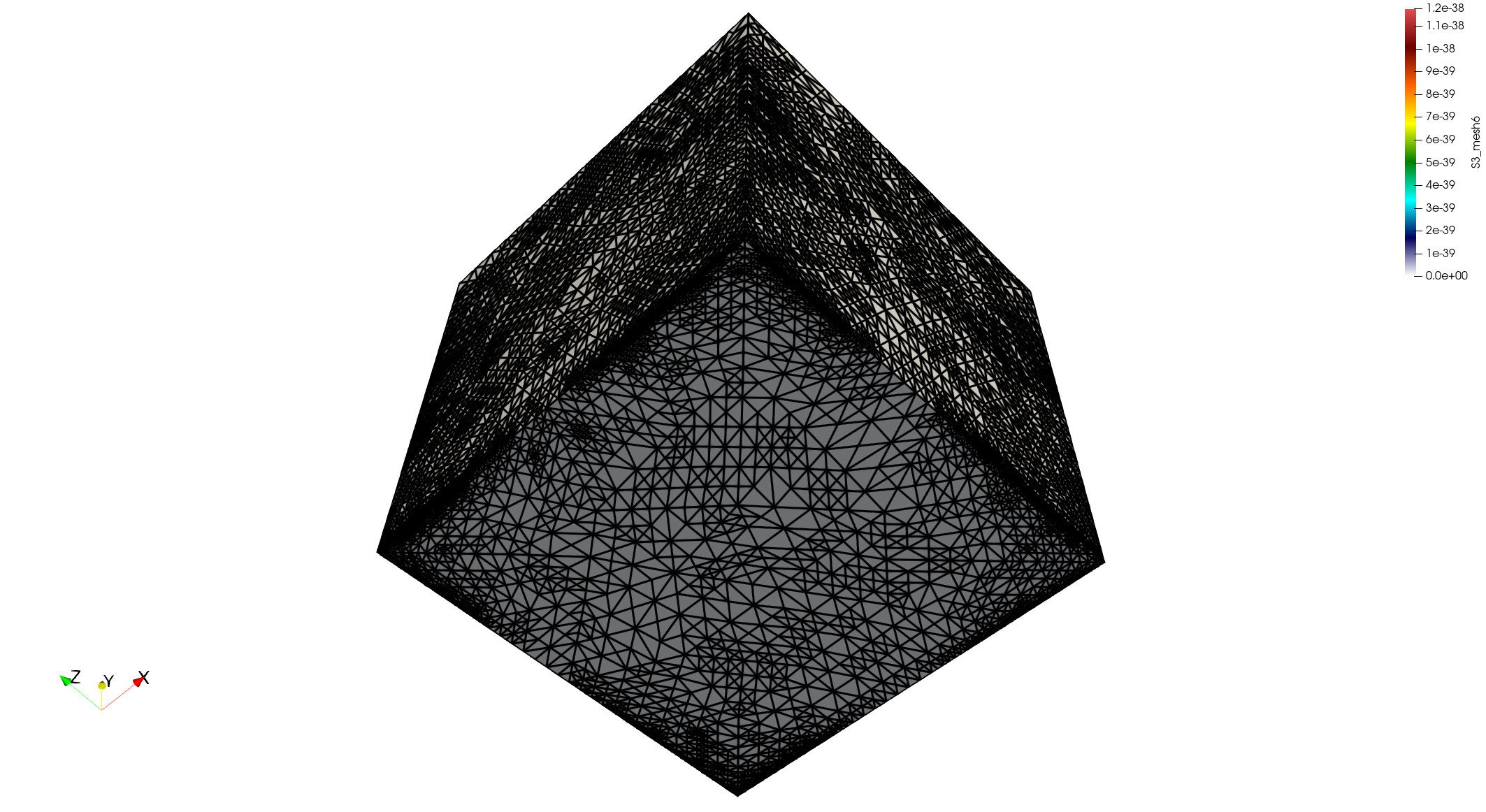}
\end{minipage}
\begin{minipage}{0.24\linewidth}\centering
	{\text{$\Omega_f$, $\omega_{h,3}$, \texttt{dof}= 1566206}\\ }
	\includegraphics[scale=0.09,trim=18cm 0cm 18cm 0cm,clip]{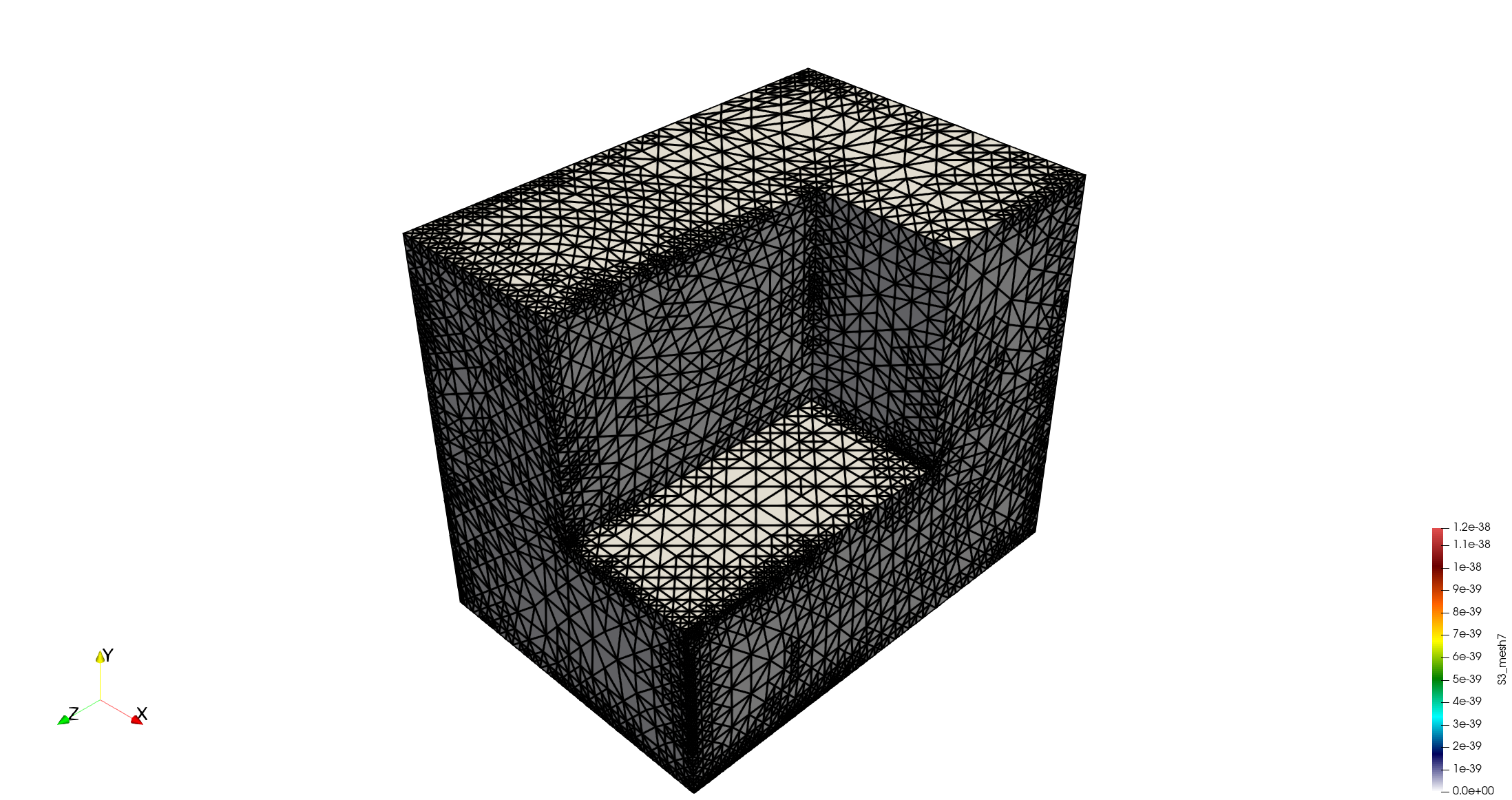}
\end{minipage}
\begin{minipage}{0.24\linewidth}\centering
	{\text{$\Omega_s$, $\omega_{h,4}$, \texttt{dof}= 213640}\\ }
	\includegraphics[scale=0.09,trim=18cm 0cm 18cm 0cm,clip]{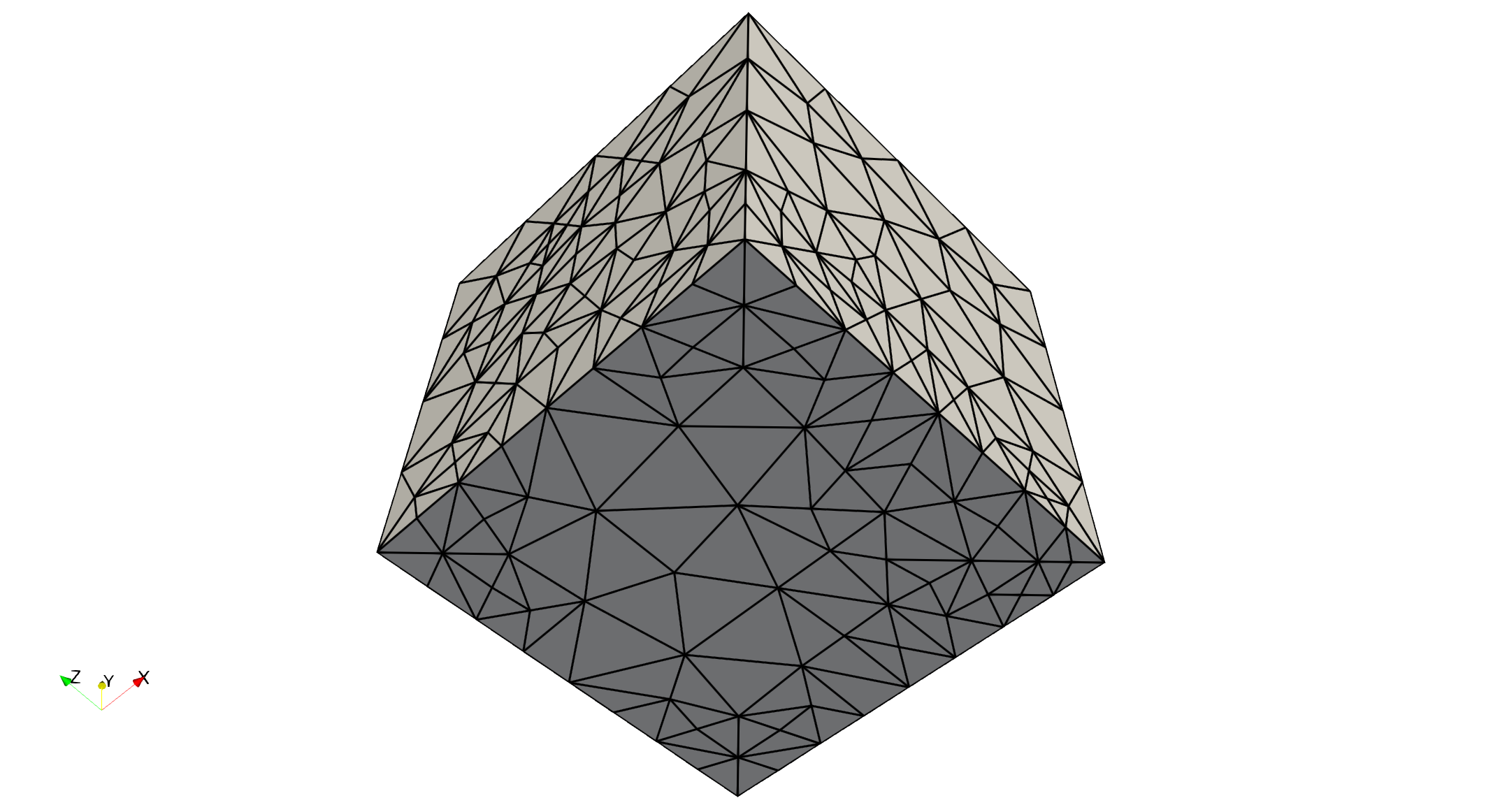}
\end{minipage}
\begin{minipage}{0.24\linewidth}\centering
	{\text{$\Omega_f$, $\omega_{h,4}$, \texttt{dof}= 213640}\\ }
	\includegraphics[scale=0.09,trim=18cm 0cm 18cm 0cm,clip]{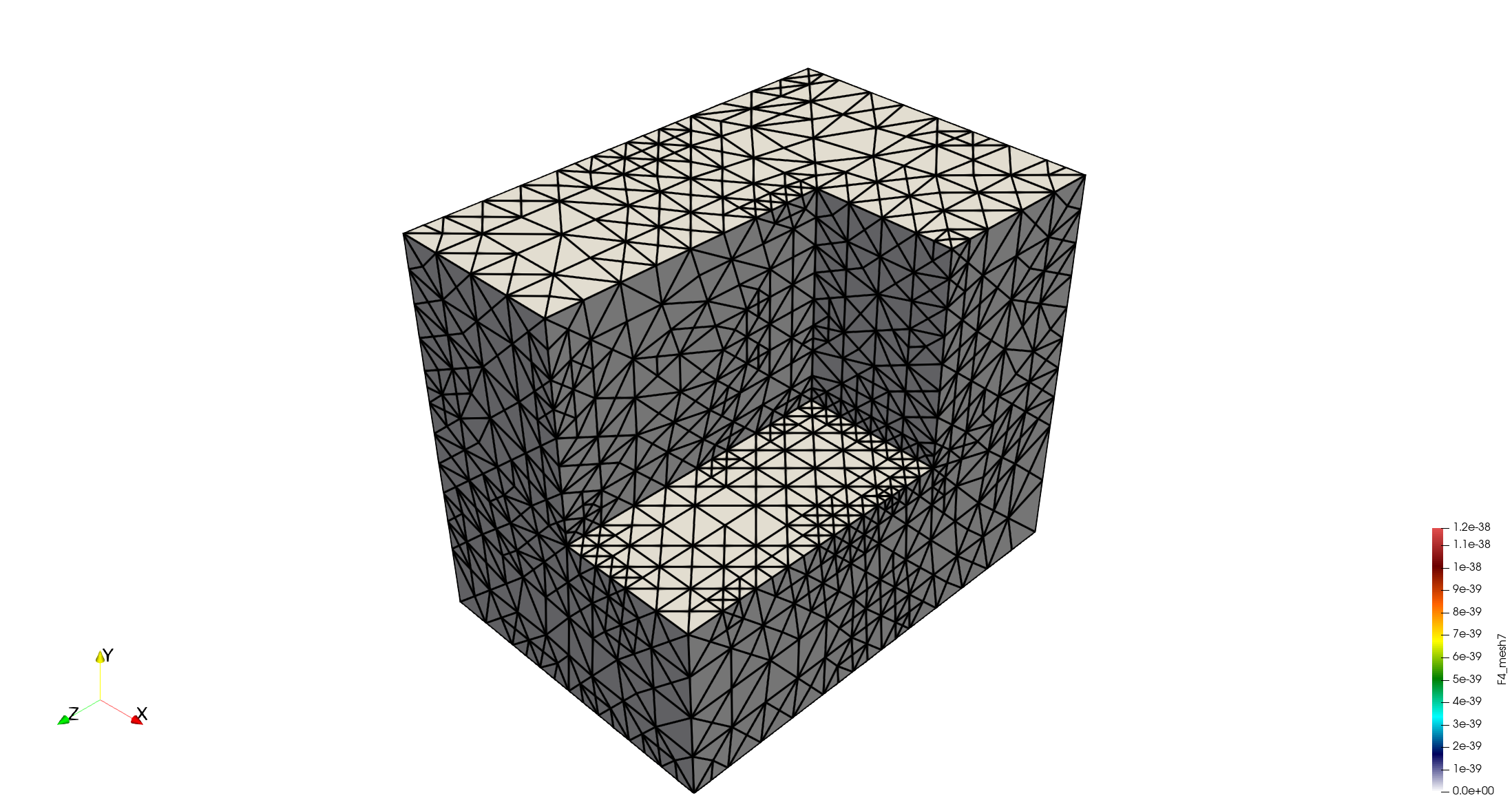}
\end{minipage}
\begin{minipage}{0.24\linewidth}\centering
	{\text{$\Omega_s$, $\omega_{h,4}$, \texttt{dof}= 1602525}\\ }
	\includegraphics[scale=0.09,trim=18cm 0cm 18cm 0cm,clip]{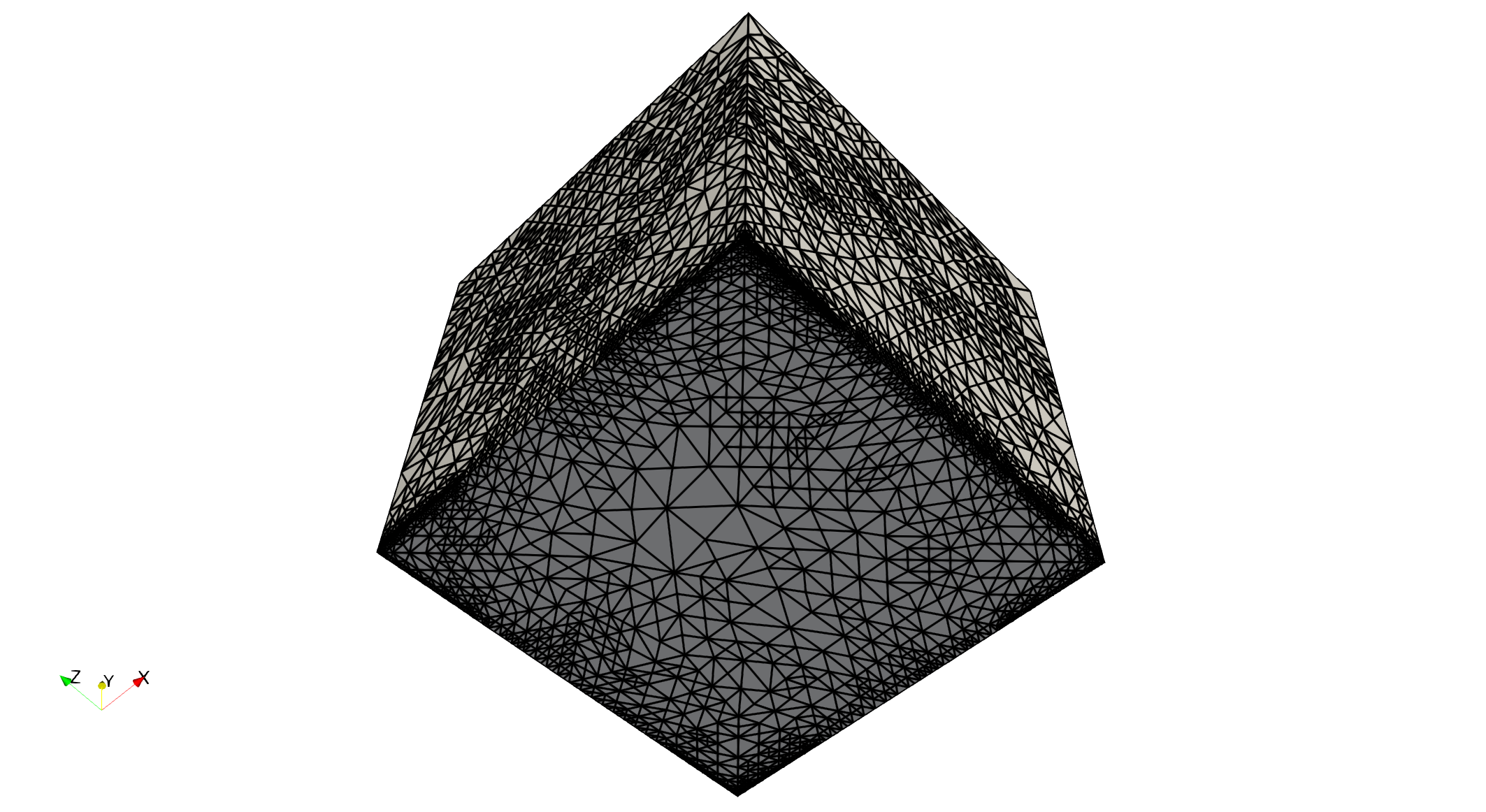}
\end{minipage}
\begin{minipage}{0.24\linewidth}\centering
	{\text{$\Omega_f$, $\omega_{h,4}$, \texttt{dof}= 1602525}\\ }
	\includegraphics[scale=0.09,trim=18cm 0cm 18cm 0cm,clip]{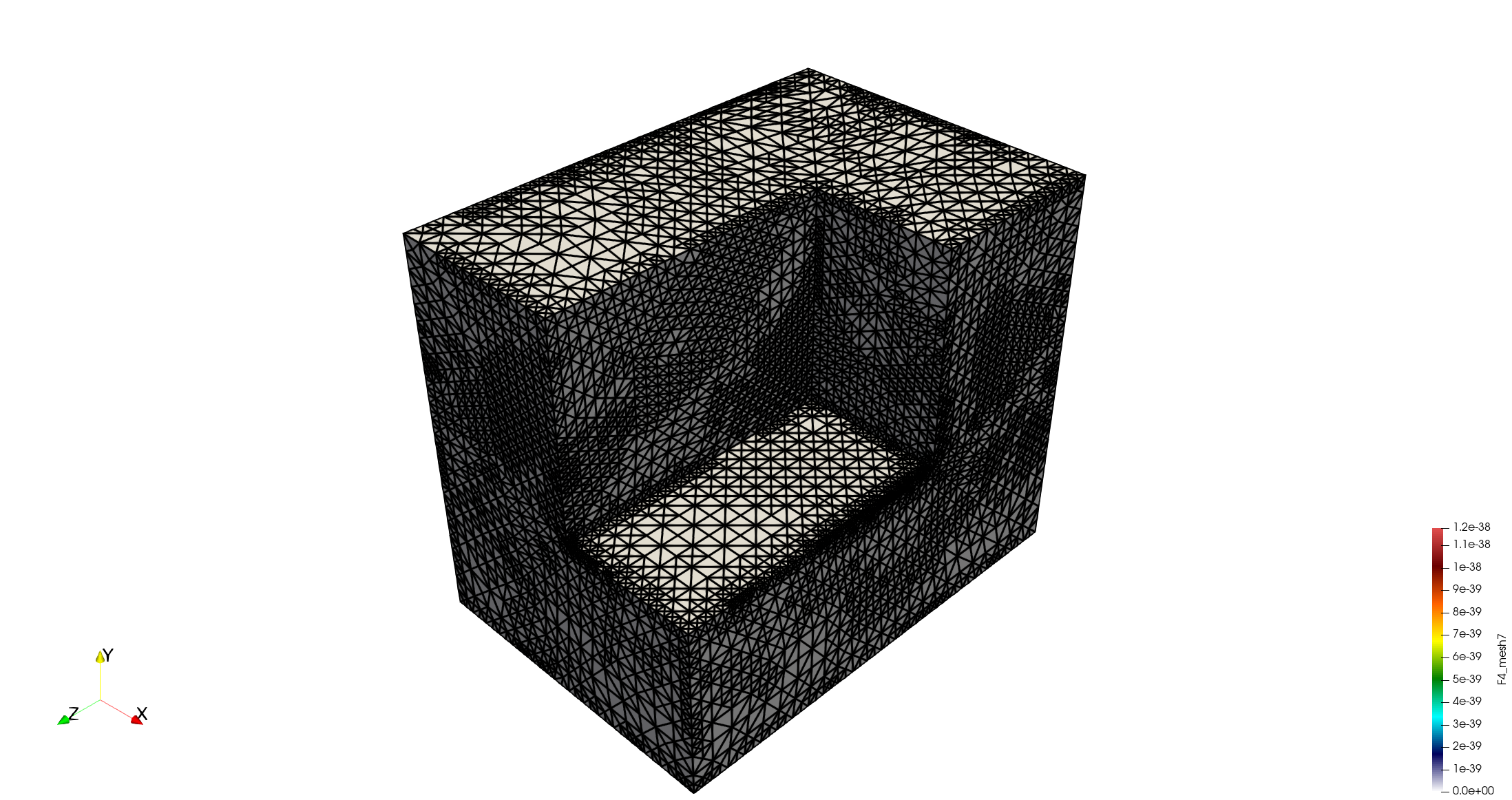}
\end{minipage}
\caption{Example \ref{subsec:3D-afem}. Comparison between intermediate meshes on the adaptive process for the first, second and fourth eigenfrequencies in $\Omega_{CF}$ with $\nu=0.35$.}
\end{figure}

\begin{figure}[!hbt]\centering
\begin{minipage}{0.32\linewidth}\centering
	\includegraphics[scale=0.26]{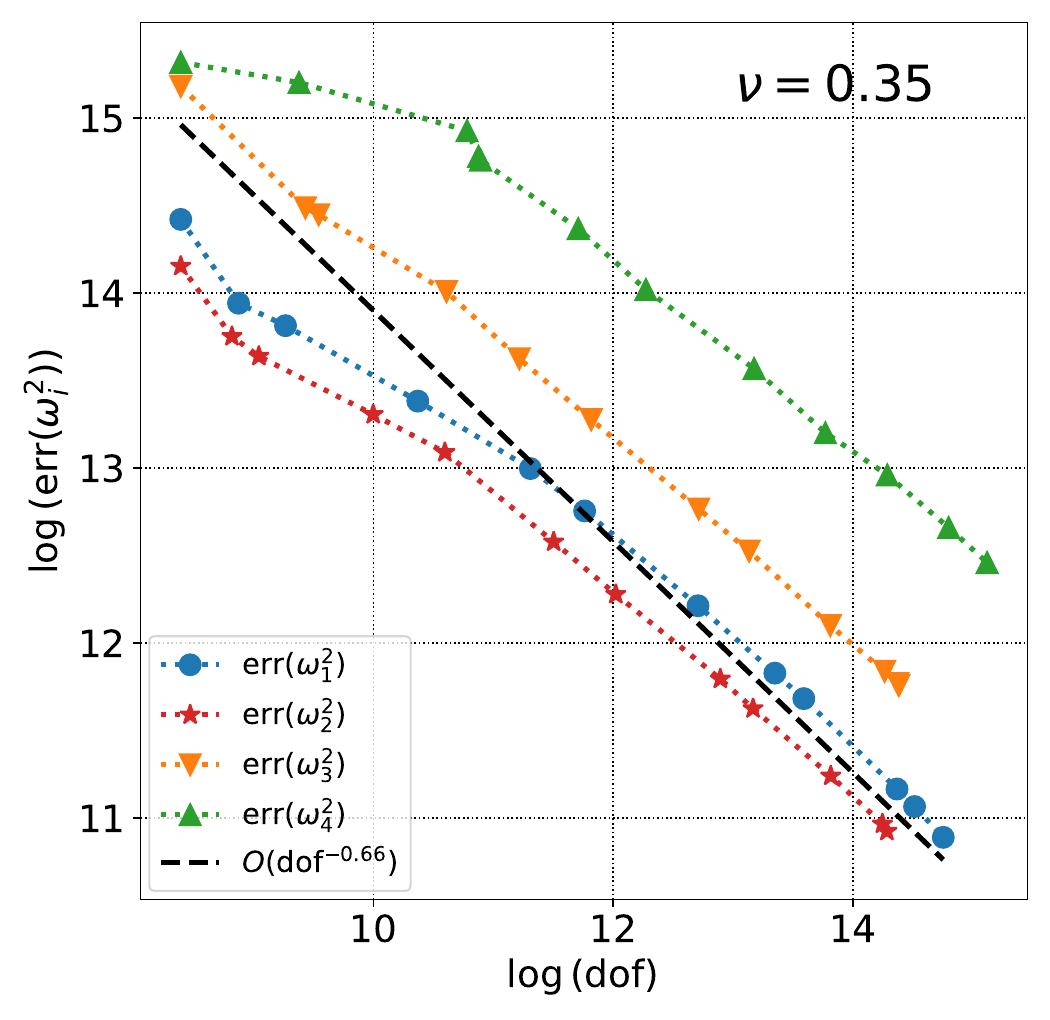}
\end{minipage}
\begin{minipage}{0.32\linewidth}\centering
	\includegraphics[scale=0.26]{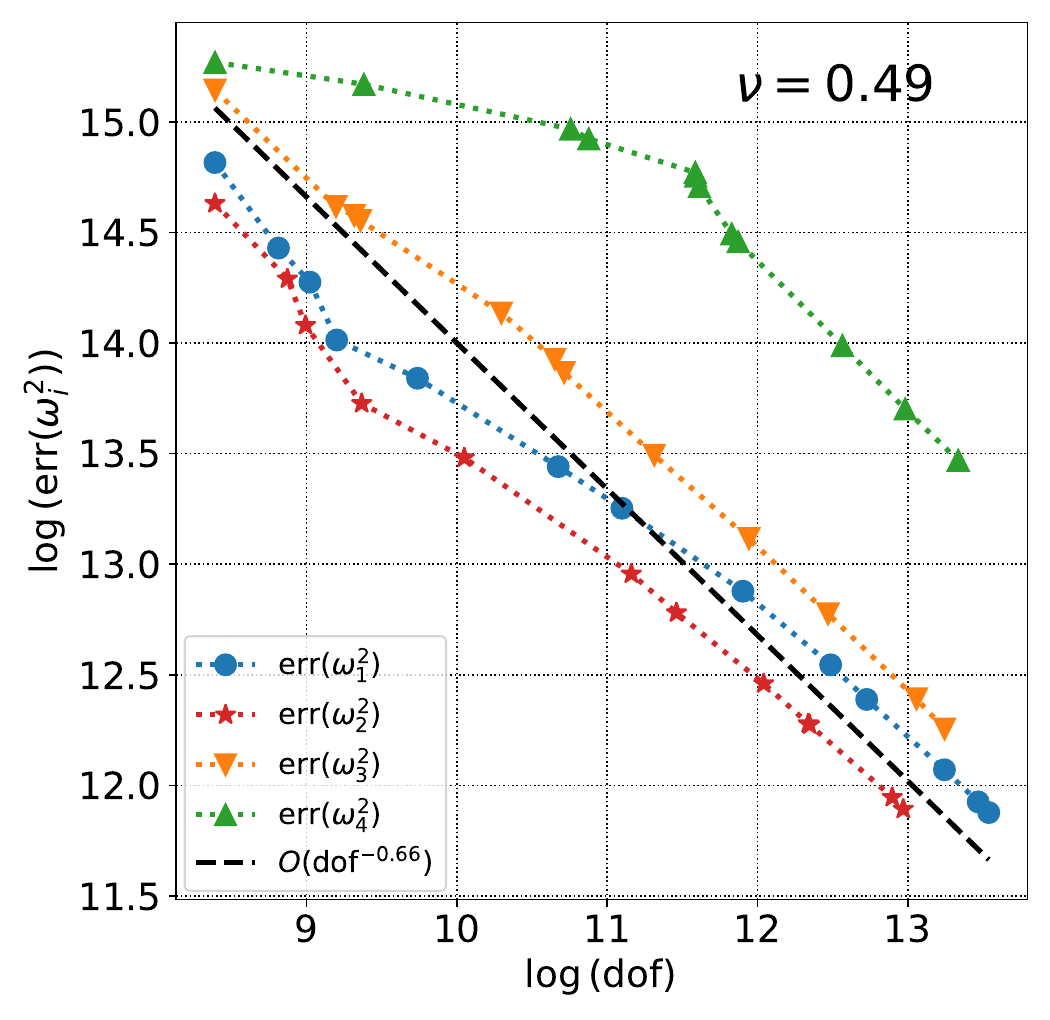}
\end{minipage}
\begin{minipage}{0.32\linewidth}\centering
	\includegraphics[scale=0.26]{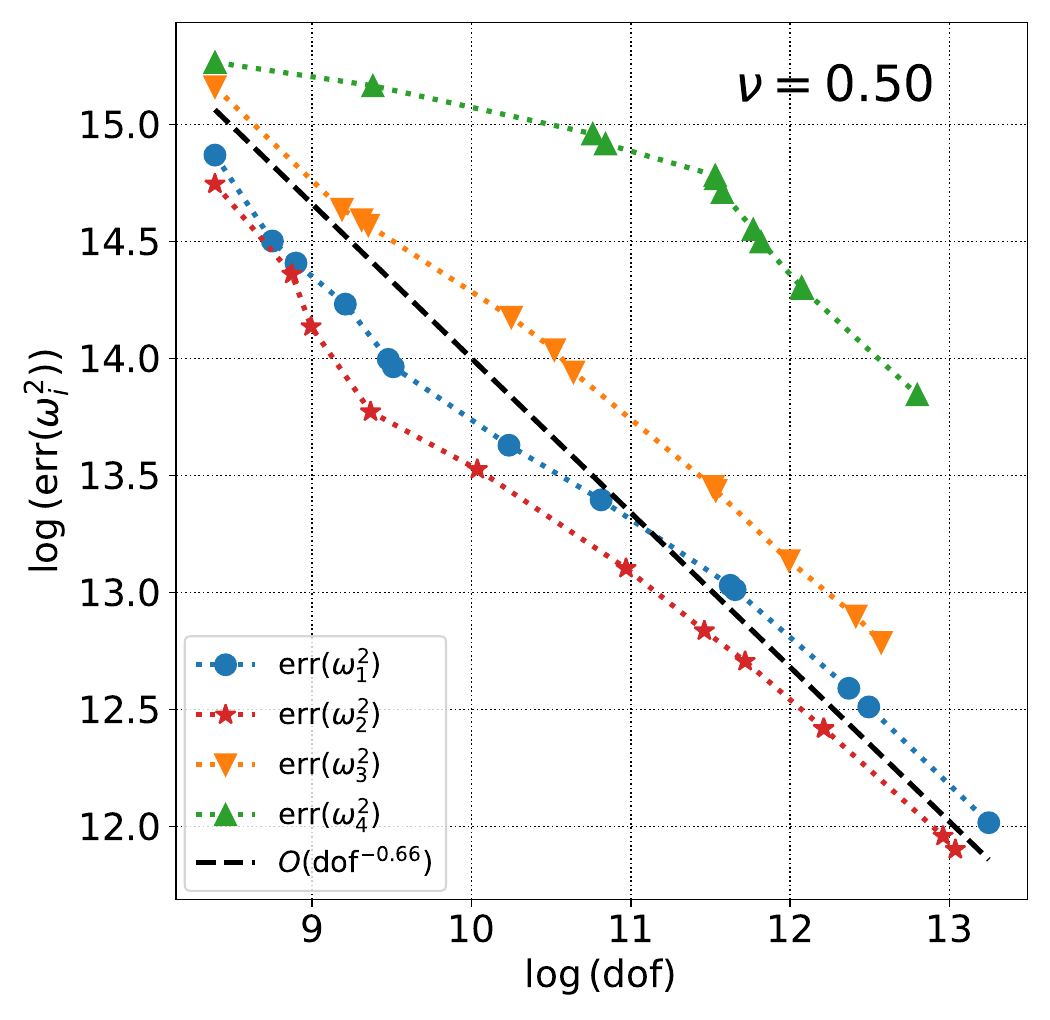}
\end{minipage}
\caption{Example \ref{subsec:afem}. Error history for the first four lowest computed frequencies in the adaptive algorithm on $\Omega_{CF}$ with different values of $\nu$.}
\label{fig:error-3D-adaptive}
\end{figure}

\begin{figure}[!hbt]\centering
\begin{minipage}{0.32\linewidth}\centering
	\includegraphics[scale=0.26]{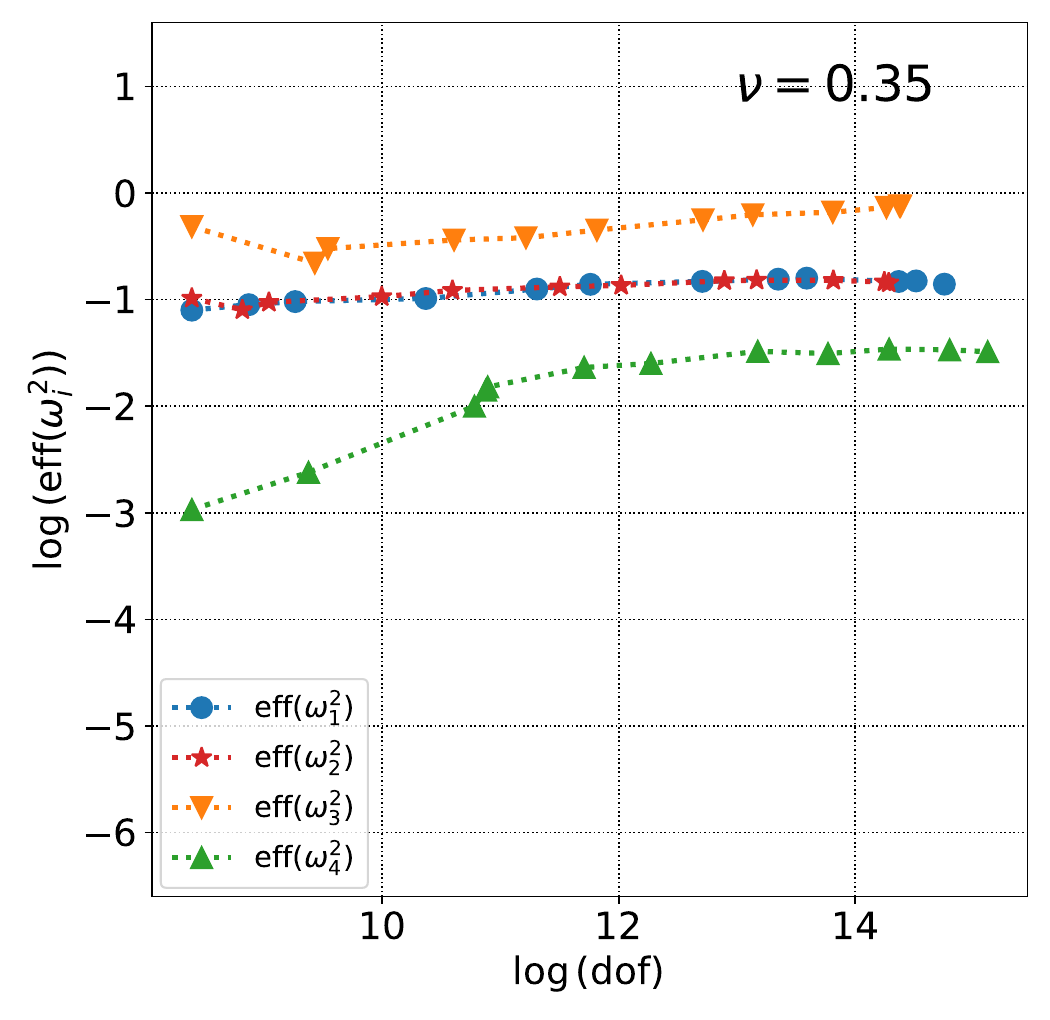}
\end{minipage}
\begin{minipage}{0.32\linewidth}\centering
	\includegraphics[scale=0.26]{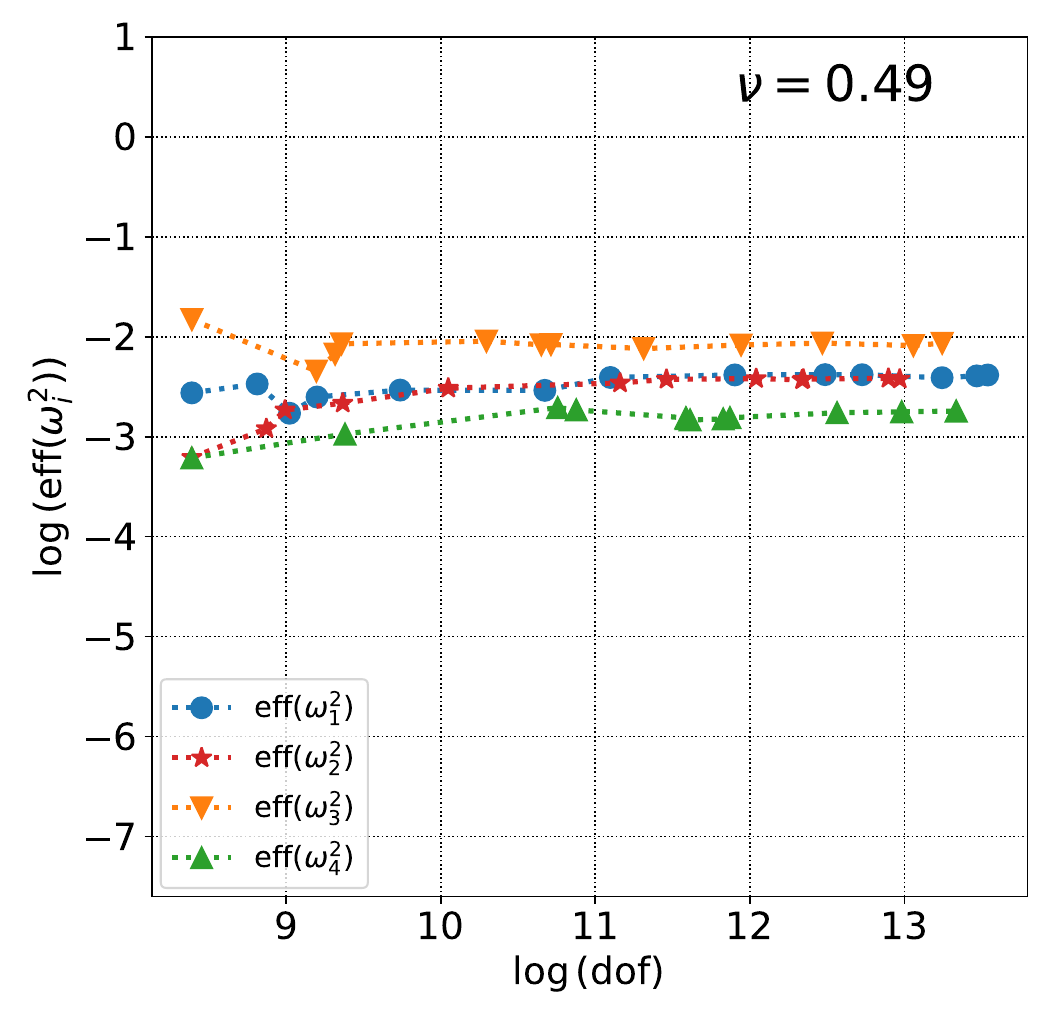}
\end{minipage}
\begin{minipage}{0.32\linewidth}\centering
	\includegraphics[scale=0.26]{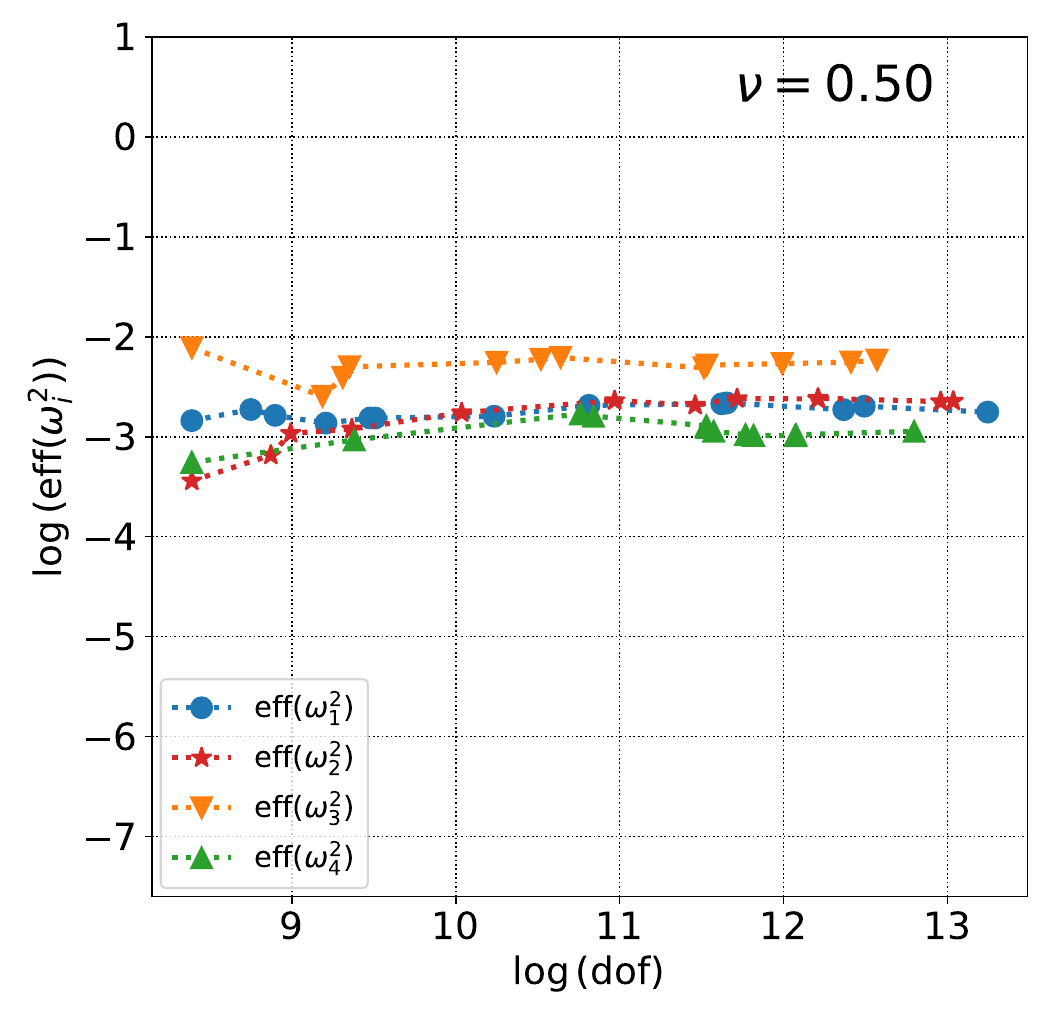}
\end{minipage}
\caption{Example \ref{subsec:afem}. Effectivity indexes for the first four lowest computed frequencies in the adaptive algorithm on $\Omega_{CF}$ with different values of $\nu$.}
\label{fig:efficiency-3D-adaptive}
\end{figure}

		\begin{acknowledgements}
			AK was  partially	supported by the Sponsored Research \& Industrial Consultancy (SRIC), Indian Institute of Technology Roorkee,
			India through the faculty initiation grant MTD/FIG/100878; by SERB MATRICS grant
			MTR/2020/000303; by SERB Core research grant CRG/2021/002569;
			FL was partially supported by was supported by Universidad del B\'io- B\'io through Proyecto Regular RE2514703.
			DM was partially supported by the National Agency for Research and Development, ANID-Chile through project Anillo of
			Computational Mathematics for Desalination Processes ACT210087, by FONDECYT project 1220881, and by project Centro de Modelamiento Matemático (CMM), FB210005, BASAL funds for centers of excellence.
			RRB acknowledges partial support from the Australian Research Council through the \textsc{Future Fellowship} grant FT220100496 and by 
 the Swedish Research Council under grant no. 2021-06594 while the author was in residence at Institut Mittag--Leffler in Djursholm, Sweden during the second semester of 2025. 
			JV was partially supported by the National Agency for Research and Development, ANID-Chile through FONDECYT Postdoctorado project 3230302.
		\end{acknowledgements}

\bibliographystyle{siam}
\bibliography{references}

\end{document}